\documentclass[10pt]{article}
\usepackage[numbers]{natbib}
\usepackage{amsmath,amssymb,amsthm,mathrsfs}
\usepackage{graphicx}
\usepackage[papersize={210mm,297mm},top=2cm, bottom=2cm, left=2cm , right=2cm]{geometry}
\usepackage[utf8]{inputenc}
\usepackage[T1]{fontenc}
\usepackage{bmpsize}
\usepackage{fancybox}
\usepackage{array}
\usepackage{geometry}
\usepackage{enumitem}
\usepackage{tikz}
\usepackage{pgfplots}
\usepackage{pst-solides3d}
\usepackage{pstricks-add}
\usepackage{graphicx}
\usepackage{dsfont}
\usepackage{tikz}
\usepackage{pgfplots}
\usepackage{amsmath}
\newcommand\subsetsim{\mathrel{\substack{
			\textstyle\subset\\[-0.2ex]\textstyle\sim}}}

\theoremstyle{remark}
\newtheorem{remark}{Remark}

\theoremstyle{definition}
\newtheorem{definition}{Definition}[section]
\theoremstyle{plain}
\newtheorem{proposition}[definition]{Proposition}

\newtheorem{theorem}[definition]{Theorem}
\newtheorem{corollary}[definition]{Corollary}
\newtheorem{lemma}[definition]{Lemma}

\begin{document}
	\renewcommand{\proofname}{Proof}
	\renewcommand{\contentsname}{Table of contents}

	\title{Stability estimates for an inverse Steklov problem in a class of hollow spheres}
	\author{Germain Gendron \\[12pt]
		\small  Laboratoire de Math\'ematiques Jean Leray, UMR CNRS 6629, \\ \small 2 Rue de la Houssini\`ere BP 92208, F-44322 Nantes Cedex 03. \\
		\small Email: germain.gendron@univ-nantes.fr }


	
	\date{\today}

	
	\maketitle


	\begin{abstract}
		
		In this paper, we study an inverse Steklov problem in a class of $n$-dimensional manifolds having the topology of a hollow sphere and equipped with a warped product metric. Precisely, we aim at studying the continuous dependence of the warping function defining the warped product with respect to the Steklov spectrum. We first show that the knowledge of the Steklov spectrum up to an exponential decreasing error is enough to determine uniquely the warping function in a neighbourhood of the boundary. Second, when the warping functions are symmetric with respect to 1/2, we prove a log-type stability estimate in the inverse Steklov problem. As a last result, we prove a log-type stability estimate for the corresponding Calder\'on problem.

		
		\vspace{0.5cm}
		

	\end{abstract}

	\newpage

	\newpage

\tableofcontents


\newpage

\section{Introduction}

\subsection{Framework}

\noindent For $n\ge 2$, let us consider a class of $n$-dimensional manifolds $M=[0,1]\times\mathbb{S}^{n-1}$ equipped with a warped product metric
\begin{equation*}
g=f(x)(dx^2+g_{\mathbb{S}})
\end{equation*} 

\medskip

\noindent where $g_\mathbb{S}$ denotes the usual metric on $\mathbb{S}^{n-1}$ induced by the euclidean metric on $\mathbb{R}^n$ and $f$ is a smooth and positive function on $[0,1]$. Let $\displaystyle\psi$ belong to $ H^{1/2}(\partial M)$ and $\omega$ be in $\mathbb{R}$. 

$\quad$

\noindent The Dirichlet problem is the following elliptic equation with boundary condition
\begin{equation}
\label{Schr_2}
\left\{
\begin{aligned}
& -\Delta_g u=\omega u \:\:{\rm{in}}\:\:M\\
& u=\psi\:\:{\rm{on}}\:\:\partial M, \end{aligned}\right.
\end{equation}
where, in a local coordinate system $(x_i)_{i=1,...,n}$, and setting $|g|=\det(g_{ij})$ and $(g^{ij})=(g_{ij})^{-1}$, the Laplacian operator $-\Delta_g$ has the expression 

\begin{equation*}
-\Delta_g=-\sum_{1\le i,j\le n}\frac{1}{\sqrt{|g|}}\partial_i\big(\sqrt{|g|}g^{ij}\partial_j\big).
\end{equation*} 

$\quad$

\noindent If $\omega$ does not belong to the Dirichlet spectrum of $-\Delta_g$, the equation (\ref{Schr_2}) has a unique solution in $H^1(M)$ (see \cite{salo2013calderon,taylor2010partial}), so we can define the so-called \textit{Dirichlet-to-Neumann} operator ("DN map") $\Lambda_g(\omega)$ as :
\begin{equation*}
\begin{array}{ccccc}
\Lambda_g(\omega) & : & H^{1/2}(\partial M) & \to & H^{-1/2}(\partial M) \\
& & \displaystyle\psi & \mapsto &  \displaystyle\frac{\partial u}{\partial\nu} \big|_{\partial M}\\
\end{array}
\end{equation*}

\noindent where $\nu$ is the unit outer normal vector on $\partial M$. The previous definition has to be understood in the weak sense by: 
\begin{equation}
\label{weak_sense_2}
\forall (\psi,\phi)\in H^{1/2}(\partial M)^2\: :\: \langle\Lambda_g(\omega)\psi,\phi\rangle=\int_{M}\langle du,dv\rangle_g,\ \mathrm{dVol}_g+\omega\int_{M}uv\,\mathrm{dVol}_g.
\end{equation}

\medskip

\noindent where $v$ is any element of $H^1(M)$ such that $v|_{\partial M}=\phi$, $\langle.,.\rangle$ is the standard $L^2$ duality pairing between $H^{1/2}(\partial M)$ and its dual, and $\mathrm{dVol}_g$ is the volume form induced by $g$ on $M$.

$\quad$ 

\noindent The DN map $\Lambda_g(\omega)$ is a self-adjoint pseudodifferential operator of order one on $L^2(\partial M)$. Then, it has a real and discrete spectrum accumulating at infinity. We shall denote the Steklov eigenvalues counted with multiplicity by
\begin{equation*}
\sigma(\Lambda_g(\omega))=\{0=\lambda_0<\lambda_1\le\lambda_2\le...\le\lambda_k\to+\infty\},
\end{equation*} 

\medskip

\noindent usually called the \textit{Steklov spectrum} (see \cite{jollivet2014inverse}, p.2 or \cite{girouard2017spectral} for details). 

$\quad$

\noindent The inverse Steklov problem adresses the question whether the knowledge of the Steklov spectrum is enough to recover the metric $g$. Precisely:
\begin{center}
	\textit{If $\:\displaystyle \sigma\big(\Lambda_g(\omega)\big)=\sigma\big(\Lambda_{\tilde{g}}(\omega)\big)$, is it true that $ g=\tilde{g}$ ?}
\end{center}

\medskip

\noindent It is known that the answer is negative because of some gauge invariances in the Steklov problem. These gauge invariances are (see \cite{jollivet2014inverse}):

\medskip

\noindent 1) Invariance under pullback of the metric by the diffeomorphisms of $M$ :
\begin{equation*}
\forall \psi \in\mathrm{Diff}(M),\quad  \Lambda_{\psi^*g}(\lambda)=\varphi^*\circ \Lambda_g(\lambda)\circ \varphi^{*-1}.
\end{equation*}
where $\varphi:=\psi|_{\partial M}$ and where  $\varphi^* : C^\infty(\partial M)\to C^\infty(\partial M)$ is the application defined by $\varphi^*h:=h\circ \varphi$.
\medskip

\noindent 2) In dimension $n=2$ and for $\omega=0$, there is one additional gauge invariance. Indeed, thanks to the conformal invariance of the Laplacian, for every smooth function $c>0$, we have
\begin{equation*}
\Delta_{cg}=\frac{1}{c}\Delta_g.
\end{equation*}

\noindent Consequently, the solutions of the Dirichlet problem (\ref{Schr_2}) associated to the metrics $g$ and $cg$ are the same when $\omega=0$. Moreover, if we assume that $c\equiv 1$ on the boundary, the unit outer normal vectors on $\partial M$ are also the same for both metrics. Therefore,
\begin{equation*}
\Lambda_{cg}(0)=\Lambda_{g}(0).
\end{equation*}


\noindent In our particular model, we have shown in \cite{gendron2019uniqueness} that the only gauge invariance is given by the involution $\eta:x\mapsto 1-x$ when $n=2$ and also when $n\ge 3$ under some technical estimates on the conformal factor $f$ on the boundary. Precisely, we proved:

\begin{theorem}\label{theorem1}
	Let $M=[0,1]\times \mathbb{S}^{n-1}$ be a smooth Riemannian manifold equipped with the metrics
	\begin{center}
		$ g=f(x)(dx^2+g_\mathbb{S}),$ \ \  $\tilde{g}=\tilde{f}(x)(dx^2+g_\mathbb{S})$,
	\end{center}
	and let $\omega$ be a frequency not belonging to the Dirichlet spectrum of  $-\Delta_g$ and $-\Delta_{\tilde{g}}$ on $M$. Then,
	\begin{enumerate}
		\item For $n=2$ and $\omega\ne 0$,
		\begin{equation*}
		\big(\sigma(\Lambda_g(\omega)) = \sigma(\Lambda_{\tilde{g}}(\omega))\big)	\Leftrightarrow \big(f=\tilde{f}\quad {\rm{or}}\quad f=\tilde{f}\circ \eta\big)
		\end{equation*}
		where $\eta(x)=1-x$ for all $x \in [0,1]$.
		
		\item For $n\ge 3$, and if moreover
		\begin{equation*}
		f,\tilde{f}\in\mathcal{C}_b := \bigg\{f\in C^\infty([0,1]),\: \bigg|\frac{f'(k)}{f(k)}\bigg|\le \frac{1}{n-2},\:k=0 \ {\rm{and}} \  1 \bigg\},
		\end{equation*}			
		\begin{equation*}
		\big(\sigma(\Lambda_g(\omega)) = \sigma(\Lambda_{\tilde{g}}(\omega))\big)	\Leftrightarrow \big(f=\tilde{f}\quad {\rm{or}}\quad f=\tilde{f}\circ \eta\big)
		\end{equation*}
		\noindent
	\end{enumerate}
\end{theorem}

\medskip

\begin{remark}
 In Theorem \ref{theorem1}, there is no need to assume that $\omega \ne 0$ when $n\ge 3$ whereas this condition is necessary in dimension $2$, due to the gauge invariance by a conformal factor.
\end{remark}

\medskip

\noindent In this paper, we will show two additional results on the Steklov inverse problem, that follow and precise the question of uniqueness. Namely, we will prove some \emph{local uniqueness} and \emph{stability} results. Before stating our results, recall that the boundary $\partial M$ of $M$ has two connected components. If we denote $-\Delta_{g_\mathbb{S}}$ the Laplace-Beltrami operator on $(\mathbb{S}^{n-1},g_\mathbb{S})$ and
\begin{equation*}
\sigma(-\Delta_{g_\mathbb{S}}):=\{0=\mu_0<\mu_1\le\mu_2\le...\le \mu_m\le...\to+\infty\}
\end{equation*}
\noindent the sequence of the eigenvalues of $-\Delta_{g_\mathbb{S}}$, counted with multiplicity, one can show that the spectrum of $\Lambda_g(\omega)$ is made of two sets of eigenvalues $\{\lambda^-(\mu_m)\}$ \ and \ $\{\lambda^+(\mu_m)\}$ whose asymptotics are given later in Lemma \ref{asymp_vp_2}.

\subsection{Closeness of two Steklov spectra}

\noindent Let us define first what is the \textit{closeness} between two spectra $\sigma\big(\Lambda_g(\omega)\big)$ and $\sigma\big(\Lambda_{\tilde{g}}(\omega)\big)$.

\begin{definition}
	\label{proximite_2}
	\noindent Let $(\varepsilon_m)_m$ a sequence of positive numbers. We say that $\sigma\big(\Lambda_g(\omega)\big)$ is close to $\sigma\big(\Lambda_{\tilde{g}}(\omega)\big)$ up to the sequence $(\varepsilon_m)_m$ if, for every $\lambda^{\pm}(\mu_m)\in\sigma\big(\Lambda_g(\omega)\big)$:
	\begin{enumerate}
	\item[$\bullet$] there is $\tilde{\lambda}^\pm$ in $\sigma\big(\Lambda_{\tilde{g}}(\omega)\big)$ such that  $|\lambda^{\pm}(\mu_m)-\tilde{\lambda}^\pm|\le \varepsilon_m$.
	\item[$\bullet$] $\mathrm{Card}\big\{\lambda^\pm\in \sigma\big(\Lambda_{g}(\omega)\big),\:\: |\lambda^\pm(\mu_m)-\lambda^\pm|\le \varepsilon_m\big\}=\mathrm{Card}\big\{\tilde{\lambda}^\pm\in \sigma\big(\Lambda_{\tilde{g}}(\omega)\big),\:\: |\lambda^\pm(\mu_m)-\tilde{\lambda}^\pm|\le \varepsilon_m\big\}$.
\end{enumerate}

	\medskip
	
	\noindent where $\mathrm{Card} \:A$ is the cardinal of the set $A$. We denote it
	\begin{center}
	$\sigma\big(\Lambda_{g}(\omega)\big)\underset{(\varepsilon_m)}{\subsetsim}\sigma\big(\Lambda_{\tilde{g}}(\omega)\big)$.
	\end{center}
\end{definition}

\begin{remark}
	The second point of Definition \ref{proximite_2} amounts to taking into account the multiplicity of the eigenvalues.
\end{remark}

\medskip

\begin{definition}
	\label{expcloseness_2}
	We say that $\sigma\big(\Lambda_g(\omega)\big)$ and $\sigma\big(\Lambda_{\tilde{g}}(\omega)\big)$ are close up to $(\varepsilon_m)_m$ if
	\begin{center}
	$\sigma\big(\Lambda_{g}(\omega)\big)\underset{(\varepsilon_m)}{\subsetsim}\sigma\big(\Lambda_{\tilde{g}}(\omega)\big)\quad$ and $\quad\sigma\big(\Lambda_{\tilde{g}}(\omega)\big)\underset{(\varepsilon_m)}{\subsetsim}\sigma\big(\Lambda_{g}(\omega)\big)$.
	\end{center} 
	
	\noindent We denote it $\sigma\big(\Lambda_g(\omega)\big)\underset{(\varepsilon_m)}{\asymp}\sigma\big(\Lambda_{\tilde{g}}(\omega)\big)$. 
	
	\medskip
	
	\noindent {\bf Constant sequence} : if $(\varepsilon_m)$ is a constant sequence such that, for all $m$, $\varepsilon_m=\varepsilon$, we just denote
	\begin{equation*}
	\sigma\big(\Lambda_g(\omega)\big)\underset{\varepsilon}{\asymp}\sigma\big(\Lambda_{\tilde{g}}(\omega)\big).  
	\end{equation*} 
		\end{definition}

	\begin{definition}
	\noindent If $A$ and $\tilde{A}$ are any finite subset of $\mathbb{R}$, we will denote $A\underset{\varepsilon}{\subsetsim} \tilde{A}$ if, for every $a\in A$
	\begin{enumerate}
		\label{closeness_2}
		\item[$\bullet$] There is $\tilde{a}\in \tilde{A}$ such that $|a-\tilde{a}|\le \varepsilon$,
		\item[$\bullet$] $\mathrm{Card}\big\{\lambda\in A,\:\: |\lambda-a|\le \varepsilon\big\}=\mathrm{Card}\big\{\tilde{\lambda}\in \tilde{A},\:\: |a-\tilde{\lambda}|\le \varepsilon\big\}$.
	\end{enumerate}
	\end{definition}

\medskip

\noindent We denote $A\underset{\varepsilon}{\asymp}\tilde{A}\:$ if $\: A\underset{\varepsilon}{\subsetsim} \tilde{A}$ and $\tilde{A}\underset{\varepsilon}{\subsetsim} A$.

$\quad$

\medskip

\noindent This work is based on ideas developped by Daudé, Kamran and Nicoleau in \cite{daude2019stability}. However, due to the specific structure of our model that possesses a \emph{disconnected} boundary (contrary to the model studied in \cite{daude2019stability}), some new difficulties arise.

$\quad$

\noindent {\bf Local uniqueness.} We would like to answer the following question : if the data of the Steklov spectrum is known up to some exponentially decreasing sequence, is it possible to recover the conformal factor $f$ in the neigbourhood of the boundary (or one of its component) up to a natural gauge invariance ? The main difficulty that appears here is due to the presence of two sets of eigenvalues, in each spectrum $\sigma\big(\Lambda_g(\omega)\big)$ and $\sigma\big(\Lambda_{\tilde{g}}(\omega)\big)$, instead of one as in \cite{daude2019stability}. With the previous definitions of closeness, it is not clear that we can get, for example, this kind of implication : 
\begin{center}
	$\big($ $\sigma\big(\Lambda_g(\omega)\big)$ and $\sigma\big(\Lambda_{\tilde{g}}(\omega)\big)$ close $\big)$ $\Rightarrow$ $\big($ $\lambda^-(\mu_m)$ and $\tilde {\lambda}^-(\mu_m)$ close for all $m\in\mathbb{N}$ $\big)$ 
	
	$\hspace{7.2cm}$ or $\big($ $\lambda^-(\mu_m)$ and $\tilde {\lambda}^+(\mu_m)$ close for all $m\in\mathbb{N}$ $\big)$.
\end{center}
In order to overcome this problem, we will need to do some hypotheses on the warping function $f$ to introduce a kind of asymmetry on the metric on each component. In that way, the previous implication will be true by replacing $\mathbb{N}$ with an infinite subset $\mathcal{L}\subset \mathbb{N}$ that satisfies, for $m$ large enough, $\mathcal{L}\cap\{m,m+1\}\ne \emptyset$. In other word, the frequency of integers belonging to $\mathcal{L}$ will be greater than $1/2$.

$\quad$

\noindent {\bf Stability.} As regards the problem of stability, if the Steklov eigenvalues are known up to a positive, fixed and small $\varepsilon$, is it possible to find an approximation of the conformal factor $f$ depending on $\varepsilon$ ? Thanks to Theorem \ref{theorem1}, we know that there is no uniqueness in the problem of recovering $f$ from $\sigma\big(\Lambda_g(\omega)\big)$. This seems to be a serious obstruction to establish any result of stability in a general framework. Indeed, the uniqueness problem solved in Theorem \ref{theorem1}  is quite rigid (as well as the local uniqueness result) and is based on analyticity results that can no longer be used here. On the contrary, the condition  
	\begin{equation*}
\sigma\big(\Lambda_g(\omega)\big)\underset{\varepsilon}{\asymp}\sigma\big(\Lambda_{\tilde{g}}(\omega)\big)
\end{equation*} 
seems much less restrictive than an equality, and the non-uniqueness makes this new problem quite difficult to tackle. From Theorem \ref{theorem1}, we see that the only way to get a strict uniqueness result is to assume that $f$ is symmetric with respect to $1/2$. This natural - albeit restrictive - condition will be made on $f$ in Section 4 devoted to the stability result.

\subsection{The main results}	

\medskip

\begin{definition}
The class of functions $\mathcal{D}_b$ is defined by 
\begin{center}
	$\mathcal{D}_b=\{h\in C^\infty([0,1])\:\:\big|\:\:\:\exists k\in\mathbb{N},\:\: h^{(k)}(0)\ne (-1)^kh^{(k)}(1)\}$.
\end{center}
\end{definition}

\medskip

\begin{definition}
	The potential associated to the conformal factor $f$ is the function $q_f$ defined on $[0,1]$ by $\displaystyle q_f=\frac{(f^{\frac{n-2}{4}})''}{f^{\frac{n-2}{4}}}-\omega f$.
\end{definition}

\noindent The potential $q_f$ naturally appears when we solve the problem (\ref{Schr_2}) by separating the variables in order to get an infinite system of ODE. We have at last to precise the following notation that will appear in the statement of Theorem \ref{pseudostab_2}.

$\quad$

\noindent {\bf Notation:} Let $x_0$ be in $\mathbb{R}$ and $g$ be a real function such that $\lim\limits_{x\to x_0}g(x)=0$. 

\medskip

\noindent We say that $f(x)\underset{x_0}{=}\tilde{O}\big(g(x)\big)$ if
\begin{center}
$\displaystyle \forall \varepsilon>0\,\:\: \lim\limits_{x\to x_0}\frac{|f(x)|}{|g(x)|^{1-\varepsilon}}=0$.
\end{center}
	
\noindent Here is our local uniqueness result.

\begin{theorem}
	\label{pseudostab_2}
	Let $M=[0,1]\times \mathbb{S}^{n-1}$ be a smooth Riemannian manifold equipped with the metrics
	\begin{center}
		$ g=f(x)(dx^2+g_\mathbb{S}),$ \ \  $\tilde{g}=\tilde{f}(x)(dx^2+g_\mathbb{S})$,
	\end{center}
	and let $\omega$ be a frequency not belonging to the Dirichlet spectrum of  $-\Delta_g$ or $-\Delta_{\tilde{g}}$ on $M$. Let $a\in ]0,1[$ and $\mathcal{E}$ be the set of all the positive sequences $(\varepsilon_m)_m$ satisfying 
	\begin{equation*}
	\varepsilon_m= \tilde{O}\big(e^{-2a\sqrt{\mu_m}}\big),
	\end{equation*}
	
	$\quad$
	
	\noindent In order to simplify the statements of the results, let us denote the propositions:
	
	\begin{enumerate}
		\item[$*$] $(P_1)$ : $f=\tilde{f}\:\:\mathrm{on}\:\: [0,a]$ 
		\item[$*$] $(P_2)$ : $f=\tilde{f}\circ \eta  \:\:\mathrm{on}\:\: [0,a]$
		\item[$*$] $(P_3)$: $f=\tilde{f}\:\:\mathrm{on}\:\: [1-a,1]$
		\item[$*$] $(P_4)$ : $f=\tilde{f}\circ \eta  \:\:\mathrm{on}\:\: [1-a,1]$
	\end{enumerate}
	\noindent where $\eta(x)=1-x$ for all $x \in [0,1]$.
	
	$\quad$
	
	\noindent Assume that $f$ and $\tilde{f}$ belong to $\displaystyle C^\infty([0,1])\cap\mathcal{C}_b$ where $\mathcal{C}_b$ is defined as 
	\begin{center}
	$\displaystyle \mathcal{C}_b=\bigg\{\bullet\:\: \bigg|\frac{f'(k)}{f(k)}\bigg|\le \frac{1}{n-2},\:k\in\{0,1\},\qquad \bullet\:\: q_f\in\mathcal{D}_b\bigg\}$.
	\end{center} Then :
	
	$\quad$
	
	\noindent $\bullet$ For $n=2$ and $\omega\ne 0$ or $n\ge 3$:
	\begin{equation*}
	\bigg(\exists \:(\varepsilon_m)\in\mathcal{E},\:\:\sigma(\Lambda_g(\omega)) \underset{(\varepsilon_m)}{\asymp} \sigma(\Lambda_{\tilde{g}}(\omega))\bigg)	\Rightarrow (P_1)\:\:\mathrm{or}\:\: (P_2)\:\:\mathrm{or}\:\: (P_3)\:\:\mathrm{or}\:\: (P_4),
	\end{equation*}
	
	$\quad$
	
\begin{remark}
	When $n=2$, the condition $\displaystyle \bigg|\frac{f'(k)}{f(k)}\bigg|\le \frac{1}{n-2}$ is always satisfied.
\end{remark}
	
	\begin{remark}
		The converse is not true if $f(0)\ne f(1)$. If one of the $(P_i)$ is satisfied, it cannot imply more than the closeness of one of the subsequence $\big(\lambda^-(\mu_m)\big)$ or $\big(\lambda^+(\mu_m)\big)$ with $\big(\tilde{\lambda}^-(\mu_m)\big)$ or $\big(\tilde{\lambda}^+(\mu_m)\big)$.
	\end{remark}

\medskip

	\noindent {\bf Special case} : when $f(0)=f(1)$, we have the following equivalence :
	\begin{equation*}
	\begin{aligned}
	\bigg(\exists \:(\varepsilon_m)\in\mathcal{E},\:\:\sigma(\Lambda_g(\omega)) \underset{(\varepsilon_m)}{\asymp} \sigma(\Lambda_{\tilde{g}}(\omega))\bigg)	\Leftrightarrow \bigg((P_1)\:\:\mathrm{and}\:\: (P_3)&\bigg)\:\:\mathrm{or}\:\: \bigg((P_2)\:\:\mathrm{and}\:\: (P_4)\bigg).
	\end{aligned}
	\end{equation*}
\end{theorem}

$\quad$

$\quad$

\noindent Let us also give our stability result. It requires to assume that, for some $A>0$, the unknown conformal factor belongs to the set of $A$-admissible functions that we define now.

\medskip 

\begin{definition}
	Let $A>0$. The set of $A$-admissible functions is defined as :
	\begin{center}
		$\displaystyle \mathcal{C}(A)=\bigg\{\bullet \:f\in C^2([0,1])\quad\:\:\quad\bullet\:\forall k\in [\![0,2]\!],\:\: \|f^{(k)}\|_\infty +\bigg\|\frac{1}{f}\bigg\|_\infty\le A\bigg\}$
	\end{center}
\end{definition}

\noindent

\noindent 

$\quad$

\noindent Our stability result for the Steklov problem is the following:

\medskip

\begin{theorem}
	\label{stabresult_2}
Let $M=[0,1]\times \mathbb{S}^{n-1}$ be a smooth Riemanniann manifold equipped with the metrics
\begin{center}
	$ g=f(x)(dx^2+g_\mathbb{S})$, \ \  $\tilde{g}=\tilde{f}(x)(dx^2+g_\mathbb{S})$,
\end{center}
with $f$, $\tilde{f}$ positive on $[0,1]$ and symmetric with respect to $x=1/2$.
	
	$\quad$
	
	\noindent Let $A>0$ be fixed and $\omega$ be a frequency not belonging to the Dirichlet spectrum of $-\Delta_g$ and $-\Delta_{\tilde{g}}$ on $M$. Then, for $n\ge2$, for a sufficiently small $\varepsilon>0$ and under the assumption 
	\begin{equation*}
	f,\tilde{f}\in\mathcal{C}(A)
	\end{equation*}			 we have the implication :
	\begin{equation*}
	\sigma(\Lambda_g(\omega)) \underset{\varepsilon}{\asymp} \sigma(\Lambda_{\tilde{g}}(\omega)\big)	\Rightarrow \big\|q_f-\tilde{q}_f\big\|_2\le C_A\:\frac{1}{\ln\big(\frac{1}{\varepsilon}\big)}
	\end{equation*}
	
	$\quad$
	
	\noindent where $C_A$ is a constant that only depends on $A$.
\end{theorem}

\medskip

\noindent As a by-product, we get two corollaries :

\medskip

\begin{corollary}
	\label{cor_difpot_2}
	Using the same notations and assumptions as in Theorem \ref{stabresult_2}, for all $0\le s\le 2$, we have
	\begin{equation*}
	\sigma(\Lambda_g(\omega)) \underset{\varepsilon}{\asymp} \sigma(\Lambda_{\tilde{g}}(\omega)\big)	\Rightarrow \big\|q_f-\tilde{q}_f\big\|_{H^s(0,1)}\le C_A\:\frac{1}{\ln\big(\frac{1}{\varepsilon}\big)^\theta}
	\end{equation*}
	\noindent where $\theta = (2-s)/2$ and $C_A$ is a constant that only depends on $A$. In particular, from the Sobolev embedding, one gets
	\begin{equation*}
	\sigma(\Lambda_g(\omega)) \underset{\varepsilon}{\asymp} \sigma(\Lambda_{\tilde{g}}(\omega)\big)	\Rightarrow \big\|q_f-\tilde{q}_f\big\|_{\infty}\le C_A\:\sqrt{\frac{1}{\ln\big(\frac{1}{\varepsilon}\big)}}
	\end{equation*}
\end{corollary}

\medskip

\begin{corollary}
	\label{cor_steklov_stab_2}
	Using the same notations and assumptions as in Theorem \ref{stabresult_2}, if moreover $\omega=0$ and $n\ge 3$, one has
	\begin{equation*}
\sigma(\Lambda_g(\omega)) \underset{\varepsilon}{\asymp} \sigma(\Lambda_{\tilde{g}}(\omega)\big)	\Rightarrow \big\|f-\tilde{f}\big\|_\infty\le C_A\:\frac{1}{\ln\big(\frac{1}{\varepsilon}\big)}
\end{equation*}
\noindent where $C_A$ is a constant that only depends on $A$.
\end{corollary}

$\quad$

\noindent The stability in the inverse Calder\'on problem somehow precedes the inverse Steklov problem, so we say few words about it. Let $\mathcal{B}(H^{1/2}(\partial M))$ be the set of bounded operators from $H^{1/2}(\partial M)$ to $H^{1/2}(\partial M)$ equipped with the norm
\begin{equation*}
\|F\|_*=\sup_{\psi\in H^{1/2}(\partial M)\backslash\{0\}}\frac{\|F\psi\|_{H^{1/2}}}{\|\psi\|_{H^{1/2}}}.
\end{equation*}

	\noindent In Lemma \ref{bounded_equiv_22} (Section $5$) we show the equivalence
	\begin{equation*}
	\Lambda_g(\omega)-\Lambda_{\tilde{g}}(\omega)\in \mathcal{B}(H^{1/2}(\partial M))\Leftrightarrow \left\{\begin{aligned}
	&f(0)=\tilde{f}(0)\\
	&f(1)=\tilde{f}(1)
	\end{aligned}\right.
	\end{equation*}

\noindent and prove the following stability result for the Calder\'on problem. We draw the reader's attention to the fact that the symmetry hypothesis no longer occurs here since the strict uniqueness is true (see \cite{daude2019non}). However, it is replaced by a technical assumption on the mean of the difference of the potentials.

\begin{theorem}
	\label{Calderon_stab_2}
	Let $M=[0,1]\times \mathbb{S}^{n-1}$ be a smooth Riemanniann manifold equipped with the metrics
	\begin{center}
		$ g=f(x)(dx^2+g_\mathbb{S})$, \ \ $\tilde{g}=\tilde{f}(x)(dx^2+g_\mathbb{S})$,
	\end{center}
	with $f$ and $\tilde{f}$ positive on $[0,1]$.
	
	$\quad$
	
	\noindent Let $A>0$ be fixed and $\omega$ be a frequency not belonging to the Dirichlet spectrum of $-\Delta_g$ and $-\Delta_{\tilde{g}}$ on $M$. Let $n\ge2$, $\varepsilon>0$ and assume that
	\begin{enumerate}
		\item [$\bullet$] $f(0)=\tilde{f}(0)\:$ and  $\:f(1)=\tilde{f}(1),$
		\item [$\bullet$] $f,\tilde{f}\in\mathcal{C}(A)$,
		\item [$\bullet$] $\displaystyle \bigg|\int_{0}^{1}\big(q_f-\tilde{q}_f\big)\bigg|+\|\Lambda_g(\omega)-\Lambda_{\tilde{g}}(\omega)\|_*\le \varepsilon$.
	\end{enumerate}  Then :		
	\begin{equation*}
	\big\|q_f-\tilde{q}_f\big\|_2\le C_A\:\frac{1}{\ln\big(\frac{1}{\varepsilon}\big)},
	\end{equation*}
	
	$\quad$
	
	\noindent where $C_A$ is a constant that only depends on $A$.
	
	$\quad$
	
\end{theorem}

\begin{corollary}
	\label{cor_difpot_2}
	Using the same notations and assumptions as in Theorem \ref{Calderon_stab_2}, for all $0\le s\le 2$, we obtain also
	\begin{equation*}
 \big\|q_f-\tilde{q}_f\big\|_{H^s(0,1)}\le C_A\:\frac{1}{\ln\big(\frac{1}{\varepsilon}\big)^\theta}
	\end{equation*}
	\noindent where $\theta = (2-s)/2$ and $C_A$ is a constant that only depends on $A$. In particular, from the Sobolev embedding, one gets
	\begin{equation*}
 \big\|q_f-\tilde{q}_f\big\|_{\infty}\le C_A\:\sqrt{\frac{1}{\ln\big(\frac{1}{\varepsilon}\big)}}
	\end{equation*}
\end{corollary}

\medskip

\begin{corollary}
	Using the same notations and assumptions as in Theorem \ref{Calderon_stab_2}, if moreover $\omega=0$ and $n\ge 3$, one has the stronger conclusion: 
	\begin{equation*}
\big\|f-\tilde{f}\big\|_\infty\le C_A\:\frac{1}{\ln\big(\frac{1}{\varepsilon}\big)}
	\end{equation*}
	\noindent where $C_A$ is a constant that only depends on $A$.
\end{corollary}

$\quad$

\section{Asymptotics of the Steklov spectrum}

\medskip

\noindent The proof of both theorems is based on the separation of variables that leads to reformulating the Dirichlet problem in terms of boundary value problems for ordinary differential equations. All the details can be found in \cite{gendron2019uniqueness,daude2019non} but we outline the main points for the sake of completeness. 

\medskip

\subsection{From PDE to ODE using separation of variables}

\medskip

\noindent The equation
\begin{equation}
\left\{
\begin{aligned}
& -\Delta_g u=\omega u \:\:{\rm{in}}\:\:M\\
& u=\psi\:\:{\rm{on}}\:\:\partial M \end{aligned}\right.
\end{equation}

\noindent can be reduced to a countable system of Sturm Liouville boundary value problems on $[0,1]$. Indeed, the boundary $\partial M$ of the manifold $M$ has two distinct connected components
\begin{center}
	$\Gamma_0=\{0\}\times \mathbb{S}^{n-1}$ and $\Gamma_1=\{1\}\times \mathbb{S}^{n-1}$,
\end{center}
so we can decompose $H^{1}(\partial M)$ as the direct sum :
\begin{center}
	$H^{1/2}(\partial M)=H^{1/2}(\Gamma_0) \bigoplus H^{1/2}(\Gamma_1)$.
\end{center}

\medskip

\noindent
Each element $\psi$ of $H^{1/2}(\partial M)$ can be written as
\begin{center}
	$\displaystyle \psi=\begin{pmatrix}
	\psi^0\\
	\psi^1
	\end{pmatrix},\quad\quad$ $\psi^0\in H^{1/2}(\Gamma_0)\:$ and $\:\psi^1\in H^{1/2}(\Gamma_1)$.
\end{center}

\noindent Using separation of variables, one can write the solution of (\ref{Schr_2}) as
\begin{equation*}
\begin{aligned}
&u(x,y)=\sum_{m=0}^{+\infty}u_m(x)Y_m(y),\\
\end{aligned}
\end{equation*}
\noindent and $\psi^0$ and $\psi^1$ as

\begin{equation*}
\psi^0=\sum_{m\in\mathbb{N}}\psi_m^0 Y_m,\quad\quad \psi^1=\sum_{m\in\mathbb{N}}\psi_m^1 Y_m,
\end{equation*}

\noindent where $(Y_m)$ represents an orthonormal basis of eigenfunctions of $-\Delta_\mathbb{S}$ associated to the sequence of its eigenvalues counted with multiplicity
\begin{equation*}
\sigma\big(\Delta_{\mathbb{S}}\big)=\{0=\mu_0\le \mu_1...\le \mu_m \to+\infty\}.
\end{equation*}

\noindent By setting \begin{center}
	$\displaystyle v(x,y)=f^{\frac{n-2}{4}}u(x,y)=\sum_{m=0}^{+\infty}v_m(x)Y_m(y)$, $\quad x\in[0,1]$, $\:\:y\in\mathbb{S}^{n-1}$,
\end{center} 

\noindent one can show the equivalence (cf \cite{gendron2019uniqueness}) :
\begin{equation}
\label{ODE_2}
u\:\:\mathrm{solves}\:\:(\ref{Schr_2})\Leftrightarrow \forall m\in\mathbb{N},\:\:\left\{
\begin{aligned}
& -v^{''}_m(x)+q_f(x)v_m(x)=-\mu_mv_m(x),\quad \forall x\in]0,1[\\
& v_m(0)=f^{\frac{n-2}{4}}(0)\psi_m^0,\:\:v_m(1)=f^{\frac{n-2}{4}}(1)\psi_m^1, \end{aligned}\right.
\end{equation}
\noindent with $\displaystyle\:\:q_f=\frac{(f^{\frac{n-2}{4}})''}{f^{\frac{n-2}{4}}}-\omega f$ (the dependence in $f$ will be omitted in the following and we will just write $q$ instead of $q_f$).

$\quad$

\noindent We thus are brought back to a countable system of 1D Schrödinger equations whose potential does not depend on $m\in\mathbb{N}$. Thanks to the Weyl-Titchmarsh theory, we are able to give a nice representation of the DN map that involves the Weyl-Titchmarsh functions of (\ref{ODE_2}) evaluated at the sequence $(\mu_m)$.

\subsection{Diagonalization of the DN map}

\medskip

\noindent From the equation on $[0,1]$
\begin{equation}
\label{eqdif_2}
-u''+qu=-zu,\quad z\in\mathbb{C}.
\end{equation}

\noindent one can define two fundamental systems of solutions of (\ref{eqdif_2}), $\{c_0,s_0\}$ and $\{c_1,s_1\}$, whose initial Cauchy conditions satisfy
\begin{equation}
\label{solfond}
\left\{
\begin{aligned}
& c_0(0,z)=1,\:c_0'(0,z)=0\\
& c_1(1,z)=1,\:c_1'(1,z)=0 \end{aligned}\right.\quad{\rm{and}}\quad \left\{
\begin{aligned}
& s_0(0,z)=0,\:s_0'(0,z)=1\\
& s_1(1,z)=0,\:s_1'(1,z)=1. \end{aligned}\right.
\end{equation}

\noindent We shall add the subscript $\quad\tilde{}\quad$ to all the quantities referring to $\tilde{q}$. 

\medskip

\noindent The characteristic function $\Delta(z)$ associated to the equation (\ref{eqdif_2}) is defined by the Wronskian
\begin{equation}
\label{charac_2}
\Delta(z)=W(s_0,s_1):=s_0s_1'-s_0's_1.
\end{equation} 

\noindent Furthermore, there are two (uniqueLy defined) Weyl-solutions $\psi$ and $\phi$ of (\ref{eqdif_2}) having the form :
\begin{equation*}
\psi(x)=c_0(x)+M(z)s_0(x),\quad \phi(x)=c_1(x)-N(z)s_1(x)
\end{equation*}
with Dirichlet boundary conditions at $x=1$ and $x=0$ respectively. The meromorphic functions $M$ and $N$ are called the \textit{Weyl-Titchmarsh functions}. Denoting 
\begin{equation*}
D(z):=W(c_0,s_1),\quad\quad E(z):=W(c_1,s_0)
\end{equation*}
\noindent an easy calculation leads to
\begin{equation*}
M(z)=-\frac{D(z)}{\Delta(z)},\quad\quad N(z)=-\frac{E(z)}{\Delta(z)}.
\end{equation*}

\begin{remark}
	\label{WT_symmetry}
	The function $N$ has the same role as $M$ for the potential $q(1-x)$, \emph{i.e} :
	\begin{equation*}
	N(z,q)=M(z,q\circ \eta)
	\end{equation*}
	where, for all $x\in [0,1]$, $\eta(x)=1-x$.
\end{remark}

\medskip

\noindent Those meromorphic functions naturally appear in the expression of the DN map $\Lambda_g(\omega)$ in a specific basis of $H^{1/2}(\Gamma_0)\oplus H^{1/2}(\Gamma_1)$. More precisely, in the basis $\mathscr{B}=\big(\{e_m^1,e_m^2\}\big)_{m\ge 0}$ where $e_m^1$ and $e_m^2$ are defined as :
\begin{equation*}
e_m^1=(Y_m,0)\quad\quad e_m^2=(0,Y_m)
\end{equation*}
\noindent one can prove the that the operator $\Lambda_g(\omega)$ can be bloc-diagonalized :
\begin{equation*}
[\Lambda_g]_{\mathscr{B}}=\begin{pmatrix}
&\Lambda_g^1(\omega)&&0&&0&\cdots&\\
&0&&\Lambda_g^2(\omega)&&0&\cdots&\\
&0&&0&&\Lambda_g^3(\omega)&\cdots&\\
&\vdots&&\vdots&&\vdots&\ddots
\end{pmatrix},
\end{equation*}

\noindent with, for every $m\in\mathbb{N}$ and setting $h=f^{n-2}$ (cf \cite{daude2019non}): \begin{equation*}
\Lambda_g^m(\omega)=\begin{pmatrix}
-\frac{M(\mu_m)}{\sqrt{f(0)}}+\frac{1}{4\sqrt{f(0)}}\frac{h'(0)}{h(0)}&-\frac{1}{\sqrt{f(0)}}\frac{h^{1/4}(1)}{h^{1/4}(0)}\frac{1}{\Delta(\mu_m)}\\
-\frac{1}{\sqrt{f(1)}}\frac{h^{1/4}(0)}{h^{1/4}(1)}\frac{1}{\Delta(\mu_m)}&-\frac{N(\mu_m)}{\sqrt{f(1)}}-\frac{1}{4\sqrt{f(1)}}\frac{h'(1)}{h(1)}
\end{pmatrix}.
\end{equation*}

\subsection{Asymptotics of the eigenvalues}

\noindent It is then possible, with this representation of $\Lambda_g(\omega)$, to get the following precise asymptotics of the eigenvalues $\lambda^\pm(\mu_m)$ of $\Lambda_g(\omega)$ :
\begin{lemma}
	\label{asymp_vp_2}
	When $q$ belongs to $\mathcal{D}_b$, $\Lambda_g^m(\omega)$ has two eigenvalues $\lambda^-(\mu_m)$ and $\lambda^+(\mu_m)$ whose asymptotics are given by :
	\begin{equation*}
	\left\{
\begin{aligned}
&\lambda^-(\mu_m)=-\frac{N(\mu_m)}{\sqrt{f(1)}}-\frac{(\ln h)'(1)}{4\sqrt{f(1)}}+\tilde{O}\bigg(e^{-2\sqrt{\mu_m}}\bigg)	 \\
&\lambda^+(\mu_m)=-\frac{M(\mu_m)}{\sqrt{f(0)}}+\frac{(\ln h)'(0)}{4\sqrt{f(0)}}+\tilde{O}\bigg(e^{-2\sqrt{\mu_m}}\bigg).
\end{aligned}\right.
	\end{equation*}
\end{lemma}

\begin{proof}
	 The characteristic polynomial $\chi_m(X)$ of $\Lambda_g^m(\omega)$ is					
	
	\begin{equation*}
	\chi_m(X)=X^2-{\rm{Tr}}(\Lambda_g^m(\omega))X+\det(\Lambda_g^m(\omega)).
	\end{equation*}
	
\medskip
	
	\noindent To simplify the notations, we set
	
	\medskip
	
	\begin{center}
		$\displaystyle C_0=\frac{\ln(h)'(0)}{4\sqrt{f(0)}},\quad$ $\quad\displaystyle C_1=\frac{\ln(h)'(1)}{4\sqrt{f(1)}}$.
	\end{center}
	
	\medskip
	
	\noindent Thanks to the matrix representation of $\Lambda_g(\omega)$, we see that ${\rm{Tr}}(\Lambda_g^m(\omega))$ and $\det(\Lambda_g^m(\omega))$ satisfy the equalities:
	\begin{equation*}
	\left\{
	\begin{aligned}
	&{\rm{Tr}}(\Lambda_g^m(\omega))=-\frac{M(\mu_m)}{\sqrt{f(0)}}-\frac{N(\mu_m)}{\sqrt{f(1)}}+C_0-C_1.	 \\
	&\det(\Lambda_g^m(\omega))=\bigg(-\frac{M(\mu_m)}{\sqrt{f(0)}}+C_0\bigg)\bigg(-\frac{N(\mu_m)}{\sqrt{f(1)}}-C_1\bigg)+O(\mu_m e^{-2\sqrt{\mu_m}}),\qquad m\to +\infty
	\end{aligned}\right.
	\end{equation*}
	
	\medskip
	
	\noindent The asymptotics of the discriminant $\delta_m$ of $\chi_m(X)$ is thus :
	\begin{equation*}
	\begin{aligned}
	\delta_m &= \bigg(-\frac{M(\mu_m)}{\sqrt{f(0)}}+C_0-\frac{N(\mu_m)}{\sqrt{f(1)}}-C_1\bigg)^2-4\bigg(-\frac{M(\mu_m)}{\sqrt{f(0)}}+C_0\bigg)\bigg(-\frac{N(\mu_m)}{\sqrt{f(1)}}-C_1\bigg)\\
	&\hspace{11cm}+O(\mu_m e^{-2\sqrt{\mu_m}}).\\
	&=\bigg(-\frac{M(\mu_m)}{\sqrt{f(0)}}+C_0+\frac{N(\mu_m)}{\sqrt{f(1)}}+C_1\bigg)^2+O(\mu_m e^{-2\sqrt{\mu_m}}).
	\end{aligned}
	\end{equation*}
	
	\medskip
	
	\noindent Now, let us recall the result obtained by Simon in \cite{simon1999new} :
	
	\medskip
	
	\begin{theorem}
		\label{Simon_2}
		$M(z^2)$ has the following asymptotic expansion :
		\begin{equation*}
		\forall B\in\mathbb{N},\: \:-M(z^2)\underset{z \to \infty}{=}z+\sum_{j=0}^{B}\frac{\beta_j(0)}{z^{j+1}}+o\bigg(\frac{1}{z^{B+1}}\bigg)
		\end{equation*}
		\noindent where, for every $x\in[0,1]$, $\beta_j(x)$ is defined by : $\left\{\begin{aligned}
		&\beta_0(x)=\frac{1}{2}q(x)\\
		&\beta_{j+1}(x)=\frac{1}{2}\beta'_j(x)+\frac{1}{2}\sum_{l=0}^{j}\beta_l(x)\beta_{j-l}(x).
		\end{aligned}\right.$
	\end{theorem}

	\noindent From Remark \ref{WT_symmetry}, by symmetry, one has immediately:
	
	\medskip
	
	\begin{corollary}
		\label{CorSimon_2}
		$N(z^2)$ has the following asymptotic expansion :
		\begin{equation*}
		\forall B\in\mathbb{N},\: \:-N(z^2)\underset{_{\substack{z \to \infty}}}{=}z+\sum_{j=0}^{B}\frac{\gamma_j(0)}{z^{j+1}}+o\bigg(\frac{1}{z^{B+1}}\bigg)
		\end{equation*}
		\noindent where, for all $x\in[0,1]$, $\gamma_j(x)$ is defined by : $\left\{\begin{aligned}
		&\gamma_0(x)=\frac{1}{2}q(1-x)\\
		&\gamma_{j+1}(x)=\frac{1}{2}\gamma'_j(x)+\frac{1}{2}\sum_{l=0}^{j}\gamma_l(x)\gamma_{j-l}(x).
		\end{aligned}\right.$
	\end{corollary}

	\noindent
	If $f(0)\ne f(1)$, we deduce from Theorem \ref{Simon_2} and Corollary \ref{CorSimon_2}:
	
	\medskip
	
	\begin{equation*}
	-\frac{M(\mu_m)}{\sqrt{f(0)}}+\frac{N(\mu_m)}{\sqrt{f(1)}}=\underbrace{\bigg(\frac{1}{\sqrt{f(0)}}-\frac{1}{\sqrt{f(1)}}\bigg)}_{\ne 0}\sqrt{\mu_m}+O\bigg(\frac{1}{\sqrt{\mu_m}}\bigg).
	\end{equation*}
	
	\noindent If $f(0)= f(1)$, we will need the elementary general following lemma:
	
	\medskip
	
	\begin{lemma}
		\label{equiv02}
		We have the equivalence :
		\begin{center}
			$q^{(k)}(0)=(-1)^kq^{(k)}(1),\:\:\forall k\in\mathbb{N}\quad$ $\Leftrightarrow$ $\quad \beta_k(0)=\gamma_k(0),\:\:\forall k\in \mathbb{N}$. 
		\end{center}
	\end{lemma}
	
	\begin{proof}
		
		$\quad$
		
		$\quad$
		
		\noindent Let us prove by induction that, for every $j\in\mathbb{N}$, there exists $P_j\in\mathbb{R}[X_1,...,X_j]$ such that
		\begin{equation}
		\label{pol_2}
		\left\{\begin{aligned}
		&\displaystyle \beta_j(x)=\frac{1}{2^{j+1}} q^{(j)}(x)+P_j(q,q',...,q^{(j-1)})(x)\\
		&\displaystyle \gamma_j(x)=\frac{1}{2^{j+1}} (q\circ \eta) ^{(j)}(x)+P_j(q\circ \eta,(q\circ \eta)',...,(q\circ \eta)^{(j-1)})(x)
		\end{aligned}\right.
		\end{equation}
		\noindent where $\eta (x)=1-x$.

		$\quad$
		
		\begin{enumerate}
			\item[$\bullet$] $\displaystyle \beta_0(x)=\frac{1}{2}q(x)$ and $\displaystyle \gamma_0(x)=\frac{1}{2}q(1-x)$, so the result holds with $P_0(X)=0$.
			
			\item[$\bullet$] Let $j\in\mathbb{N}$ and assume that 
			\begin{equation*}
			\left\{\begin{aligned}
			&\displaystyle \beta_k(x)=\frac{1}{2^{k+1}} q^{(k)}(x)+P_k(q,q',...,q^{(k-1)})(x)\\
			&\displaystyle \gamma_k(x)=\frac{1}{2^{k+1}} (q\circ \eta) ^{(k)}(x)+P_k(q\circ \eta,(q\circ \eta)',...,(q\circ \eta)^{(k-1)})(x),
			\end{aligned}\right.
			\end{equation*}
			for every $0\le k\le j$. Then :
			\begin{equation*}
			\begin{aligned}
			\beta_{j+1}(x)&=\frac{1}{2}\beta'_j(x)+\frac{1}{2}\sum_{l=0}^{j}\beta_l(x)\beta_{j-l}(x)\\
			&= \frac{1}{2^{j+2}}q^{(j+1)}(x)+P_{j+1}(q,q',...,q^{(j)})(x),
			\end{aligned}
			\end{equation*}
			\noindent \noindent where we have set $\displaystyle P_{j+1}(q,q',...,q^{(j)})= \frac{1}{2}\big[P_{j}(q,q',...,q^{(j-1)})\big]'+\frac{1}{2}\sum_{l=0}^{j}\beta_l(x)\beta_{j-l}(x)$. 
			
			\noindent In the same way, one also has
			\begin{equation*}
			\begin{aligned}
			\gamma_{j+1}(x)&=\frac{1}{2}\gamma'_j(x)+\frac{1}{2}\sum_{l=0}^{j}\gamma_l(x)\gamma_{j-l}(x)\\
			&= \frac{1}{2^{j+2}}(q\circ \eta)^{(j+1)}(x)+P_{j+1}(q\circ \eta,(q\circ \eta)',...,(q\circ \eta)^{(j)})(x).
			\end{aligned}
			\end{equation*}

			\item[$\bullet$] Hence, we get the result by induction.
		\end{enumerate}
		
		\medskip
		
		\noindent We are now able to prove the equivalence.
		
		$\quad$
		
		\noindent $(\Rightarrow)$ If $q^{(j)}(0)=(-1)^jq^{(j)}(1)$ for every $j\in\mathbb{N}$ then one has, for every $k\in\mathbb{N}$ and every $P\in\mathbb{R}[X_1,...,X_k]$ :
		\begin{equation*}
		P(q,q',...,q^{(k-1)})(0)=P(q\circ\eta,(q\circ\eta)',...,(q\circ\eta)^{(k-1)})(0),
		\end{equation*}	
		\noindent so, thanks to (\ref{pol_2}):\begin{center}
			$\beta_j(0)=\gamma_j(0)$ for every $j\in\mathbb{N}$.
		\end{center}
		
		\medskip
		
		\noindent $(\Leftarrow)$ Conversely, assume that there is $j\in\mathbb{N}$ such that $q^{(j)}(0)\ne (-1)^jq^{(j)}(1)$ and set $k=\min\{j\in\mathbb{N},\: q^{(j)}(0)\ne (-1)^jq^{(j)}(1) \}$. As previously, for every $P\in\mathbb{R}[X_1,...,X_{k}]$ :
		\begin{center}
			$P(q,q',...,q^{(k-1)})(0)=P(q\circ\eta,(q\circ\eta)',...,(q\circ\eta)^{(k-1)})(0)$.
		\end{center} Hence :
		\begin{equation*}
		\begin{aligned}
		\beta_{k}(0)\ne \gamma_{k}(0) &\Leftrightarrow \frac{1}{2^{k+1}} q^{(k)}(0)+P_k(q,...,q^{(k-1)})(0) \ne \frac{1}{2^{k+1}} (q\circ \eta)^{(k)}(0)\\
		&\hspace{7cm}+P_k\big((q\circ \eta),...,(q\circ \eta)^{(k-1)}\big)(0)\\
		&\Leftrightarrow \frac{1}{2^{k+1}} q^{(k)}(0)\ne \frac{1}{2^{k+1}} (q\circ \eta)^{(k)}(0)\\
		&\Leftrightarrow q^{(k)}(0)\ne (-1)^kq^{(k)}(1),
		\end{aligned}
		\end{equation*}
		
		\medskip
		
		\noindent and that is true by definition of $k$. 
	\end{proof}
	
	\noindent As we have assumed that $q$ belongs to $\mathcal{D}_b$, by setting $k=\min\{j\in\mathbb{N},\: q^{(j)}(0)\ne (-1)^jq^{(j)}(1) \}$, we get, thanks to (\ref{Simon_2}), (\ref{CorSimon_2}) and Lemma \ref{equiv02}:
			\begin{equation*}
		-\frac{M(\mu_m)}{\sqrt{f(0)}}+\frac{N(\mu_m)}{\sqrt{f(1)}}=\underbrace{\bigg(\frac{\beta_k(0)-\gamma_k(0)}{\sqrt{f(0)}}\bigg)}_{\ne 0}\frac{1}{(\sqrt{\mu_m})^{k+1}}+O\bigg(\frac{1}{(\sqrt{\mu_m})^{k+2}}\bigg).
		\end{equation*}

	\noindent In both cases, there is $A\in\mathbb{R}\backslash\{0\}$ and $k\in\mathbb{Z}$ such that
	\begin{equation}
	-\frac{M(\mu_m)}{\sqrt{f(0)}}+\frac{N(\mu_m)}{\sqrt{f(1)}}=A(\sqrt{\mu_m})^k + o\big((\sqrt{\mu_m})^k\big)
	\end{equation}
	
	$\quad$
	
	\noindent Thus, recalling that
	\begin{equation*}
	\begin{aligned}
	\delta &=\bigg(-\frac{M(\mu_m)}{\sqrt{f(0)}}+C_0+\frac{N(\mu_m)}{\sqrt{f(1)}}+C_1\bigg)^2+O(\mu_m e^{-2\sqrt{\mu_m}}),
	\end{aligned}
	\end{equation*}
	
	\noindent we obtain, as $A$ is not $0$:
	
	\begin{equation*}
	\begin{aligned}
	\sqrt{\delta}&=\bigg[\bigg(-\frac{M(\mu_m)}{\sqrt{f(0)}}+C_0+\frac{N(\mu_m)}{\sqrt{f(1)}}+C_1\bigg)^2+O(\mu_m e^{-2\sqrt{\mu_m}})\bigg]^{\frac{1}{2}}\\
	&=\bigg[\bigg(A(\sqrt{\mu_m})^k + C_0+C_1+ o\big((\sqrt{\mu_m})^k\big)\bigg)^2+O(\mu_m e^{-2\sqrt{\mu_m}})\bigg]^{\frac{1}{2}}\\
	&= \bigg|A(\sqrt{\mu_m})^k + C_0+C_1+ o\big((\sqrt{\mu_m})^k\big)\bigg|\bigg[1+O\bigg(\big(\sqrt{\mu_m}\big)^{-2k+2}e^{-2\sqrt{\mu_m}}\bigg)\bigg]^{\frac{1}{2}}  \\
	&=\bigg|\frac{N(\mu_m)}{\sqrt{f(1)}}-\frac{M(\mu_m)}{\sqrt{f(0)}}+C_0+C_1\bigg|\bigg[1+O\bigg(\big(\sqrt{\mu_m}\big)^{-2k+2}e^{-2\sqrt{\mu_m}}\bigg)\bigg]\\
	&=\bigg|\frac{N(\mu_m)}{\sqrt{f(1)}}-\frac{M(\mu_m)}{\sqrt{f(0)}}+C_0+C_1\bigg|+\tilde{O}\big(e^{-2\sqrt{\mu_m}}\big).
	\end{aligned}
	\end{equation*}
	
	\medskip
	
	\noindent Therefore, the two eigenvalues $\lambda^\pm(\mu_m)$ of $\Lambda_g^m(\omega)$ satisfy the asymptotics equalities
	
	\begin{equation*}
	\left\{
	\begin{aligned}
	&\lambda^-(\mu_m)=-\frac{N(\mu_m)}{\sqrt{f(1)}}-\frac{\ln(h)'(1)}{4\sqrt{f(1)}}+\tilde{O}\big(e^{-2\sqrt{\mu_m}}\big) \\
	&\lambda^+(\mu_m)=-\frac{M(\mu_m)}{\sqrt{f(0)}}+\frac{\ln(h)'(0)}{4\sqrt{f(0)}}+\tilde{O}\big(e^{-2\sqrt{\mu_m}}\big).
	\end{aligned}\right.
	\end{equation*}
\end{proof}

$\quad$

\noindent In fact, the eigenvalues $\mu_m$ being counted with multiplicity, the asymptotics of Lemma \ref{asymp_vp_2} won't be sufficiently precise for our purpose. Indeed, by Theorem \ref{Simon_2} and its Corollary, the Weyl-Titchmarsh functions always satisfy
\begin{equation*}
\left\{
	\begin{aligned}
	&-N(z^2)=z+O\bigg(\frac{1}{z}\bigg)\\
	&-M(z^2)=z+O\bigg(\frac{1}{z}\bigg).
	\end{aligned}
	\right.
\end{equation*}
\noindent So, using the Weyl law, one can prove immediately that
\begin{equation}
\label{Azrael_2}
\left\{
\begin{aligned}
&\lambda^-(\mu_m)=\frac{\sqrt{\mu_m}}{\sqrt{f(1)}} -\frac{(\ln h)'(1)}{4\sqrt{f(1)}}+ O\bigg(\frac{1}{\sqrt{\mu_m}}\bigg) = C_1m^{\frac{1}{n-1}}+O(1)\\
&\lambda^+(\mu_m)=\frac{\sqrt{\mu_m}}{\sqrt{f(0)}} +\frac{(\ln h)'(0)}{4\sqrt{f(0)}}+O\bigg(\frac{1}{\sqrt{\mu_m}}\bigg) = C_0m^{\frac{1}{n-1}}+O(1)
\end{aligned}
\right.
\end{equation}
\noindent with $C_0,C_1>0$. In order to have a much more precise asymptotic expansion in $m$, let us introduce the set
\begin{equation}
\label{banane_2}
\Sigma\big(\Lambda_g(\omega)\big)=\{\lambda^\pm(\kappa_m),\: m\in\mathbb{N}\}
\end{equation}

\medskip

\noindent where the $\kappa_m$ are the eigenvalues of $-\Delta_{\mathbb{S}}$ counted \textit{without multiplicity}. We have an explicit formula for $\kappa_m$ (cf \cite{shubin1987pseudodifferential}) given by \begin{center}
	$\kappa_m=m(m+n-2)$.
\end{center} 
\noindent From now on, we will always use the asymptotics of Lemma \ref{asymp_vp_2} by replacing $\mu_m$ by $\kappa_m$. Of course, one can also define the closeness between $\Sigma\big(\Lambda_g(\omega)\big)$ and $\Sigma\big(\Lambda_{\tilde{g}}(\omega)\big)$ up to a sequence $(\varepsilon_m)$ by replacing $\mu_m$ by $\kappa_m$ in Definitions \ref{proximite_2} and \ref{expcloseness_2}.

\section{A local uniqueness result}

\noindent Now, let us give the proof of Theorem \ref{pseudostab_2}. 

\begin{proof} 
	
	Let $(\varepsilon_m)$ be a sequence such that $\displaystyle \varepsilon_m=\tilde{O}(e^{-2a\sqrt{\mu_m}})\:$ and $\:\displaystyle 	\sigma(\Lambda_{g}(\omega))\underset{(\varepsilon_m)}{\asymp}\sigma(\Lambda_{\tilde{g}}(\omega))$.
	
\noindent 	Then, there is a subsequence of $(\varepsilon_m)$, that we will still denote $(\varepsilon_m)$ which satisfies the estimate
	\begin{equation*}
\varepsilon_m=\tilde{O}(e^{-2a\sqrt{\kappa_m}})
\end{equation*}
	and the relation
	\begin{equation*}
\Sigma(\Lambda_{g}(\omega))\underset{(\varepsilon_m)}{\asymp}\Sigma(\Lambda_{\tilde{g}}(\omega)).
\end{equation*}

		\begin{lemma}
			\label{alt_2_1}
		Under the hypothesis $\Sigma(\Lambda_{g}(\omega))\underset{(\varepsilon_m)}{\asymp}\Sigma(\Lambda_{\tilde{g}}(\omega))$, we have the alternative :
		\begin{center}
			$\left\{
			\begin{aligned}
			&f(0)=\tilde{f}(0)	 \\
			&f(1)=\tilde{f}(1)
			\end{aligned}\right.\:\:$  or  $\:\:\left\{\begin{aligned}
			&f(0)=\tilde{f}(1)	 \\
			&f(1)=\tilde{f}(0).
			\end{aligned}\right.$
		\end{center}
	\end{lemma}	
	
	$\quad$
	
	\begin{proof}
	
	$\quad$
	
	$\quad$
	
	\begin{enumerate}[label=\alph*), align=left, leftmargin=*, noitemsep]
		\item[$\bullet$] We first show the equality
		\begin{equation}
		\label{coef_2}
		\sqrt{f(0)}+\sqrt{f(1)}=\sqrt{\tilde{f}(0)}+\sqrt{\tilde{f}(1)}
		\end{equation}
	\end{enumerate}
	
	\medskip
	
	\noindent As $\displaystyle \sqrt{\kappa_m}=m+\frac{n-2}{2}+O\bigg(\frac{1}{m}\bigg)$, we get from Lemma \ref{asymp_vp_2} the following asymptotics
	\begin{equation}
	\label{canard_2}
	\left\{
	\begin{aligned}
	&\lambda^-(\kappa_m)=\frac{m}{\sqrt{f(1)}} +\frac{n-2}{2\sqrt{f(1)}} -\frac{(\ln h)'(1)}{4\sqrt{f(1)}}+ O\bigg(\frac{1}{m}\bigg) \\
	&\lambda^+(\kappa_m)=\frac{m}{\sqrt{f(0)}}+\frac{n-2}{2\sqrt{f(0)}} +\frac{(\ln h)'(0)}{4\sqrt{f(0)}}+O\bigg(\frac{1}{m}\bigg)
	\end{aligned}
	\right.
	\end{equation}
	\noindent

	\noindent Let $L>0$. The sequences $(\lambda^{\pm}(\kappa_m))$ are asymptotically in arithmetic progression. Combined with the relation
	\begin{equation*}
	\Sigma\big(\Lambda_g(\omega)\big)\underset{(\varepsilon_m)}{\asymp}\Sigma\big(\Lambda_{\tilde{g}}(\omega)\big),
	\end{equation*}
	\noindent this leads to the equality (when $L\to+\infty$)
	\begin{equation}
	\label{egcard}
	\begin{aligned}
	\mathrm{Card}&\big\{m\in\mathbb{N},\: \lambda^-(\kappa_m)\le L\big\}\quad+\quad\mathrm{Card}\big\{m\in\mathbb{N},\: \lambda^+(\kappa_m)\le L\big\}\\&\quad=\quad\mathrm{Card}\big\{m\in\mathbb{N},\: \tilde{\lambda}^-(\kappa_m)\le L\big\}\quad+\quad\mathrm{Card}\big\{m\in\mathbb{N},\: \tilde{\lambda}^+(\kappa_m)\le L\big\} + O(1).
	\end{aligned}
	\end{equation} 
	
	$\quad$ 
	
	\noindent Thanks to the asymptotics (\ref{canard_2}), we deduce that :	
	\begin{equation*}
	\begin{aligned}
	\mathrm{Card}\big\{m\in\mathbb{N},\: m&\le \sqrt{f(1)}L\big\}\quad+\quad\mathrm{Card}\big\{m\in\mathbb{N},\: m\le \sqrt{f(0)}L\big\}\\&\quad=\quad\mathrm{Card}\big\{m\in\mathbb{N},\: m\le \sqrt{\tilde{f}(1)}L\big\}\quad+\quad\mathrm{Card}\big\{m\in\mathbb{N},\: m\le \sqrt{\tilde{f}(0)}L\big\} + O(1),
	\end{aligned}
	\end{equation*}
	and then that :
	\begin{equation*}
	\sqrt{f(1)}L+\sqrt{f(0)}L=\sqrt{\tilde{f}(1)}L+\sqrt{\tilde{f}(0)}L+O(1),\qquad L\to+\infty.
	\end{equation*}
	
	\medskip
	
	\noindent As $L$ is any positive number, this proves (\ref{coef_2}).
	
$\quad$

		\noindent $\bullet$ Now, we have to show : $f(0)\in\{\tilde{f}(0),\tilde{f}(1)\}$.

$\quad$
	
	\noindent Assume that it is not true, for example 
	\begin{equation}
	\label{chou_2}
	\displaystyle f(0)<\min\{\tilde{f(0)},\tilde{f}(1)\}.
	\end{equation}
	Then (\ref{coef_2}) implies 
	\begin{equation}
	\label{fleur_2}
	f(1)>\max\{\tilde{f}(0),\tilde{f}(1)\}.
	\end{equation}

	$\quad$

	\noindent Let $m$ be in $\mathbb{N}$. There is $\ell\in\mathbb{N}$ such that
	\begin{equation}
	\label{babayetu_2}
	\lambda^-(\kappa_m)-\tilde{\lambda}^-(\kappa_\ell)= O(\varepsilon_m),
	\end{equation} 
	
	\noindent with $|O(\varepsilon_m)|\le \varepsilon_m$. From the assumption  $\sigma\big(\Lambda_g(\omega)\big)\underset{(\varepsilon_m)}{\asymp}\sigma\big(\Lambda_{\tilde{g}}(\omega)\big)$, each of the eigenvalues $\lambda^-(\kappa_{m-1})$, $\lambda^-(\kappa_{m+1})$ and $\lambda^-(\kappa_{m+2})$ is also close to an element of $\sigma\big(\Lambda_{\tilde{g}}(\omega)\big)$. If $m$ is large enough, the situation is necessary the following:
	
	
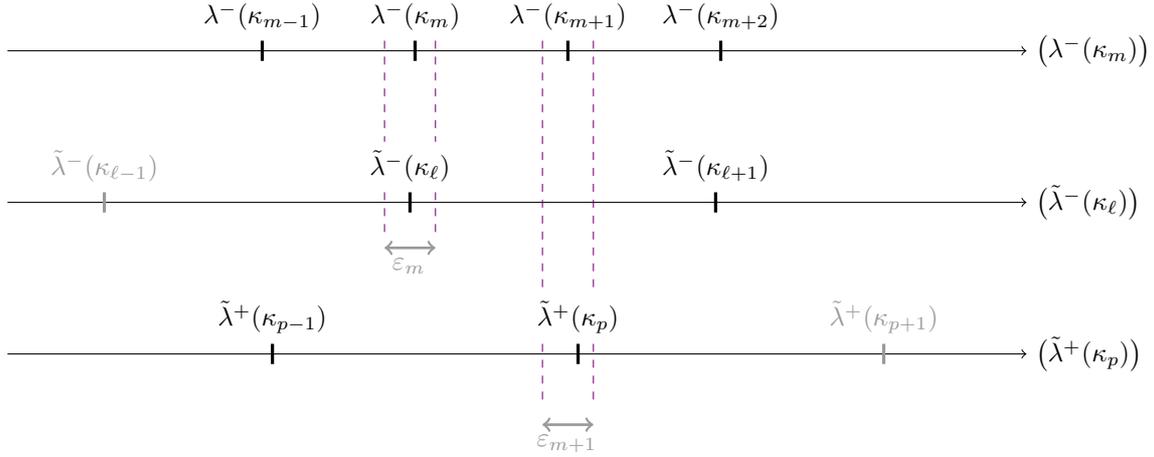
\begin{center}
		\begin{tikzpicture}[scale=6.7]
		[domain=0:4]
		\draw[->] (0,0) -- (2,0) node[right] {$\big(\lambda^-(\kappa_m)\big)$} ;
		\draw [black,line width=1.2pt](0.5,-0.02)--(0.5,0.02)node[above]{$\lambda^-(\kappa_{m-1})$};
		\draw [black,line width=1.2pt](0.8,-0.02)--(0.8,0.02)node[above]{$\lambda^-(\kappa_m)$};
		\draw [black,line width=1.2pt](1.1,-0.02)--(1.1,0.02)node[above]{$\lambda^-(\kappa_{m+1})$};
		\draw [black,line width=1.2pt](1.4,-0.02)--(1.4,0.02)node[above]{$\lambda^-(\kappa_{m+2})$};
		
		\draw[->] (0,-0.3) -- (2,-0.3) node[right] {$\big(\tilde{\lambda}^-(\kappa_\ell)\big)$};
		\draw [black!40,line width=1.2pt](0.19,-0.32)--(0.19,-0.28)node[above]{$\tilde{\lambda}^-(\kappa_{\ell-1})$};
		\draw [black,line width=1.2pt](0.79,-0.32)--(0.79,-0.28)node[above]{$\tilde{\lambda}^-(\kappa_\ell)$};
		\draw [black,line width=1.2pt](1.39,-0.32)--(1.39,-0.28)node[above]{$\tilde{\lambda}^-(\kappa_{\ell+1})$};
		
		\draw[->] (0,-0.6) -- (2,-0.6) node[right] {$\big(\tilde{\lambda}^+(\kappa_p)\big)$};
		\draw [black,line width=1.2pt](0.52,-0.62)--(0.52,-0.58)node[above]{$\tilde{\lambda}^+(\kappa_{p-1})$};
		\draw [black,line width=1.2pt](1.12,-0.62)--(1.12,-0.58)node[above]{$\tilde{\lambda}^+(\kappa_{p})$};
		\draw [black!40,line width=1.2pt](1.72,-0.62)--(1.72,-0.58)node[above]{$\tilde{\lambda}^+(\kappa_{p+1})$};

		\draw [dashed,color=violet] (0.74,0.02) -- (0.74,-0.18) ;
		\draw [dashed,color=violet] (0.84,0.02) -- (0.84,-0.18) ;
		\draw [dashed,color=violet] (0.74,-0.28) -- (0.74,-0.36) ;
		\draw [dashed,color=violet] (0.84,-0.28) -- (0.84,-0.36) ;
		\draw [<->,samples=100,color=black!40,line width=1pt](0.74,-0.39)--(0.84,-0.39)node[below]{$\varepsilon_m \qquad$};
		
		\draw [dashed,color=violet] (1.05,0.02) -- (1.05,-0.48) ;
		\draw [dashed,color=violet] (1.15,0.02) -- (1.15,-0.48) ;
		\draw [dashed,color=violet] (1.05,-0.58) -- (1.05,-0.7) ;
		\draw [dashed,color=violet] (1.15,-0.58) -- (1.15,-0.7) ;
		\draw [<->,samples=100,color=black!40,line width=1pt] (1.05,-0.74)--(1.15,-0.74)node[below]{$\varepsilon_{m+1} \qquad$};
		\end{tikzpicture}
	\end{center}

	\noindent Indeed, since $\displaystyle f(1)>\tilde{f}(1)$, for $m$ large enough, from (\ref{canard_2}) and (\ref{babayetu_2}), we have
	\begin{equation*}
	\begin{aligned}
	\lambda^-(\kappa_{m+1})&=\lambda^-(\kappa_{m}) +\frac{1}{\sqrt{f(1)}}+o(1)\\
	&=\tilde{\lambda}^-(\kappa_\ell)+O(\varepsilon_m)+\frac{1}{\sqrt{f(1)}}+o(1)\\
	&=\tilde{\lambda}^-(\kappa_{\ell+1})+\underbrace{\frac{1}{\sqrt{f(1)}}-\frac{1}{\sqrt{\tilde{f}(1)}}}_{<0}+O(\varepsilon_m)+o(1)
	\end{aligned}
	\end{equation*}
	\noindent Let us chose $m$ large enough such that
	\begin{equation*}
	\left\{\begin{aligned}
	&\frac{1}{\sqrt{f(1)}}-\frac{1}{\sqrt{\tilde{f}(1)}}+O(\varepsilon_m)+o(1) <\varepsilon_{m+1}\\
	&O(\varepsilon_m)+\frac{1}{\sqrt{f(1)}} +o(1)>\varepsilon_{m+1}
	\end{aligned}\right.
	\end{equation*}
	\noindent Then:
	\begin{equation*}
	\tilde{\lambda}^-(\kappa_{\ell})+\varepsilon_{m+1}<\lambda^-(\kappa_{m+1})<\tilde{\lambda}^-(\kappa_{\ell+1})-\varepsilon_{m+1}
	\end{equation*}
	
	\noindent so, as $\big(\tilde{\lambda}^-(\kappa_\ell)\big)$ is a strictly increasing sequence, for $m$ large enough $\lambda^-(\kappa_{m+1})$ is not $\varepsilon_{m+1}$-close to any element of $\big(\tilde{\lambda}^-(\kappa_\ell)\big)$ : there is thus $p\in\mathbb{N}$ such that \begin{equation}
	\label{mozart_2}
	\lambda^-(\kappa_{m+1})-\tilde{\lambda}^+(\kappa_p)=O(\varepsilon_{m+1}),
	\end{equation}
	
	\noindent with $|O(\varepsilon_{m+1})|\le \varepsilon_{m+1}$. 
	
	$\quad$
	
	\noindent For the same reasons, we get also
	$\displaystyle \tilde{\lambda}^-(\kappa_{\ell-1})+\varepsilon_{m-1}<\lambda^-(\kappa_{m-1})<\tilde{\lambda}^-(\kappa_{\ell})-\varepsilon_{m-1}$
	\noindent and one deduces that (the previous picture helps to visualize it)
	\begin{equation}
	\label{beethov_2}
	\lambda^-(\kappa_{m-1})-\tilde{\lambda}^+(\kappa_{p-1})=O(\varepsilon_{m-1}),
	\end{equation} 
	
	\noindent with $|O(\varepsilon_{m-1})|\le \varepsilon_{m-1}$. Since we have $f(1)>\tilde{f}(0)$, we get also
	\begin{equation*}
	\tilde{\lambda}^+(\kappa_{p})+\varepsilon_{m+2}<\lambda^-(\kappa_{m+2})<\tilde{\lambda}^+(\kappa_{p+1})-\varepsilon_{m+2}.
	\end{equation*}
	Consequently
	\begin{equation}
	\label{Faure_2}
	\lambda^-(\kappa_{m+2})-\tilde{\lambda}^-(\kappa_{\ell+1})=O(\varepsilon_{m+2})
	\end{equation} 
	
	\noindent with $|O(\varepsilon_{m+2})|\le \varepsilon_{m+2}$. 
	
	$\quad$
	
	\noindent Then, by (\ref{babayetu_2}), (\ref{mozart_2}), (\ref{beethov_2}) and (\ref{Faure_2}), we have for $m$ large enough :
	
	\begin{equation*}
	\left\{\begin{aligned}
	&\lambda^-(\kappa_{m+1})-\lambda^-(\kappa_{m-1})=\tilde{\lambda}^+(\kappa_{p})-\tilde{\lambda}^+(\kappa_{p-1})+O(\varepsilon_{m+1}) - O(\varepsilon_{m-1})\\
	&\lambda^-(\kappa_{m+2})-\lambda^-(\kappa_{m})=\tilde{\lambda}^-(\kappa_{\ell+1})-\tilde{\lambda}^-(\kappa_{\ell})+O(\varepsilon_{m+2}) - O(\varepsilon_{m}).\\
	\end{aligned}
	\right.
	\end{equation*}
	
	\medskip
	
	\noindent In particular:
	\begin{equation*}
	\left\{\begin{aligned}
	&\frac{2}{\sqrt{f(1)}}=\frac{1}{\sqrt{\tilde{f}(0)}}+o(1)\\
	&\frac{2}{\sqrt{f(1)}}=\frac{1}{\sqrt{\tilde{f}(1)}}+o(1),\\
	\end{aligned}
	\right.
	\end{equation*}
	and so, taking $m\to+\infty$, one deduces 
	\begin{equation*}
	2\sqrt{\tilde{f}(0)}=\sqrt{f(1)}\quad \mathrm{and}\quad 2\sqrt{\tilde{f}(1)}=\sqrt{f(1)}.
	\end{equation*}
	
	\noindent As $\sqrt{f(0)}+\sqrt{f(1)}=\sqrt{\tilde{f}(0)}+\sqrt{\tilde{f}(1)}$, we get
	\begin{equation*}
	\begin{aligned}
	2\sqrt{f(0)}= 2\bigg(\sqrt{\tilde{f}(0)}+\sqrt{\tilde{f}(1)} -\sqrt{f(1)}\bigg) &=\bigg(2\sqrt{\tilde{f}(0)}-\sqrt{f(1)}\bigg)+\bigg(2\sqrt{\tilde{f}(1)}-\sqrt{f(1)}\bigg)\\
	&=0.
	\end{aligned}
	\end{equation*}
	and we get a contradiction as $f(0)>0$.
	
	$\quad$
	
	\noindent Hence $f(0)\in\{\tilde{f}(0),\tilde{f}(1)\}$. The equality (\ref{coef_2}) gives the conclusion.
\end{proof}

	\noindent Assume from now that $f(0)=\tilde{f}(0)$ and $f(1)=\tilde{f}(1)$. The other case is obtained by substituting the roles of $\tilde{\lambda}^-(\kappa_m)$ and $\tilde{\lambda}^+(\kappa_m)$.

		\subsection{The case $\bf f(0)\ne f(1)$}
	
	\noindent Without loss of generality, we assume that $f(0)<f(1)$. The following lemma focuses on the sequence $\big(\lambda^-(\kappa_p)\big)$ since this is the sequence that grows slower. If we had treated the case $f(0)>f(1)$, the sequence considered in this section would have been $\big(\lambda^+(\kappa_p)\big)$.
	
	\medskip
	
	\begin{lemma}
		\label{un_sur_deux_local_2}
		Assume that $f(0)<f(1)$. There is an infinite subset $\mathcal{L}$ of $\mathbb{N}$ such that
		\begin{enumerate}
			\item[$\bullet$] $\:\:\lambda^-(\kappa_p)-\tilde{\lambda}^-(\kappa_p)=\tilde{O}(e^{-2a\kappa_p}),\quad p\in\mathcal{L}$.
			\item[$\bullet$] For all $m$ in $\mathbb{N}$ large enough, $\{m,m+1\}\cap\mathcal{L}\ne \emptyset$.
		\end{enumerate}
	\end{lemma}

	\begin{proof}
	\noindent Let us denote $U$ the subset of $\{\lambda^-(\kappa_m),\:\: m\in\mathbb{N}\}$ such that :
	\begin{center}
		$U\underset{(\varepsilon_m)}{\subsetsim}\:\{\tilde{\lambda}^+(\kappa_m),\:\: m\in\mathbb{N}\}$
	\end{center} 
	
	\noindent \textit{Case 1} : $U$ is finite. Then there is $m_0\in\mathbb{N}$ such that :
	\begin{equation*}
	\forall m\ge m_0,\quad \exists p\in\mathbb{N},\:\: |\lambda^-(\kappa_m)-\tilde{\lambda}^-(\kappa_p)|\le \varepsilon_m
	\end{equation*} 
	\noindent This can be written as :
	\begin{equation*}
	\lambda^-(\kappa_m)-\tilde{\lambda}^-(\kappa_p)= O(\varepsilon_m)
	\end{equation*}
	\noindent with $|O(\varepsilon_m)|\le \varepsilon_m$. By replacing the eigenvalues by their asymptotics (\ref{canard_2}) in the previous equality, one finds, as $f(1)=\tilde{f}(1)$:
	\begin{equation*}
	\frac{m}{\sqrt{f(1)}}+\frac{n-2}{2\sqrt{f(1)}}-\frac{\ln(h)'(1)}{4\sqrt{f(1)}}=\frac{p}{\sqrt{f(1)}}+\frac{n-2}{2\sqrt{f(1)}}-\frac{\ln(\tilde{h})'(1)}{4\sqrt{f(1)}}+O(\varepsilon_m)
	\end{equation*}
	
	\medskip
	
	\begin{enumerate}
		\item[$\bullet$] If $n=2$ then $h=f^{n-2}$ is a constant. One has $\displaystyle \frac{m}{\sqrt{f(1)}}=\frac{p}{\sqrt{f(1)}}+O(\varepsilon_m)$. Hence, as $m$ and $p$ are integers, if $m$ is large enough, we have $m=p$.
		\item[$\bullet$] If $n\ge 3$, then $\displaystyle m-p=\frac{\ln(h)'(1)}{4}-\frac{\ln(\tilde{h})'(1)}{4}+O(\varepsilon_m)$. By hypothesis :
		\begin{equation*}
		\bigg|\frac{\ln(h)'(1)}{4}-\frac{\ln(\tilde{h})'(1)}{4}\bigg|=\frac{n-2}{4}\bigg|\frac{f'(1)}{f(1)}-\frac{\tilde{f}'(1)}{f(1)}\bigg|\le \frac{n-2}{4}\times \frac{2}{n-2}= \frac{1}{2}.
		\end{equation*}
		Hence, $m=p$ for $m,p$ greater than some integer $m_0$. 
	\end{enumerate}
	
	$\quad$
	
	\noindent The set $\mathcal{L}=\{m\in\mathbb{N},\:\:m\ge m_0 \}$ satisfies the properties of Lemma \ref{un_sur_deux_local_2}.
	
	$\quad$	
	
	\noindent \textit{Case 2} : $U$ is infinite. Then, there exists $\varphi,\psi :\mathbb{N}\to\mathbb{N}$ two strictly increasing functions such that :
	\begin{equation}
	\label{egu_2}
	\lambda^-(\kappa_{\psi(m)})-\tilde{\lambda}^+(\kappa_{\varphi(m)})=O(\varepsilon_{\psi(m)}),
	\end{equation}
	
	\noindent with $|O(\varepsilon_{\psi(m)})|\le\varepsilon_{\psi(m)}$.
	
	\medskip
	
	\begin{remark}
		\label{imeg}
		$\varphi$ and $\psi$ are built in such a way that an integer $m\in\mathbb{N}$ that is not in the range of $\psi$ (respectively not in the range of $\varphi$) satisfies $|\lambda^-(\kappa_m)-\tilde{\lambda}^-(\kappa_n)|<\varepsilon_m$ for some $n\in\mathbb{N}$ (respectively $|\lambda^+(\kappa_n)-\tilde{\lambda}^+(\kappa_m)|<\varepsilon_n$ for some $n\in\mathbb{N}$). 
	\end{remark}
	
	$\quad$
	
	\noindent By replacing $\lambda^+(\kappa_{\varphi(m)})$ and $\tilde{\lambda}^-(\kappa_{\psi(m)})$ with their asymptotics in the equality (\ref{egu_2}), one has :
	\begin{equation}
	\label{subsequence_2}
	\frac{\varphi(m)}{\sqrt{f(0)}}=\frac{\psi(m)}{\sqrt{f(1)}}+C+O(\varepsilon_{\psi(m)})+O\bigg(\frac{1}{\varphi(m)}\bigg)
	\end{equation} 
	
	$\quad$
	
	\medskip
	
	\noindent where $\displaystyle C=-\frac{\ln(h)'(1)}{4\sqrt{f(1)}}-\frac{\ln(h)'(0)}{4\sqrt{f(0)}}+\frac{n-2}{2\sqrt{f(1)}}-\frac{n-2}{2\sqrt{f(0)}}$.
	
	\medskip		
			
				\begin{lemma}
				\label{image_2}
				There is an integer $m_0\in\mathbb{N}$ such that \big($m\ge m_0\Rightarrow\psi(m+1)\ge\psi(m)+ 2\big)$. 
			\end{lemma}
			
			\begin{proof}
				Set $\displaystyle B=\frac{\sqrt{f(1)}}{\sqrt{f(0)}}>1\:$ and $\: C'=-\sqrt{f(1)}C$. From (\ref{subsequence_2}), one gets :
				\begin{equation*}
				\psi(m)=B\varphi(m)+C'+o(1).
				\end{equation*}
				Assume $\psi(m+1)=\psi(m)+1$. Then :
				\begin{equation*}
				\begin{aligned}
				\psi(m)+1=\psi(m+1)&=B\varphi(m+1)+C'+o(1)\\
				&\ge B(\varphi(m)+1)+C'+o(1)\\
				&=B\varphi(m)+C'+B+o(1)\\
				&=\psi(m)+B+o(1).
				\end{aligned}
				\end{equation*}
				
				\noindent Thus, we find :
				
				\begin{equation*}
				1\ge B+o(1)
				\end{equation*} 
				
				\noindent which is false if $m\ge m_0$ for some $m_0\in\mathbb{N}$.
			\end{proof}

				\noindent Consequently, the range of $\psi$ does not contain two consecutive integers. Let us set :
			\begin{equation*}
			\mathcal{L}=\{m\in\mathbb{N}\:\big|\: m\ge m_0,\:\: m\notin \mathrm{range}(\psi)\}.
			\end{equation*}
			
			\noindent Then $\mathcal{L}$ satisfies the condition $\mathcal{L}\cap\{m,m+1\}\ne\emptyset$ for any $m\ge m_0$. Moreover, for any $m\in\mathcal{L}$, there is $\ell\in\mathbb{N}$ such that $|\lambda^-(\kappa_m)-\tilde{\lambda}^-(\kappa_\ell)|\le \varepsilon_m$. One deduces, as previously, $m=\ell$.
	\end{proof}
	
	\begin{remark}
		\label{constants_2}
		Lemma \ref{un_sur_deux_local_2} and asymptotics (\ref{canard_2}) imply in particular \begin{center}
			$\displaystyle \frac{(\ln h)'(1)}{4\sqrt{f(1)}}=\frac{(\ln \tilde{h})'(1)}{4\sqrt{\tilde{f}(1)}}$.
		\end{center} 
	\end{remark}
	
	\medskip
	
	\noindent Now, let us recall an asymptotic integral representation of the Weyl-Titschmarsh function $N(z^2)$ obtained by Simon in \cite{simon1999new} (Theorem 3.1) :
	
	\medskip
	
	\begin{theorem}
		\label{Simon_repr_2}
		For every $0<a<1$, there is $A\in L^1([0,a])$ such that
		\begin{equation}
		N(z^2)=-z-\int_{0}^{a}A(x)e^{-2xz}dx + \tilde{O}(e^{-2az}),\qquad z\to+\infty.
		\end{equation}
	\end{theorem}
	
	$\quad$
	
	\noindent From the asymptotic of $\lambda^-(\kappa_m)$ obtained in Lemma \ref{asymp_vp_2}, we have
	\begin{equation*}
	\lambda^-(\kappa_m)=-\frac{N(\kappa_m)}{\sqrt{f(1)}}-\frac{(\ln h)'(1)}{4\sqrt{f(1)}}+\tilde{O}\bigg(e^{-2\sqrt{\kappa_m}}\bigg).	 
	\end{equation*}
	
	$\quad$
	
	\noindent Hence
	\begin{equation*}
	\begin{aligned}
	\lambda^-(\kappa_m)-&\tilde{\lambda}^-(\kappa_m)=
	 \tilde{O}(e^{-2a\sqrt{\kappa_m}}) \\
	&\Rightarrow -\frac{N(\kappa_m)}{\sqrt{f(1)}}-\frac{(\ln h)'(1)}{4\sqrt{f(1)}}=-\frac{\tilde{N}(\kappa_m)}{\sqrt{\tilde{f}(1)}}-\frac{(\ln \tilde{h})'(1)}{4\sqrt{\tilde{f}(1)}}+\tilde{O}(e^{-2a\sqrt{\kappa_m}})+\tilde{O}\bigg(e^{-2\sqrt{\kappa_m}}\bigg)
	\end{aligned}
	\end{equation*}
	\noindent  The equalities $f(1)=\tilde{f}(1)\:$ and $\:\displaystyle \frac{(\ln h)'(1)}{4\sqrt{f(1)}}=\frac{(\ln \tilde{h})'(1)}{4\sqrt{\tilde{f}(1)}}\:$ (see Remark \ref{constants_2}) imply
	\begin{equation*}
	\begin{aligned}
	\lambda^-(\kappa_m)-\tilde{\lambda}^-(\kappa_m)=
	\tilde{O}(e^{-2a\sqrt{\kappa_m}})&\Rightarrow N(\kappa_m)=\tilde{N}(\kappa_m)+\tilde{O}(e^{-2a\sqrt{\kappa_m}})\\
	&\Rightarrow \int_{0}^{a}A(x)e^{-2x\sqrt{\kappa_m}}dx=\int_{0}^{a}\tilde{A}(x)e^{-2x\sqrt{\kappa_m}}dx + \tilde{O}(e^{-2a\sqrt{\kappa_m}})\\
	&\Rightarrow \int_{0}^{a}(A(x)-\tilde{A}(x))e^{-2x\sqrt{\kappa_m}}dx = \tilde{O}(e^{-2a\sqrt{\kappa_m}}).
	\end{aligned}
	\end{equation*}

	\noindent Let $\varepsilon>0$ and set $\displaystyle F(z)=e^{2a(1-\varepsilon)z}\int_{0}^{a}(A(x)-\tilde{A}(x))e^{-2xz}dx$. The function $F$ is entire and satisfies
	\begin{equation*}
	\forall z\in\mathbb{C},\quad \mathrm{Re}(z)> 0 \Rightarrow |F(z)|\le \|A-\tilde{A}\|_{1}e^{2a(1-\varepsilon)\mathrm{Re}(z)}
	\end{equation*}

	\noindent Let $m$ be an integer large enough. From Lemma \ref{un_sur_deux_local_2}, we can find an integer $p$ in $\{2m,2m+1\}$ such that  $|\lambda^-(\kappa_p)-\tilde{\lambda}^-(\kappa_p)|\le \varepsilon_p$. We can thus build a sequence $(u_m)$ by setting, for each $m$ large enough, $\displaystyle u_m=\frac{\sqrt{\kappa_p}}{2}$. This sequence satisfies
	\begin{equation*}
	u_m-m= O(1).
	\end{equation*}
	
	\noindent We set at last $G(z)=F(2z)$. Then $|G(z)|\le \|A-\tilde{A}\|_{1}e^{4a(1-\varepsilon)\mathrm{Re}(z)}$ and moreover :
	\begin{equation*}
	G(u_m) = o(1)
	\end{equation*}
	
	\noindent Consequently (cf \cite{boas2011entire}, Theorem 10.5.1, p.191), $G$ is bounded on $\mathbb{R}_+$, and so is $F$ :
	\begin{equation*}
	\forall u\in\mathbb{R}_+,\quad\int_{0}^{a}(A(x)-\tilde{A}(x))e^{-2xu}dx=O\big(e^{-2a(1-\varepsilon)u}\big)
	\end{equation*}
	
	\noindent As this estimate is true for all $\varepsilon>0$, we have :
	\begin{equation*}
	\forall u\in\mathbb{R}_+,\quad \int_{0}^{a}(A(x)-\tilde{A}(x))e^{-2xu}dx=\tilde{O}\big(e^{-2au}\big),
	\end{equation*}
	
	\noindent hence (\cite{simon1999new}, Lemma A.2.1) $A=\tilde{A}$ on $[0,a]$. One deduces :
	
	\begin{equation*}
	\forall t\in\mathbb{R},\quad N(t^2)-\tilde{N}(t^2)=\tilde{O}(e^{-2at})
	\end{equation*}
	
	\medskip
	
	\noindent From Remark \ref{WT_symmetry}, $N$ (resp. $\tilde{N}$) has the same role as $M$ (resp. $\tilde{M}$) for the potential $x\mapsto q(1-x)$ (resp. $x\mapsto \tilde{q}(1-x)$). Now, it follows, from \cite{simon1999new}, Theorem A.1.1, that we get $q(1-x)=\tilde{q}(1-x)$ for all $x\in[0,a]$, \emph{i.e}
	\begin{equation*}
	\frac{(f^{n-2})''(x)}{f^{n-2}(x)}-\omega f(x) = \frac{(\tilde{f}^{n-2})''(x)}{\tilde{f}^{n-2}(x)}-\omega \tilde{f}(x):=r(x),\quad \forall x\in [1-a,1].
	\end{equation*}
	
	\noindent The functions $f$ and $\tilde{f}$ solve on $[1-a,1]$ the same ODE
	\begin{equation*}
	(y^{n-2})''(x)-\lambda y^{n-1}(x) = r(x) y^{n-2}(x)
	\end{equation*}
	
	 \noindent Moreover $f(1)=\tilde{f}(1)$ and, thanks to the equality
	\begin{center}
		$\displaystyle \frac{(\ln h)'(1)}{4\sqrt{f(1)}}=\frac{(\ln\tilde{h})'(1)}{4\sqrt{\tilde{f}(1)}}$,
	\end{center} we also have $f'(1)=\tilde{f}'(1)$. Hence, the Cauchy-Lipschitz Theorem entails that $f=\tilde{f}$ on $[1-a,1]$.
	
	$\quad$
	
	\begin{remark}
	If we had assumed that $f(0)>f(1)$, we would have worked with $\big(\lambda^+(\kappa_p)\big)$ and $\big(\tilde{\lambda}^+(\kappa_p)\big)$, and found that $f=\tilde{f}$ on $[0,a]$.
	\end{remark}

	\subsection{The case ${\bf f(0)= f(1)}$}
	
	\medskip
	
	\noindent Assume, without loss of generality, that $f(0)=f(1)=1$. From Lemma \ref{asymp_vp_2}, the eigenvalues $\lambda^\pm(\kappa_m)$ satisfy the asymptotics :
	\begin{equation*}
	\left\{
	\begin{aligned}
	&\lambda^-(\kappa_m)=-N(\kappa_m)-\frac{(\ln h)'(1)}{4}+\tilde{O}\bigg(e^{-2\sqrt{\kappa_m}}\bigg)	 \\
	&\lambda^+(\kappa_m)=-M(\kappa_m)+\frac{(\ln h)'(0)}{4}++\tilde{O}\bigg(e^{-2\sqrt{\kappa_m}}\bigg).
	\end{aligned}\right.
	\end{equation*}
	
	\medskip
	
	\noindent Let us denote $\displaystyle C_0=\frac{(\ln h)'(0)}{4}$, $\displaystyle C_1=\frac{(\ln h)'(1)}{4}$, $\displaystyle \tilde{C}_0=\frac{(\ln\tilde{h})'(0)}{4}$ and $\displaystyle \tilde{C}_1=\frac{(\ln\tilde{h})'(1)}{4}$. 
	
	\medskip
	
	\noindent Using the asymptotics of $M$ and $N$ given in Theorem \ref{Simon_2} and Corollary \ref{CorSimon_2}, and the explicit expression of $\kappa_m=m(m+n-2)$, we get
	\begin{equation}
	\label{asymptvp_2}
	\left\{
	\begin{aligned}
	&\lambda^-(\kappa_m)=m+\frac{n-2}{2}-C_1+O\bigg(\frac{1}{m}\bigg)	 \\
	&\lambda^+(\kappa_m)=m+\frac{n-2}{2}+C_0+O\bigg(\frac{1}{m}\bigg).
	\end{aligned}\right.\qquad\quad m\to +\infty.
	\end{equation}
	
	\noindent Let us set also
	
	\begin{center}
	$V_m=\displaystyle \bigg\{\lambda^-(\kappa_m)-\frac{n-2}{2}\:,\:\lambda^+(\kappa_m)-\frac{n-2}{2}\bigg\}\:\:$ and $\:\:\tilde{V}_m=\displaystyle \bigg\{\tilde{\lambda}^-(\kappa_m)-\frac{n-2}{2}\:,\:\tilde{\lambda}^+(\kappa_m)-\frac{n-2}{2}\bigg\}$.
	\end{center} 

\noindent As $f$ and $\tilde{f}$ belong to $\mathcal{C}_b$, one has
\begin{equation}
	\label{girafe_2}
|C_i|\le\frac{1}{4},\quad |\tilde{C}_i|\le\frac{1}{4}	,\quad i\in\{0,1\}.
\end{equation} Hence, thanks to (\ref{asymptvp_2}) and (\ref{girafe_2}), we get for $m$ large enough:
	\begin{equation}
	\label{Sets_2}
	V_m,\tilde{V}_m\subset \bigg[m-\frac{1}{3},m+\frac{1}{3}\bigg]
	\end{equation}
	
	\noindent Of course, the assumption $\Sigma\big(\Lambda_g(\omega)\big)\underset{(\varepsilon_m)}{\asymp}\Sigma\big(\Lambda_{\tilde{g}}(\omega)\big)$ implies 
	\begin{equation}
	\label{lapin_2}
	\displaystyle -\frac{n-2}{2}+\Sigma\big(\Lambda_g(\omega)\big)\underset{(\varepsilon_m)}{\asymp}-\frac{n-2}{2}+\Sigma\big(\Lambda_{\tilde{g}}(\omega)\big)
	\end{equation}
	
	\noindent From (\ref{Sets_2}) and (\ref{lapin_2}), for each $m$ large enough, we have the alternative
	\begin{equation}
	\label{dindon_2}
	\left\{\begin{aligned}
	&\lambda^-(\kappa_m)-\tilde{\lambda}^-(\kappa_m) =O\big(e^{-2a\sqrt{\kappa_m}}\big)\\
	&\lambda^+(\kappa_m)-\tilde{\lambda}^+(\kappa_m) =O\big(e^{-2a\sqrt{\kappa_m}}\big)\end{aligned}\right.\quad\mathrm{or}\quad \left\{\begin{aligned}
	&\lambda^-(\kappa_m)-\tilde{\lambda}^+(\kappa_m) =O\big(e^{-2a\sqrt{\kappa_m}}\big)\\
	&\lambda^+(\kappa_m)-\tilde{\lambda}^-(\kappa_m) =O\big(e^{-2a\sqrt{\kappa_m}}\big)\end{aligned}\right.
	\end{equation} 
	
	\medskip
	
\noindent There is thus an infinite set $\mathcal{S}\subset\mathbb{N}$ such that either
	\begin{equation}
	\label{poule_2}
	\begin{aligned}
	\forall m\in\mathcal{S},\:\:\left\{\begin{aligned}
	&\lambda^-(\kappa_m)-\tilde{\lambda}^-(\kappa_m) =O\big(e^{-2a\sqrt{\kappa_m}}\big)\\
	&\lambda^+(\kappa_m)-\tilde{\lambda}^+(\kappa_m) =O\big(e^{-2a\sqrt{\kappa_m}}\big)\end{aligned}\right.&\qquad\mathrm{or}\\ \forall m\in\mathcal{S},&\:\:\left\{\begin{aligned}
	&\lambda^-(\kappa_m)-\tilde{\lambda}^+(\kappa_m) =O\big(e^{-2a\sqrt{\kappa_m}}\big)\\
	&\lambda^+(\kappa_m)-\tilde{\lambda}^-(\kappa_m) =O\big(e^{-2a\sqrt{\kappa_m}}\big).\end{aligned}\right.
	\end{aligned}
	\end{equation}
	
	\medskip
	
	\noindent Assume, for example, that the former is true. Then we have, using (\ref{asymptvp_2}) :
	\begin{equation}
	\label{chat_2}
	C_1=\tilde{C}_1\quad\mathrm{and}\quad C_0=\tilde{C}_0.
	\end{equation}
	
\medskip

\noindent \textit{Case 1} : $C_0\ne -C_1$.

\medskip

	\noindent Let us denote 
	\begin{center}
		$\displaystyle \delta=\frac{|C_0+C_1|}{3}\in \big]0,\frac{1}{4}\big[$.
	\end{center}
	
	$\quad$
	
	\noindent For $m$ large enough, we have, thanks to (\ref{asymptvp_2}):
	
	\begin{enumerate}
		\item[$\bullet$] $\displaystyle \lambda^-(\kappa_m)-\frac{n-2}{2}$ and $\displaystyle \tilde{\lambda}^-(\kappa_m)-\frac{n-2}{2}$ both belong to the interval $-C_1+[m-\delta,m+\delta]:=I_{m,1}$
		\item[$\bullet$] $\displaystyle \lambda^+(\kappa_m)-\frac{n-2}{2}$ and $\displaystyle \tilde{\lambda}^+(\kappa_m)-\frac{n-2}{2}$ both belong to the interval $C_0+[m-\delta,m+\delta]:=I_{m,0}$.
	\end{enumerate}

\medskip
	
	\noindent Moreover, as $C_0\ne -C_1$, we have $d(I_{m,1} , I_{m,0})\ge \delta$ for all $m$ large enough, where $\displaystyle d(I,J)=\underset{x\in I,y\in J}{\inf}|x-y|$. We can therefore associate eigenvalues as follows :
	
	\begin{equation*}
	\left\{\begin{aligned}
	&\lambda^-(\kappa_m) = \tilde{\lambda}^-(\kappa_m) +O(e^{-2a\sqrt{\kappa_m}})\\
	&\lambda^+(\kappa_m) = \tilde{\lambda}^+(\kappa_m) +O(e^{-2a\sqrt{\kappa_m}})
	\end{aligned}\right.\qquad\qquad m\to+\infty.
	\end{equation*}
	
	\medskip
	
	\noindent One shows, as in Section 3.1, that 
	\begin{equation*}
	\left\{\begin{aligned}
	&N(t^2)-\tilde{N}(t^2)=\tilde{O}(e^{-2at})\\
	&M(t^2)-\tilde{M}(t^2)=\tilde{O}(e^{-2at})
	\end{aligned}\right.\qquad t\to+\infty
	\end{equation*} and then, that 
	\begin{equation*}
	f(x)=\tilde{f}(x)\quad \forall x\in [1-a,1]\quad \mathrm{and}\quad f(x)=\tilde{f}(x)\quad \forall x\in [0,a].
	\end{equation*}
	
	$\quad$
	
	\noindent \textit{Case 2} : $C_0= -C_1$.
	
	$\quad$

		\medskip
		
		\noindent By hypothesis : $f,\tilde{f}$ belong to $\mathcal{C}_b$ so $q$ belongs to $\mathcal{D}_b$. Thanks to Lemma \ref{equiv02}, there is $j_0\in\mathbb{N}$ such that $\beta_j(0)\ne \gamma_j(0)$. Let us set 
		\begin{center}
		$j_0=\min\{j\ge 2,\: \beta_{j}(0)\ne \gamma_{j}(0) \}$.
		\end{center} 
	
	\medskip
	
	\noindent The asymptotics given by Theorem \ref{Simon_2} and Corollary (\ref{CorSimon_2}) imply 
	
	\medskip
	
	\begin{center}
		$\displaystyle \lambda^-(\kappa_m)-\lambda^+(\kappa_m)= \frac{\gamma_{j_0}-\beta_{j_0}}{m^{j_0}}+O\bigg(\frac{1}{m^{j_0+1}}\bigg)$.
	\end{center}
	
	$\quad$
	
	\noindent Because of the relation $\Sigma(\Lambda_g(\omega)) \underset{(\varepsilon_m)}{\asymp} \Sigma(\Lambda_{\tilde{g}}(\omega))$, one can show that, for the same $j_0$:
	\begin{center}
		$\displaystyle \tilde{\lambda}^-(\kappa_m)-\tilde{\lambda}^+(\kappa_m)= \frac{\tilde{\gamma}_{j_0}-\tilde{\beta}_{j_0}}{m^{j_0}}+O\bigg(\frac{1}{m^{j_0+1}}\bigg)$
	\end{center}
	
	$\quad$
	
	\noindent It is then possible to order the eigenvalues $\lambda^-(\kappa_m)$ and $\lambda^+(\kappa_m)$ (also $\tilde{\lambda}^-$ and $\tilde{\lambda}^+$ ), and this order depends on the sign of $\gamma_{j_0}-\beta_{j_0}$ (resp. $\tilde{\gamma}_{j_0}-\tilde{\beta}_{j_0}$). If $\gamma_{j_0}-\beta_{j_0}$ and $\tilde{\gamma}_{j_0}-\tilde{\beta}_{j_0}$ have the same sign, we claim that
	
	\begin{equation}
	\label{lala_2}
	\left\{ \begin{aligned}
	&\lambda^-(\kappa_m)=\tilde{\lambda}^-(\kappa_m)+\tilde{O}(e^{-2a\sqrt{\kappa_m}})\\
	&\lambda^+(\kappa_m)=\tilde{\lambda}^+(\kappa_m)+\tilde{O}(e^{-2a\sqrt{\kappa_m}}).
	\end{aligned} \right.
	\end{equation}
	
	$\quad$
	
	\noindent Indeed, if not, from (\ref{dindon_2}), there is an infinite subset $\mathcal{F}\subset \mathbb{N}$ such that : 
	
	\medskip
	
	\begin{center}
	$\displaystyle \left\{\begin{aligned}
		&\lambda^-(\kappa_m)=\tilde{\lambda}^+(\kappa_m)+\tilde{O}(e^{-2a\sqrt{\kappa_m}})\\
		&\lambda^+(\kappa_m)=\tilde{\lambda}^-(\kappa_m)+\tilde{O}(e^{-2a\sqrt{\kappa_m}}).
		\end{aligned} \right.,\qquad\quad m\to+\infty,\:\: m\in \mathcal{F}$.
	\end{center}
	
	\medskip
	
	\noindent Then
	$\lambda^-(\kappa_m)-\lambda^+(\kappa_m)=\tilde{\lambda}^+(\kappa_m)-\tilde{\lambda}^-(\kappa_m)+O(e^{-2a\sqrt{\kappa_m}})$, and letting $m$ go to infinity :
	\begin{equation*}
	\gamma_{j_0}-\beta_{j_0}= \tilde{\beta}_{j_0}-\tilde{\gamma}_{j_0}
	\end{equation*}
	
	\noindent and we have a contradiction.
	\noindent Using (\ref{lala_2}) and the same method as in Section 3.1, we find 
	
	\medskip
	
	\begin{center}
		$\displaystyle\forall t\in\mathbb{R}_+,\quad M(t^2)-\tilde{M}(t^2)=\tilde{O}(e^{-2at})\quad$ and $\quad\displaystyle N(t^2)-\tilde{N}(t^2)=\tilde{O}(e^{-2at})$
	\end{center}
	
	\noindent and at last \begin{center}
		$f=\tilde{f}$ on $[0,a]\quad$ and $\quad f=\tilde{f}$ on $[1-a,1]$.
	\end{center} 

$\quad$
	
	\noindent If $\gamma_{j_0}-\beta_{j_0}$ and $\tilde{\gamma}_{j_0}-\tilde{\beta}_{j_0}$ have opposite sign, then :
	
	\begin{equation*}
	\left\{ \begin{aligned}
	&\lambda^-(\kappa_m)=\tilde{\lambda}^+(\kappa_m)+O(e^{-2a\sqrt{\kappa_m}})\\
	&\lambda^+(\kappa_m)=\tilde{\lambda}^-(\kappa_m)+O(e^{-2a\sqrt{\kappa_m}}).
	\end{aligned} \right.
	\end{equation*}
	
\noindent In this case, one can prove that
\begin{center}
	$f=\tilde{f}\circ \eta$ on $[0,a]\quad$ and $\quad f=\tilde{f}\circ\eta$ on $[1-a,1]$.
\end{center} 

	\subsection{Special case} 
	
	\noindent When $f(0)=f(1)$, the direct implication has already been established. Now, we prove the converse in this case. Let $a\in ]0,1[$ and assume that, for example:\begin{equation*}
	\begin{aligned}
	f=\tilde{f}&\:\:\mathrm{on}\:\: [0,a]   \quad {\rm{and}}\quad  f=\tilde{f}\:\:\mathrm{on}\:\: [1-a,1] 
	\end{aligned}
	\end{equation*}
	
	$\quad$
	
	\noindent In that case, $q=\tilde{q}$ on $[0,a]$ and $q\circ\eta=\tilde{q}\circ\eta$ on $[0,a]$. But thanks to Theorem 3.1 in \cite{simon1999new}, the potential $q$ determines the function $A$ that appears in the representation (\ref{Simon_repr_2}). Hence, 
	\begin{equation*}
	\left\{\begin{aligned}
	&M(z^2)-\tilde{M}(z^2)=\tilde{O}(e^{-2az})\\
	&N(z^2)-\tilde{N}(z^2)=\tilde{O}(e^{-2az})
	\end{aligned}\right.
	\end{equation*}
	\noindent The hypothesis $(P_1)$ implies in particular that $f(0)=\tilde{f}(0)$ and $f'(0)=\tilde{f}'(0)$. From $(P_3)$ we have also $f(1)=\tilde{f}(1)$ and $f'(1)=\tilde{f}'(1)$. Using the asymptotics given by Lemma \ref{asymp_vp_2}, one deduces immediately that 
	\begin{equation*}
	\left\{\begin{aligned}
	&\lambda^+(\kappa_m)-\tilde{\lambda}^+(\kappa_m)=\tilde{O}(e^{-2a\sqrt{\kappa_m}})\\
	&\lambda^-(\kappa_m)-\tilde{\lambda}^-(\kappa_m)=\tilde{O}(e^{-2a\sqrt{\kappa_m}})
	\end{aligned}\right.\qquad m\to+\infty
	\end{equation*}
	\noindent which concludes the proof.
\end{proof}

\begin{remark}
	We emphasize that if $f(0)=f(1)$ and $\displaystyle \frac{1}{2}\le a<1$, then we have a global uniqueness result.
\end{remark}

$\quad$

\begin{corollary}
	If $f$ and $\tilde{f}$ are analytic functions on $[0,1]$ the previous local uniqueness result becomes a global uniqueness result without the additional constraint that $q,\tilde{q}\in\mathcal{D}_b$.
\end{corollary}

\begin{proof}
	
	\noindent In proving Theorem \ref{pseudostab_2}, we needed the hypothesis $q,\tilde{q}\in\mathcal{D}_b$ only in the case where $f(0)=f(1)$ \ and \ $C_0=-C_1$. In all other cases, without this hypothesis, one of the properties $(P_1)$, $(P_2)$, $(P_3)$ or $(P_4)$  was obtained and, as $f$ and $\tilde{f}$ are assumed to be analytic, the corresponding equalities extend over $[0,1]$ by analytic continuation. Then, only this latter case remains to be dealt with. Let us assume, without loss of generality, that $f(0)=1$.
	
	\medskip
	
	\noindent {\bf Subcase 1} : $q,\tilde{q}\in\mathcal{D}_b$ and the situation has already been studied.  
	
	\medskip
	
	\noindent {\bf Subcase 2} : $q$ or $\tilde{q}$ does not belong to $\mathcal{D}_b$. Assume, for example, that $q\notin \mathcal{D}_b$. As $f$ is analytic, so is $q$. The function $\varphi:[0,1]\mapsto \mathbb{R},\:\: x\mapsto q(x)-q(1-x)$ is also analytic and, as $q$ is not in $\mathcal{D}_b$, satisfies :
	\begin{equation*}
	\forall k\in\mathbb{N},\quad \varphi^{(k)}(0)=0.
	\end{equation*}
	
	\noindent By analytic continuation, $\varphi$ vanishes everywhere on $[0,1]$, {\it i.e} $q=q\circ\eta$ on $[0,1]$. This is equivalent to $M=N$. The eigenvalues $\lambda^\pm(\kappa_m)$ then satisfy the asymptotics
	\begin{equation*}
	\left\{
	\begin{aligned}
	&\lambda^-(\kappa_m)=M(\kappa_m)+C_0+\tilde{O}\bigg(e^{-2a\sqrt{\kappa_m}}\bigg)	 \\
	&\lambda^+(\kappa_m)=M(\kappa_m)+C_0+\tilde{O}\bigg(e^{-2a\sqrt{\kappa_m}}\bigg)
	\end{aligned}\right.
	\end{equation*}
	
	\noindent Using the assumption $\sigma(\Lambda_{g}(\omega))\underset{(\varepsilon_m)}{\asymp}\sigma(\Lambda_{\tilde{g}}(\omega))$ and the same arguments as above, we prove that $C_0=\tilde{C_0}=-\tilde{C}_1$ and
	\begin{equation*}
	\left\{\begin{aligned}
	&M(\kappa_m) = \tilde{M}(\kappa_m)+ O\big(e^{-2a\sqrt{\kappa_m}}\big)\\
	&M(\kappa_m) = \tilde{N}(\kappa_m)+ O\big(e^{-2a\sqrt{\kappa_m}}\big) \end{aligned}\right.\qquad m\to+\infty
	\end{equation*}
	
	\medskip
	
	\noindent One can show, as in the proof of Theorem \ref{pseudostab_2}, that :
	
	\begin{equation*}
	M(t^2)-\tilde{M}(t^2) = O(e^{-2at})\quad\mathrm{and}\quad M(t^2)-\tilde{N}(t^2) = O(e^{-2at}).
	\end{equation*} 
	
	\medskip
	
	\noindent Hence $f=\tilde{f}\:\: \mathrm{on}\:\: [0,a]$ and $f=\tilde{f}\circ\eta\:\: \mathrm{on}\:\: [0,a]$. By analytic continuation :
	\begin{equation*}
	f=\tilde{f}=\tilde{f}\circ \eta\:\: \mathrm{sur}\:\: [0,1].
	\end{equation*} 	
\end{proof}	

\noindent The next section is devoted to the proof of Theorem \ref{stabresult_2}.

\section{Stability estimates for symmetric conformal factors}
\subsection{Discrete estimates on Weyl-Titchmarsh functions}

\medskip

\noindent \textit{Preliminary remarks}: 

\medskip

\begin{enumerate}
	\item Until the end of the paper, we will denote by $C_A$ any constant depending only on $A$, even within the same calculation.
	\item In this section, each factor $f$ and $\tilde{f}$ is supposed so be symmetric. This simplifies many formula. However, in order to generalize our arguments as much as possible, we will use this property of symmetry only when it seems necessary and write the formulas in their generic forms. For example, we will distinguish $M$ from $N$ whereas those two functions are equal.
\end{enumerate}

$\quad$

\noindent The goal of this subsection is to prove the following result which will be useful in Subsection 4.2.

\medskip

\begin{proposition}
	\label{WT_estimates_2}
	Let $\varepsilon>0$ small enough. Assume that $\sigma(\Lambda_g(\omega))\underset{\varepsilon}{\asymp}\sigma(\Lambda_{\tilde{g}}(\omega))$. There is $C_A>0$ and $m_0\in\mathbb{N}$ (independant of $\varepsilon$) such that, for all $\displaystyle m\ge m_0$ and by setting $y_m=\sqrt{\kappa_m}$ :
	\begin{equation*}
\bigg|\bigg(M(\kappa_m)N(\kappa_m)-\frac{1}{\Delta^2(\kappa_m)}\bigg)-\bigg(\tilde{M}(\kappa_m)\tilde{N}(\kappa_m)-\frac{1}{\tilde{\Delta}^2(\kappa_m)}\bigg)\bigg| \le C_A \varepsilon\times y_m
	\end{equation*}
\end{proposition}

\begin{proof}
 We first need the following result.
	
	\begin{lemma}
		\label{coeff_2}
		Under the hypothesis $\sigma(\Lambda_{g}(\omega))\underset{\varepsilon}{\asymp}\sigma(\Lambda_{\tilde{g}}(\omega))$, and under the hypothesis that $f$ and $\tilde{f}$ are symmetric, we have :
		\begin{equation*}
		f(0)=\tilde{f}(0).
		\end{equation*}	
	\end{lemma}	

\medskip
	
\begin{proof} Using the same argument as in the proof of Lemma \ref{alt_2_1}, one proves the equality 
	\begin{equation*}
		\sqrt{f(1)} + \sqrt{f(0)} = \sqrt{\tilde{f}(1)}+\sqrt{\tilde{f}(0)}.
	\end{equation*}
	\noindent As $f$ and $\tilde{f}$ are symmetric with respect to $1/2$, we have $f(0)=f(1)$ and $\tilde{f}(0)=\tilde{f}(1)$. Hence $2f(0)=2\tilde{f}(0)$. This proves Lemma \ref{coeff_2}.
\end{proof}
	
	\begin{lemma}
		\label{prox_induite}
		For $m$ large enough, we have : \begin{equation}
		\{ \lambda^-(\kappa_m),\lambda^+(\kappa_m) \}\underset{\varepsilon}{\asymp} \{ \tilde{\lambda}^-(\kappa_m),\tilde{\lambda}^+(\kappa_m) \}
		\end{equation}
	\end{lemma}
	
	\begin{proof}
	\noindent For every $m\in\mathbb{N}$, there $p$ such that :
	
	\begin{equation}
	\label{plusmoins_2}
		|\lambda^{\pm}(\kappa_m)- \tilde{\lambda}^{\pm}(\kappa_{p}|\le \varepsilon
	\end{equation}
	
	\noindent Let us denote	
	\begin{equation*}
	C_1=\frac{1}{4\sqrt{f(1)}}\frac{h'(1)}{h(1)}\quad\mathrm{et}\quad C_0=\frac{1}{4\sqrt{f(0)}}\frac{h'(0)}{h(0)}.
	\end{equation*}	
	
	\noindent Since $f$ and $\tilde{f}$ are supposed symmetric, we have
	\begin{equation*}
	C_0=-C_1\quad\mathrm{and}\quad \tilde{C}_0=-\tilde{C}_1
	\end{equation*}
	
	$\quad$
	
	\noindent Thus, setting
	\begin{equation*}
	C=C_0-\tilde{C}_0
	\end{equation*}
\noindent one has, from Lemma \ref{asymp_vp_2} :
	
	\begin{equation*}
	\sqrt{f(0)}\bigg(\lambda^{\pm}(\kappa_m)- \tilde{\lambda}^{\pm}(\kappa_{p})\bigg) = \big(m-p\big) + C + o(1).
	\end{equation*}
	
	\noindent Let $k=\lfloor C \rfloor$. Then
	\begin{equation}
	\label{decalage_2}
	m-p+k= \sqrt{f(0)}\bigg(\lambda^{\pm}(\kappa_m)- \tilde{\lambda}^{\pm}(\kappa_{p})\bigg) + \underbrace{k- C}_{\in ]-1,0[}+o(1)
	\end{equation}
	and, as $m-p+k$ is an integer, using (\ref{plusmoins_2}), this leads, for $m$ large enough and $\varepsilon$ small enough, to
	\begin{equation*}
		p=m+k.
	\end{equation*}
	
	\noindent  Hence, for $m$ large enough, (\ref{plusmoins_2}) is equivalent to
	
\begin{equation*}
\left\{\begin{aligned}
&|\lambda^-(\kappa_m)-\tilde{\lambda}^-(\kappa_{m+k})|\le \varepsilon\\
&|\lambda^+(\kappa_m)-\tilde{\lambda}^+(\kappa_{m+k})|\le \varepsilon
\end{aligned}\right.\qquad \mathrm{or}\qquad \left\{\begin{aligned}
&|\lambda^-(\kappa_m)-\tilde{\lambda}^+(\kappa_{m+k})|\le \varepsilon\\
&|\lambda^+(\kappa_m)-\tilde{\lambda}^-(\kappa_{m+k})|\le \varepsilon
\end{aligned}\right.
\end{equation*}
	
	\medskip
	
	\noindent The relation $\Sigma\big(\Lambda_g(\omega)\big)\underset{\varepsilon}{\asymp}\Sigma\big(\Lambda_{\tilde{g}}(\omega)\big)$ implies 
	\begin{equation*}
		2m = 2(m+k)
	\end{equation*}
	
	\noindent and then $k=0$.
	
	$\quad$
	
	\noindent This means that, for $m$ greater than some $m_0$ (that does not depend on $\varepsilon$) one has
\begin{equation}
\{ \lambda^-(\kappa_m),\lambda^+(\kappa_m) \}\underset{\varepsilon}{\asymp} \{ \tilde{\lambda}^-(\kappa_m),\tilde{\lambda}^+(\kappa_m) \}
\end{equation}
		\end{proof}
	
	\noindent Of course, Lemma \ref{prox_induite} is still true by replacing $\kappa_m$ by $\mu_m$. For $m$ large enough, we have :
	\begin{equation}
	\label{proximite_induite_2}
	\{ \lambda^-(\mu_m),\lambda^+(\mu_m) \}\underset{\varepsilon}{\asymp} \{ \tilde{\lambda}^-(\mu_m),\tilde{\lambda}^+(\mu_m) \}
	\end{equation}
	
	\noindent Recall that
	\begin{equation*}
	\Lambda_g^m(\omega)=\begin{pmatrix}
	-\frac{M(\mu_m)}{\sqrt{f(0)}}+C_0&-\frac{1}{\sqrt{f(0)}}\frac{h^{1/4}(1)}{h^{1/4}(0)}\frac{1}{\Delta(\mu_m)}\\
	-\frac{1}{\sqrt{f(1)}}\frac{h^{1/4}(0)}{h^{1/4}(1)}\frac{1}{\Delta(\mu_m)}&-\frac{N(\mu_m)}{\sqrt{f(1)}}-C_1\\
	\end{pmatrix}
	\end{equation*}
	
	\noindent Hence
	\begin{equation}
	\label{Corona_2}
\begin{aligned}
\mathrm{Tr}\big(\Lambda_g^m(\omega)\big)- \mathrm{Tr}&\big(\Lambda_{\tilde{g}}^m(\omega)\big)\\ &= \bigg[-\frac{M(\mu_m)}{\sqrt{f(0)}}+C_0 -\frac{N(\mu_m)}{\sqrt{f(1)}}-C_1\bigg] - \bigg[-\frac{\tilde{M}(\mu_m)}{\sqrt{\tilde{f}(0)}}+\tilde{C}_0 -\frac{\tilde{N}(\mu_m)}{\sqrt{\tilde{f}(1)}}-\tilde{C}_1\bigg]\\
&= -\frac{1}{\sqrt{f(0)}}\bigg(M(\mu_m)-\tilde{M}(\mu_m)\bigg) -\frac{1}{\sqrt{f(1)}}\bigg(N(\mu_m)-\tilde{N}(\mu_m)\bigg)\\
&\hspace{6cm}+ (\tilde{C}_0-C_0) + (C_1-\tilde{C}_1)\\
\end{aligned}
	\end{equation}
	
\noindent Thanks to (\ref{decalage_2}), with $k=0$ and $m-p=0$, we have
\begin{equation}
\label{hypocondriaque_2}
	|C|=|\tilde{C}_0-C_0|=|C_1-\tilde{C}_1| \le C_A\varepsilon.
\end{equation}
	
\noindent Hence, combining (\ref{proximite_induite_2}), (\ref{Corona_2}) and (\ref{hypocondriaque_2}), we get :

\begin{equation*}
\begin{aligned}
\bigg|\frac{1}{\sqrt{f(0)}}\bigg(M(\mu_m)-\tilde{M}(\mu_m)\bigg) +\frac{1}{\sqrt{f(1)}}\bigg(N(\mu_m)-\tilde{N}(\mu_m)\bigg)\bigg| &\le  \underbrace{\big|\mathrm{Tr}\big(\Lambda_g^m(\omega)\big)- \mathrm{Tr}\big(\Lambda_{\tilde{g}}^m(\omega)\big)\big|}_{\le 2\varepsilon} + C_A\varepsilon\\
&\le C_A\varepsilon.
\end{aligned}
\end{equation*}

\noindent As $f$ and $\tilde{f}$ are symmetric with respect to $1/2$, this leads to
\begin{equation*}
	\bigg|M(\mu_m)-\tilde{M}(\mu_m)\bigg|\le C_A \varepsilon
\end{equation*}

\noindent We also have an estimate on the determinant. From (\ref{proximite_induite_2}), assume for example that
\begin{equation*}
	\big|\lambda^+(\mu_m)-\tilde{\lambda}^+(\mu_m)\big|\le \varepsilon\quad\mathrm{and}\quad \big|\tilde{\lambda}^-(\mu_m)-\lambda^-(\mu_m)\big|\le \varepsilon.
\end{equation*} 

\noindent Then :
\begin{equation*}
\begin{aligned}
\bigg|\det(\Lambda_g^m(\omega))-\det(\Lambda_{\tilde{g}}^m(\omega)) \bigg|&= \bigg|\lambda^-(\mu_m)\lambda^+(\mu_m)-\tilde{\lambda}^-(\mu_m)\tilde{\lambda}^+(\mu_m)\bigg|\\
&\le \big|\lambda^-(\mu_m)\big|\big|\lambda^+(\mu_m)-\tilde{\lambda}^+(\mu_m)\big|+\big|\tilde{\lambda}^-(\mu_m)-\lambda^-(\mu_m)\big|\big|\tilde{\lambda}^+(\mu_m)\big|\\
&\le C_A \varepsilon \times \sqrt{\mu_m}
\end{aligned}
\end{equation*}

\noindent We write :

\begin{equation*}
\det(\Lambda_g^m(\lambda))-\det(\Lambda_{\tilde{g}}^m(\lambda))= \mathrm{I}(\mu_m)+\mathrm{II}(\mu_m)+\mathrm{III}(\mu_m)+\mathrm{IV}
\end{equation*}

\noindent with

\begin{equation*}
\mathrm{I}(\mu_m)=\frac{1}{\sqrt{f(0)f(1)}}\bigg[\bigg(M(\mu_m)N(\mu_m)-\frac{1}{\Delta^2(\mu_m)}\bigg)-\bigg(\tilde{M}(\mu_m)\tilde{N}(\mu_m)-\frac{1}{\tilde{\Delta}^2(\mu_m)}\bigg)\bigg],
\end{equation*}

\begin{equation*}
\mathrm{II}(\mu_m)=\frac{1}{\sqrt{f(0)}}\bigg[C_1(M(\mu_m)-\tilde{M}(\mu_m))+(C_1-\tilde{C_1})\tilde{M}(\mu_m)\bigg]
\end{equation*}

\begin{equation*}
\mathrm{III}(\mu_m)=\frac{1}{\sqrt{f(1)}}\bigg[\tilde{C}_0(\tilde{N}(\mu_m)-N(\mu_m))+(\tilde{C}_0-C_0)N(\mu_m)\bigg]
\end{equation*}
and
\begin{equation*}
\mathrm{IV} = (\tilde{C}_0-C_0)\tilde{C}_1 + C_0(\tilde{C}_1-C_1)
\end{equation*}

\noindent We have:

\begin{equation*}
\begin{aligned}
|\mathrm{II}(\mu_m)| &\le \frac{1}{\sqrt{f(0)}}|C_1||M(\mu_m)-\tilde{M}(\mu_m)| + \frac{1}{\sqrt{f(0)}}|C_1-\tilde{C}_1||\tilde{M}(\mu_m)|\\
&\le C_A \varepsilon + C_A \varepsilon \sqrt{\mu_m}\\
&\le C_A \varepsilon \sqrt{\mu_m}.
\end{aligned}
\end{equation*}

\noindent Similarly:
\begin{equation*}
	|\mathrm{II}(\mu_m)|\le C_A \varepsilon \sqrt{\mu_m}\quad\mathrm{and}\quad |\mathrm{IV}|\le C_A\varepsilon.
\end{equation*}

\noindent Finally :
\begin{equation*}
\begin{aligned}
|I(\mu_m)|&\le \big|\det(\Lambda_g^m(\lambda))-\det(\Lambda_{\tilde{g}}^m(\lambda)) \big| + C_A\varepsilon\sqrt{\mu_m}\\
&\le C_A\varepsilon\sqrt{\mu_m}
\end{aligned}
\end{equation*}

\medskip

\noindent As this is true for $\mu_m$, with $m\ge m_0$ with $m_0$ not depending on $\varepsilon$, this is also true for $\kappa_m$ with $m\ge m_0$ (with $m_0$ different from the other one but still independent of $\varepsilon$). Hence, by setting $y_m=\sqrt{\kappa_m}$, we have proved that there exists $C_A>0$ such that, for $m\ge m_0$ :
\begin{equation*}
\bigg|\bigg(M(\kappa_m)N(\kappa_m)-\frac{1}{\Delta^2(\kappa_m)}\bigg)-\bigg(\tilde{M}(\kappa_m)\tilde{N}(\kappa_m)-\frac{1}{\tilde{\Delta}^2(\kappa_m)}\bigg)\bigg| \le C_A \varepsilon\times y_m
\end{equation*}
\end{proof}

\subsection{An integral estimate}

\medskip

\noindent In all this section, we will use the estimate of Proposition \ref{WT_estimates_2} in order to show that
\begin{equation*}
\|q-\tilde{q}\|_{L^2(0,1)}\le C_{A} \frac{1}{\ln\big(\frac{1}{\varepsilon}\big)}.
\end{equation*}
\noindent where $q$ is the potential defined in (\ref{ODE_2}).
$\quad$

\noindent 

$\quad$

\noindent Let us go back to the Sturm-Liouville equation
\begin{equation}
\label{eqdif}
-u''+qu=-zu,\quad z\in\mathbb{C}
\end{equation}	

\noindent and to the fundamental system of solutions $\{c_0,s_0\}$ and $\{c_1,s_1\}$ given by (\ref{solfond}). We define $\psi$ and $\phi$ as the two unique solutions of (\ref{eqdif}) that can be written as
\begin{equation}
\label{fonctions_2}
\psi(x,z)=c_0(x,z)+M(z)s_0(x,z),\quad \phi(x,z)=c_1(x,z)-N(z)s_1(x,z),
\end{equation}
with Dirichlet boundary conditions at $x=1$ and $x=0$ respectively.

$\quad$

\begin{proposition}
	\label{WT_relations_general}
	We have the following relations
	\begin{equation*}
	\begin{aligned}
		&s_0(1,z)=\Delta(z)\\
		&s_0'(1,z)=-N(z)\Delta(z)\\
		&c_0(1,z)=-M(z)\Delta(z) \\
		&c_0'(1,z)=M(z)N(z)\Delta(z)-\frac{1}{\Delta(z)}
		\end{aligned}
\qquad\quad\mathrm{and}\qquad\quad
	\begin{aligned}
		&s_1(0,z)=-\Delta(z)\\
		&s_1'(0,z)=-M(z)\Delta(z)\\
		&c_1(0,z)=-N(z)\Delta(z)\\
		&c_1'(0,z)=\frac{1}{\Delta(z)} - N(z)M(z)\Delta(z).
		\end{aligned}
	\end{equation*}
	
\end{proposition}

\begin{proof}
	
	\noindent First of all, the equalities $s_0(1,z)=\Delta(z)\:$ and $\: s_1(0,z)=-\Delta(z)$ come from the relation (\ref{charac_2}). The set of solutions of (\ref{eqdif}) that satisfy $u(0,z)=0$ is a one dimensional vector space. Therefore, there exists $A(z)\in\mathbb{C}$ such that :
\begin{equation}
\label{s1_2}
\forall x\in [0,1],\quad s_0(x,z)=A(z)\phi(x,z)
\end{equation}

\noindent The conditions on $c_1$ and $s_1$ at $x=1$ lead to the equality $A(z)=s_0(1,z)=\Delta (z)$. We get also, by differentiating $(\ref{s1_2})$:
\begin{equation*}
s_0'(1,z)= A(z)\big(c_1'(1,z)-N(z)s_1'(1,z)\big) =-N(z)A(z)=-\Delta(z)N(z).
\end{equation*}

\noindent Analogously, there is $B(z)\in\mathbb{C}$ such that
\begin{equation*}
s_1(x,z)=B(z)\psi(x,z)
\end{equation*}
and we show $B(z)=-\Delta(z)$. Hence $s_1'(1,z)=-\Delta(z)\psi'(1,z)$ and so $\displaystyle \psi'(1,z)=-\frac{1}{\Delta(z)}$. 

\noindent By differentiating $(\ref{s1_2})$ and taking $x=1$, we get:
\begin{equation*}
c_0'(1,z)+M(z)s_0'(1,z) = -\frac{1}{\Delta(z)}
\end{equation*}
and then
\begin{equation*}
c_0'(1,z) = M(z)N(z)\Delta(z)-\frac{1}{\Delta(z)}.
\end{equation*}

$\quad$

\noindent This proves the equalities on $c_0$ and $s_0$. We proceed similarly to establish those on $c_1$ and $s_1$.
\end{proof}

\medskip

\noindent Thanks to those relations, we are now able to prove the following lemma:

\medskip

\begin{lemma}
	\label{integral_relation_2}
	Denote $\mathcal{P}$ the poles of $N$. For any $z\in\mathbb{C}\backslash\mathcal{P}$ we have the equality :
	\begin{equation}
	\label{relation}
	\begin{aligned}
\tilde{\Delta}(z)\Delta(z)\bigg(M(z)N(z)-\frac{1}{\Delta(z)^2}\bigg)-M(z)\tilde{N}(z)\Delta(z)&\tilde{\Delta}(z)+1\\
&=\int_{0}^{1}\big(q(x)-\tilde{q}(x)\big)c_0(x,z)\tilde{s}_0(x,z)dx
	\end{aligned}
	\end{equation}
\end{lemma}

\begin{proof} Let us define $\displaystyle \theta :x\mapsto c_0(x,z)\tilde{s}_0'(x,z)-s_0'(x,z)\tilde{c}_0(x,z)$. Then :
	\begin{equation*}
	\begin{aligned}
	\theta'(x)&=c_0(x,z)\tilde{s}_0''(x,z)+c_0'(x,z)\tilde{s}_0'(x,z)-c_0'(x,z)\tilde{s}_0'(x,z)-c''_0(x,z)\tilde{s}_0(x,z)\\
	&=c_0(x,z)(\tilde{q}(x)\tilde{s}_0(x,z)+z\tilde{s}_0(x,z))-(q(x)c_0(x,z)+zc_0(x,z))\tilde{s}_0(x,z)\\
	&=(\tilde{q}(x)-q(x)\big)c_0(x,z)\tilde{s}_0(x,z)
	\end{aligned}
	\end{equation*}
	
	\medskip
	
	\noindent Hence, by integrating between $0$ and $1$ :
	\begin{equation*}
	\theta(1)-\theta(0)=\int_{0}^{1}\big(q(x)-\tilde{q}(x)\big)c_0(x,z)\tilde{s}_0(x,z)dx.
	\end{equation*}
	
	\noindent By replacing $c_0'(1,z)$, $\tilde{s}_0(1,z)$, $c_0(1,z)$ and $\tilde{s}_0'(1,z)$ by the expressions given in Proposition \ref{WT_relations_general}, we get the relation of Lemma \ref{integral_relation_2}.
\end{proof}

\noindent By inverting the roles of $q$ and $\tilde{q}$, we get

	\begin{equation}
	\begin{aligned}
	\Delta(z)\tilde{\Delta}(z)\bigg(\tilde{M}(z)\tilde{N}(z)-\frac{1}{\tilde{\Delta}(z)^2}\bigg)-\tilde{M}(z)N(z)&\Delta(z)\tilde{\Delta}(z)+1\\
	&=\int_{0}^{1}\big(\tilde{q}(x)-q(x)\big)\tilde{c}_0(x,z)s_0(x,z)dx
	\end{aligned}
\end{equation}

\noindent At last, from Remark \ref{WT_symmetry}, if we replace $q(x)$ and $\tilde{q}(x)$ by $q(1-x)$ and $\tilde{q}(1-x)$, then, the roles of $M$ and $N$ are inverted. Moreover, we remark that $c_1(1-x)$ and $-s_1(1-x)$ play the roles of $c_0(x)$ and $s_0(x)$ but for the potential $q(1-x)$, \emph{i.e.} denoting $\eta(x)=1-x$:
\begin{equation*}
c_0(x,z,,q\circ\eta)=c_1(1-x,z,q)\quad\mathrm{and}\quad  s_0(x,z,,q\circ\eta)=-s_1(1-x,z,q)
\end{equation*}

\noindent Hence :
	\begin{equation}
	\begin{aligned}
\Delta(z)\tilde{\Delta}(z)\bigg(\tilde{M}(z)\tilde{N}(z)-\frac{1}{\tilde{\Delta}(z)^2}\bigg)-&\tilde{N}(z)M(z)\Delta(z)\tilde{\Delta}(z)+1\\&=-\int_{0}^{1}\big(\tilde{q}(1-x)-q(1-x)\big)\tilde{c}_1(1-x,z)s_1(1-x,z)dx\\
	\end{aligned}
\end{equation}

\noindent As $q$ is symmetric, we have $c_1(1-x)=c_0(x)$ and $s_1(1-x)=-s_0(x)$. The previous equality can be written
	\begin{equation}
\begin{aligned}
\label{integral_relation_sym_2}
\Delta(z)\tilde{\Delta}(z)\bigg(\tilde{M}(z)\tilde{N}(z)-\frac{1}{\tilde{\Delta}(z)^2}\bigg)-\tilde{N}(z)M(z)\Delta(z)&\tilde{\Delta}(z)+1\\
&=\int_{0}^{1}\big(\tilde{q}(x)-q(x)\big) \tilde{c}_0(x,z)s_0(x,z)dx\\
\end{aligned}
\end{equation}

\noindent Hence, by substracting the relation of Lemma \ref{integral_relation_2} from equality (\ref{integral_relation_sym_2}), we get

\begin{equation*}
	\begin{aligned}
\Delta(x)\tilde{\Delta}(z)\bigg[\bigg(M(z)N(z)-\frac{1}{\Delta(z)^2}\bigg)-\bigg(\tilde{M}(z)\tilde{N}(z)-\frac{1}{\tilde{\Delta}(z)^2}\bigg)\bigg]=&\int_{0}^{1}\big(q(x)-\tilde{q}(x)\big)c_0(x,z)\tilde{s}_0(x,z)dx\\
&+\int_{0}^{1}\big(q(x)-\tilde{q}(x)\big)\tilde{c}_0(x,z)s_0(x,z)dx
	\end{aligned}
\end{equation*} 

\medskip

\noindent Using Proposition \ref{WT_estimates_2}, we have proved:

\medskip

\begin{proposition}
	\label{potentiels_estimates_2}
There is $m_0\in\mathbb{N}$ such that, for $m\ge m_0$ :
\begin{equation}
\label{eq_potentiels_estimates_2}
\bigg|\int_{0}^{1}\big(q(x)-\tilde{q}(x)\big)\big[c_0(x,\kappa_m)\tilde{s}_0(x,\kappa_m)+\tilde{c}_0(x,\kappa_m)s_0(x,\kappa_m)\big]dx\bigg| \le C_A  y_m |\Delta(\kappa_m)||\tilde{\Delta}(\kappa_m)|\varepsilon
\end{equation}
\end{proposition}

$\quad$

\subsection{Construction of an inverse integral operator}

\medskip

\noindent From now on, we set $L(x)=q(x)-\tilde{q}(x)$. We want to express the integrand in the left-hand-side of (\ref{eq_potentiels_estimates_2}) in terms of an operator acting on $L$.

\medskip

\begin{proposition}
	\label{operator_exist_estimate_2}
	There is an operator $B:L^2([0,1]) \to L^2([0,1])$ such that :
	\begin{enumerate}
		\item For all $m\in\mathbb{N}$, \begin{center}
			$\displaystyle 
			\int_{0}^{1}\big[c_0(x,\kappa_m)\tilde{s}_0(x,\kappa_m)+\tilde{c}_0(x,\kappa_m)s_0(x,\kappa_m)\big]L(x)dx=\frac{1}{y_m}\int_{0}^{1}\sinh(2\tau y_m)BL(\tau)d\tau.$
		\end{center}
		\item The function $\tau\mapsto BL(\tau)$ is $C^1$ on $[0,1]$ and $BL$ and $(BL)'$ are uniformly bounded by a constant $C_A$.
	\end{enumerate}
\end{proposition}

\begin{proof}

\noindent Let us extend on $[-1,0]$ $q$ and $\tilde{q}$ into even functions. From \cite{marchenko2011sturm} (page 9) we have the following integral representations of the functions $c_0$ and $s_0$:

\begin{equation*}
\begin{aligned}
&s_0(x,-z^2)=\frac{\sin(z x)}{z}+\int_{0}^{x}H(x,t)\frac{\sin(z t)}{z}dt\\
&c_0(x,-z^2)= \cos(zx) + \int_{0}^{x}P(x,t)\cos(z t)dt
\end{aligned}
\end{equation*}

\noindent where $H(x,t)$ and $P(x,t)$ can be written as 
\begin{equation}
\label{tulipe_2}
\begin{aligned}
&H(x,t)=K(x,t)-K(x,-t)\\
&P(x,t)=K(x,t)+K(x,-t)
\end{aligned}
\end{equation}
\noindent  with $K$ a $C^1$ function on $[-1,1]\times [-1,1]$ satisfying some good estimates. More precisely (\cite{marchenko2011sturm}, p.14), we have:

\begin{theorem}
	On $[-1,1]\times [-1,1]$, $K$ satisfies the estimate
	\begin{equation*}
	|K(x,t)|\le \frac{1}{2}w\bigg(\frac{x+t}{2}\bigg)\exp\bigg(\sigma_1(x)-\sigma_1\bigg(\frac{x+t}{2}\bigg)-\sigma_1\bigg(\frac{x-t}{2}\bigg)\bigg)
	\end{equation*}
\end{theorem}

\noindent with $\displaystyle w(u)=\underset{0\le \xi\le u}{\max}\bigg|\int_{0}^{\xi}q(y)dy\bigg|,\quad$ $\displaystyle\quad \sigma_0(x)=\int_{0}^{x}|q(t)|dt,\quad$ $\displaystyle\quad \sigma_1(x)=\int_{0}^{x}\sigma_0(t)dt$.

$\quad$

$\quad$

\noindent We thus have the following estimate :

\medskip

\begin{proposition}
	\label{noyau_estim_2}
There is a constant $C_A>0$, which only depends on $A$, such that
	\begin{equation*}
		\|K\|_\infty + \bigg\|\frac{\partial K}{\partial x}\bigg\|_\infty+ \bigg\|\frac{\partial K}{\partial t}\bigg\|_\infty \le C_A.
	\end{equation*}
\end{proposition}

\begin{proof}
	Since $f\in C(A)$, the potential $q$ is bounded by a constant that only depends on $A$, so are $\sigma_0$, $\sigma_1$, $w$ and $K$. Denote $J(u,v)=K(u+v,u-v)$. Then $J$ is uniformly bounded by $C_A$ and moreover (cf \cite{marchenko2011sturm}, p. 14 and 16), one has the equalities :
	\begin{equation*}
	\left\{\begin{aligned}
	& \frac{\partial J(u,v)}{\partial u}=\frac{1}{2}q(u)+\int_{0}^{v}q(u+\beta)J(u,\beta)d\beta \\
	&\frac{\partial J(u,v)}{\partial v}=\int_{0}^{u}q(v+\alpha)J(\alpha,v)d\beta \
	\end{aligned}\right.
	\end{equation*}
	\noindent We deduce that the partial derivative of $J$ are uniformly bounded by $C_A$. Returning to the $(x,t)$ coordinates, the conclusion of Proposition \ref{noyau_estim_2} follows.
\end{proof}

$\quad$

\noindent For $z=i\sqrt{\kappa_m}=:iy_m$, we have thus :

\begin{equation*}
\begin{aligned}
&s_0(x,\kappa_m)=\frac{\sinh(y_m x)}{y_m}+\int_{0}^{x}H(x,t)\frac{\sinh(y_mt)}{y_m}dt\\
&c_0(x,\kappa_m)=\cosh(y_m x)+\int_{0}^{x}H(x,t)\cosh(y_mt)dt
\end{aligned}
\end{equation*}

$\quad$

\noindent We will take advantage of this representation to write the estimates of Proposition \ref{potentiels_estimates_2} as an integral estimate. 

\medskip

\noindent 	\noindent We have
\begin{equation*}
\begin{aligned}
\int_{0}^{1}L(x)s_0(x)\tilde{c}_0(x)dx&=\int_{0}^{1}L(x)\bigg[\frac{\sinh(y_m x)}{y_m}+\int_{0}^{x}H(x,t)\frac{\sinh(y_m t)}{y_m}dt\bigg]\times\\
&\hspace{4cm}\bigg[\cosh( y_m x)+\int_{0}^{x}\tilde{P}(x,u)\cosh(y_m u)du\bigg]dx\\	
&=\mathrm{I}_0+\mathrm{II}_0+\mathrm{III}_0+\mathrm{IV}_0,
\end{aligned}
\end{equation*}

\noindent with

\medskip

\begin{enumerate}
	\item[$\bullet$] $\displaystyle \mathrm{I}_0=\int_{0}^{1}L(x)\frac{\sinh(y_m x)\cosh(y_mx)}{y_m}dx$
	
	\item[$\bullet$] $\displaystyle \mathrm{II}_0=\int_{0}^{1}L(x)\bigg[\int_{0}^{x}\tilde{P}(x,u)\frac{\sinh(y_m x)\cosh(y_m u)}{y_m}du\bigg]dx$
	
	\item[$\bullet$] $\displaystyle \mathrm{III}_0 =\int_{0}^{1}L(x)\bigg[\int_{0}^{x}H(x,t)\frac{\sinh(y_m t)\cosh(y_m x)}{y_m}dt\bigg]dx$
	
	\item[$\bullet$] $\displaystyle \mathrm{IV}_0 =\int_{0}^{1}L(x)\bigg[\int_{0}^{x}\int_{0}^{x}\tilde{P}(x,u)H(x,t)\frac{\sinh(y_m t)\cosh(y_m u)}{y_m}du\:dt\bigg]dx$
\end{enumerate}

$\quad$

\noindent Let us compute those four quantities independently.

\medskip

\begin{center}
	$\displaystyle \mathrm{I}_0=\int_{0}^{1}L(x)\frac{\sinh(y_m x)\cosh(y_mx)}{y_m}dx=\frac{1}{2y_m}\int_{0}^{1}\sinh(2xy_m)L(x)dx$.
\end{center}

$\quad$

\begin{equation*}
\begin{aligned}
\mathrm{II}_0&=\int_{0}^{1}L(x)\bigg[\int_{0}^{x}\tilde{P}(x,u)\frac{\sinh(y_m x)\cosh(y_m u)}{y_m}du\bigg]dx\\&=\frac{1}{y_m}\int_{0}^{1}L(x)\bigg[\int_{0}^{x}\tilde{P}(x,u)\frac{\sinh(y_m (x+u))+\sinh(y_m (x-u))}{2}du\bigg]dx\\
&=\frac{1}{2y_m}\int_{0}^{1}L(x)\bigg[\int_{0}^{x}\tilde{P}(x,u)\sinh(y_m (x+u))du+\int_{0}^{x}\tilde{P}(x,u)\sinh(y_m (x-u))du\bigg]dx\\
&=\frac{1}{y_m}\int_{0}^{1}L(x)\bigg[\int_{\frac{x}{2}}^{x}\tilde{P}(x,2\tau-x)\sinh(2 \tau y_m)d\tau+\int_{0}^{\frac{x}{2}}\tilde{P}(x,x-2\tau)\sinh(2\tau y_m)d\tau\bigg]dx\\
&=\frac{1}{y_m}\int_{0}^{1}\sinh(2 \tau y_m)\bigg[\int_{\tau}^{2\tau}\tilde{P}(x,2\tau-x)L(x)dx+\int_{2\tau}^{1}\tilde{P}(x,x-2\tau)L(x)dx\bigg]d\tau
\end{aligned}
\end{equation*}

$\quad$ 

\noindent But, for all $(x,\tau)$ in $\mathbb{R}^2$, we have $\tilde{P}(x,x-2\tau)=\tilde{P}(x,2\tau-x)$. Then

\begin{equation*}
\begin{aligned}
\mathrm{II}_0&=\frac{1}{y_m}\int_{0}^{1}\sinh(2 \tau y_m)\bigg[\int_{\tau}^{1}\tilde{P}(x,2\tau-x)L(x)dx\bigg]d\tau\\
\end{aligned}
\end{equation*}

\noindent Let us compute $\mathrm{III}_0$ :

\begin{equation*}
\begin{aligned}
\displaystyle \mathrm{III}_0 &=\int_{0}^{1}L(x)\bigg[\int_{0}^{x}H(x,t)\frac{\sinh(y_m t)\cosh(y_m x)}{y_m}dt\bigg]dx\\
&=\frac{1}{2y_m}\int_{0}^{1}L(x)\bigg[\int_{0}^{x}H(x,t)\sinh(y_m (t+x))dt+\int_{0}^{x}H(x,t)\sinh(y_m (t-x))dt\bigg]dx\\
&=\frac{1}{y_m}\int_{0}^{1}L(x)\bigg[\int_{\frac{x}{2}}^{x}H(x,2\tau-x)\sinh(2\tau y_m)d\tau+\int_{-\frac{x}{2}}^{0}H(x,2\tau +x)\sinh(2\tau y_m)d\tau\bigg]dx\\
\end{aligned}
\end{equation*}

\noindent By changing $\tau$ in $-\tau$, we get
\begin{equation*}
\begin{aligned}
\displaystyle \mathrm{III}_0 &=   \frac{1}{y_m}\int_{0}^{1}L(x)\bigg[\int_{\frac{x}{2}}^{x}H(x,2\tau-x)\sinh(2\tau y_m)dt+\int_{0}^{\frac{x}{2}}H(x,-2\tau +x)\sinh(-2\tau y_m)dt\bigg]dx\\
&=\frac{1}{y_m}\int_{0}^{1}L(x)\bigg[\int_{\frac{x}{2}}^{x}H(x,2\tau-x)\sinh(2\tau y_m)d\tau-\int_{0}^{\frac{x}{2}}H(x,-2\tau +x)\sinh(2\tau y_m)d\tau\bigg]dx\\
\end{aligned}
\end{equation*}

$\quad$

\noindent As $H$ is odd with respect to the second variable :
\begin{equation*}
\begin{aligned}
\displaystyle \mathrm{III}_0&=\frac{1}{y_m}\int_{0}^{1}L(x)\bigg[\int_{0}^{x}H(x,2\tau-x)\sinh(2\tau y_m)d\tau\bigg]dx\\
&=\frac{1}{y_m}\int_{0}^{1}\sinh(2\tau y_m)\bigg[\int_{\tau}^{1}H(x,2\tau-x)L(x)dx\bigg]d\tau
\end{aligned}
\end{equation*} 

\noindent At last :

\begin{equation*}
\begin{aligned}
\mathrm{IV}_0 &=\frac{1}{2y_m}\int_{0}^{1}L(x)\bigg[\int_{0}^{x}\int_{0}^{x}\tilde{P}(x,u)H(x,t)\big(\sinh(y_m (t+u))+\sinh(y_m (t-u))\big)du\:dt\bigg]dx\\
&=\mathrm{IV}_0(1) + \mathrm{IV}_0(2)
\end{aligned}
\end{equation*}

\noindent where
\begin{equation*}
\begin{aligned}
\mathrm{IV}_0(1)&=\frac{1}{2y_m}\int_{0}^{1}L(x)\int_{0}^{x}\int_{0}^{x}\tilde{P}(x,u)H(x,t)\sinh(y_m (t+u))dudtdx\\
&=\frac{1}{2y_m}\int_{0}^{1}L(x)\int_{0}^{1}{\bf 1}_{[0,x]}(t)\int_{\frac{t}{2}}^{\frac{x+t}{2}}2\tilde{P}(x,2\tau -t)H(x,t)\sinh(2\tau y_m)d\tau dtdx\\
&=\frac{1}{y_m}\int_{0}^{1}L(x)\int_{0}^{1}\sinh(2\tau y_m){\bf 1}_{[0,x]}(\tau)\int_{2\tau-x}^{2\tau}\tilde{P}(x,2\tau -t)H(x,t){\bf 1}_{[0,x]}(t)dtd\tau dx\\
&=\frac{1}{y_m}\int_{0}^{1}\sinh(2\tau y_m)\int_{\tau}^{1}L(x)\int_{2\tau-x}^{2\tau}\tilde{P}(x,2\tau -t)H(x,t){\bf 1}_{[0,x]}(t)dtdxd\tau\\
\end{aligned}
\end{equation*}

\noindent and

\begin{equation*}
\begin{aligned}
\mathrm{IV}_0(2)&=\frac{1}{2y_m}\int_{0}^{1}L(x)\int_{0}^{x}\int_{0}^{x}\tilde{P}(x,u)H(x,t)\sinh(y_m (t-u))du\:dtdx\\
&=\frac{1}{2y_m}\int_{0}^{1}L(x)\int_{0}^{1}{\bf 1}_{[0,x]}(t)\int_{\frac{t-x}{2}}^{\frac{t}{2}}2\tilde{P}(x,t-2\tau)H(x,t)\sinh(2\tau y_m)d\tau\:dtdx\\
&=\frac{1}{y_m}\int_{0}^{1}L(x)\int_{-1}^{1}\sinh(2\tau y_m){\bf 1}_{[-x,x]}(2\tau)\int_{2\tau}^{2\tau +x}\tilde{P}(x,t-2\tau)H(x,t){\bf 1}_{[0,x]}(t)\:dtd\tau dx\\
&=\frac{1}{y_m}\int_{-1}^{1}\sinh(2\tau y_m)\int_{2|\tau|}^{1}L(x)\int_{2\tau}^{2\tau +x}\tilde{P}(x,t-2\tau)H(x,t){\bf 1}_{[0,x]}(t)\:dtdxd\tau\\
&=\mathrm{IV}_0(2,1)+\mathrm{IV}_0(2,2)
\end{aligned}
\end{equation*}

\noindent with

\begin{equation*}
\begin{aligned}
\mathrm{IV}_0(2,1)&=\frac{1}{y_m}\int_{-1}^{0}\sinh(2\tau y_m)\int_{-2\tau}^{1}L(x)\int_{2\tau}^{2\tau +x}\tilde{P}(x,t-2\tau)H(x,t){\bf 1}_{[0,x]}(t)\:dtdxd\tau\\
&=\frac{1}{y_m}\int_{0}^{1}\sinh(2\tau y_m)\int_{2\tau}^{1}L(x)\int_{-2\tau}^{-2\tau +x}\tilde{P}(x,t+2\tau)H(x,t){\bf 1}_{[0,x]}(t)\:dtdxd\tau\\
&=\frac{1}{y_m}\int_{0}^{1}\sinh(2\tau y_m)\int_{2\tau}^{1}L(x)\int_{0}^{-2\tau +x}\tilde{P}(x,t+2\tau)H(x,t)\:dtdxd\tau\\
&=\frac{1}{y_m}\int_{0}^{1}\sinh(2\tau y_m)\int_{2\tau}^{1}L(x)\int_{2\tau}^{x}\tilde{P}(x,t)H(x,t-2\tau)\:dtdxd\tau
\end{aligned}
\end{equation*}

\noindent and

\begin{equation*}
\begin{aligned}
\mathrm{IV}_0(2,2)&=\frac{1}{y_m}\int_{0}^{1}\sinh(2\tau y_m)\int_{2\tau}^{1}L(x)\int_{2\tau}^{2\tau +x}\tilde{P}(x,t-2\tau)H(x,t){\bf 1}_{[0,x]}(t)\:dtdxd\tau\\
&=\frac{1}{y_m}\int_{0}^{1}\sinh(2\tau y_m)\int_{2\tau}^{1}L(x)\int_{2\tau}^{x}\tilde{P}(x,t-2\tau)H(x,t)\:dtdxd\tau\\
\end{aligned}
\end{equation*}

\noindent Finally :

\begin{equation*}
\int_{0}^{1}s_0(x)\tilde{c}_0(x)L(x)dx = \frac{1}{y_m}\int_{0}^{1}\sinh(2\tau y_m)QL(\tau)d\tau
\end{equation*}
with
\begin{equation*}
\begin{aligned}
QL(\tau) =  \frac{1}{2}L(\tau)&+\int_{\tau}^{1}\tilde{P}(x,2\tau-x)L(x)dx+\int_{\tau}^{1}H(x,2\tau-x)L(x)dx\\
&+\int_{\tau}^{1}L(x)\int_{2\tau-x}^{2\tau}\tilde{P}(x,2\tau -t)H(x,t){\bf 1}_{[0,x]}(t)dtdx\\
&+\int_{2\tau}^{1}L(x)\int_{2\tau}^{x}\tilde{P}(x,t)H(x,t-2\tau)\:dtdx\\
&+\int_{2\tau}^{1}L(x)\int_{2\tau}^{x}\tilde{P}(x,t-2\tau)H(x,t)\:dtdx
\end{aligned}
\end{equation*}

$\quad$

\noindent Similarly, inverting the $\:\tilde{}\:$, we construct as well an operator $R:L^2(0,1)\to L^2(0,1)$ such that
\begin{equation*}
\int_{0}^{1}\tilde{s}_0(x)c_0(x)L(x)dx = \frac{1}{y_m}\int_{0}^{1}\sinh(2\tau y_m)RL(\tau)d\tau
\end{equation*}
with
\begin{equation*}
\begin{aligned}
RL(\tau) =  \frac{1}{2}L(\tau)&+\int_{\tau}^{1}P(x,2\tau-x)L(x)dx+\int_{\tau}^{1}\tilde{H}(x,2\tau-x)L(x)dx\\
&+\int_{\tau}^{1}L(x)\int_{2\tau-x}^{2\tau}P(x,2\tau -t)\tilde{H}(x,t){\bf 1}_{[0,x]}(t)dtdx\\
&+\int_{2\tau}^{1}L(x)\int_{2\tau}^{x}P(x,t)\tilde{H}(x,t-2\tau)\:dtdx\\
&+\int_{2\tau}^{1}L(x)\int_{2\tau}^{x}P(x,t-2\tau)\tilde{H}(x,t)\:dtdx
\end{aligned}
\end{equation*}

$\quad$

\noindent Let us denote $B=Q+R$. Then
\begin{equation*}
	\begin{aligned}
	\int_{0}^{1}\big[c_0(x,z)\tilde{s}_0(x,z)+\tilde{c}_0(x,z)s_0(x,z)\big]L(x)dx&=\frac{1}{y_m}\int_{0}^{1}\sinh(2\tau y_m)(R+Q)L(\tau)d\tau\\
	&=\frac{1}{y_m}\int_{0}^{1}\sinh(2\tau y_m)BL(\tau)d\tau.
	\end{aligned}
\end{equation*}

	\noindent Now, let us prove the second part of the proposition. As the conformal factors $f$ and $\tilde{f}$ belong to $C(A)$, and thanks to Proposition \ref{noyau_estim_2}, we know that $H$ and $\tilde{H}$ are $C^1$ and uniformly bounded by a constant $C_A$ (and also are their partial derivatives). Moreover, it is known that, for a function $g$ that is $C^1$ on $[0,1]$, for any $a\in ]0,1[$ the function $G_a$ defined as $\displaystyle G_a(\tau)= \int_{a}^{\tau}g(\tau,x)dx$ is also $C^1$ and its derivative is
\begin{equation*}
G_a'(\tau)=\int_{a}^{\tau}\frac{\partial g}{\partial \tau}(\tau,x)dx + g(\tau,\tau).
\end{equation*} Hence $BL$ and its derivative are also bounded by some constant $C_A$.
\end{proof}

$\quad$

\noindent Thus, we have obtained:
\begin{equation}
\label{estimation_2}
\bigg|\frac{1}{y_m^2}\int_{0}^{1}\sinh(2\tau y_m)BL(\tau)d\tau\bigg| \le C_A \varepsilon\times  \Delta(\kappa_m)\tilde{\Delta}(\kappa_m)
\end{equation}

\noindent Moreover
\begin{equation*}
\begin{aligned}
y_m^2 e^{-2ym}\times \frac{1}{y_m^2}\int_{0}^{1}\sinh(2\tau y_m)BL(\tau)d\tau &= \frac{1}{2}\bigg[e^{-2ym}\int_{0}^{1}e^{2\tau y_m}BL(\tau)d\tau + e^{-2ym}\int_{0}^{1}e^{-2\tau y_m}BL(\tau)d\tau\bigg]\\
&=\frac{1}{2}\bigg[\int_{0}^{1}e^{2(\tau-1) y_m}BL(\tau)d\tau + \int_{0}^{1}e^{-2(\tau+1) y_m}BL(\tau)d\tau\bigg]\\
&=\frac{1}{2}\bigg[\int_{0}^{1}e^{-2\tau y_m}BL(1-\tau)d\tau + \int_{1}^{2}e^{-2\tau y_m}BL(\tau-1)d\tau\bigg]
\end{aligned}
\end{equation*}
and so, by multiplying (\ref{estimation_2}) by $y_m^2e^{-2y_m}$, one gets, for $m\ge m_0$ :
\begin{equation*}
\begin{aligned}
\bigg|\int_{0}^{+\infty}e^{-2\tau y_m}\bigg(BL(1-\tau){\bf 1}_{[0,1]}(\tau) + BL(\tau-1){\bf 1}_{[1,2]}(\tau)\bigg)d\tau\bigg| &\le C_A \varepsilon\times \big[y_m^2 e^{-2ym}\Delta(\kappa_m)\tilde{\Delta}(\kappa_m)\big]\\
&\le C_A \varepsilon.
\end{aligned}
\end{equation*}

\subsection{A Müntz approximation theorem}

\subsubsection{A Hausdorf moment problem}

\medskip

\noindent Let us set
\begin{equation*}
	g(\tau)=BL(1-\tau){\bf 1}_{[0,1]}(\tau) + BL(\tau-1){\bf 1}_{[1,2]}(\tau)
\end{equation*}

\medskip

\noindent The change of variable $t=e^{-\tau}$ leads to the estimates :
\begin{equation*}
\forall m\ge m_0,\quad	\bigg|\int_{0}^{1}t^{2y_m-1}g(-\ln(t))dt\bigg|\le C_A\varepsilon.
\end{equation*}

\noindent We recall that, for all $m\in\mathbb{N}$, we have set $y_m=\sqrt{\kappa_m}$, where $\kappa_m=m(m+n-2)$. Let us set $\alpha=2y_{m_0}-1$ and  

\begin{equation}
\label{key_2}
	\lambda_m:=2y_m-1-\alpha
\end{equation}

\noindent Then, by denoting 
\begin{equation*}
\displaystyle h(t)=t^\alpha g(-\ln(t)),
\end{equation*}
we get:
\begin{equation}
\label{moments_estimates}
\bigg|\int_{0}^{1}t^{\lambda_m}h(t)dt\bigg|\le C_A\varepsilon,\quad\forall m\in\mathbb{N}.
\end{equation}

$\quad$

\noindent \noindent Thus, we would like now to answer the following question : does the approximate knowledge of the moments of $h$ on the sequence $(\lambda_m)_{m\in\mathbb{N}}$ determine $h$ up to a small error in $L^2$ norm ? 

$\quad$

\noindent Let us fix $m\in\mathbb{N}$ (we will precise it later) and consider the finite real sequence :

\begin{equation*}
\Lambda_m:0= \lambda_0<\lambda_1<...<\lambda_m.
\end{equation*}

\begin{definition}
	The subspace of the Müntz polynomials of degree $\lambda_m$ is defined as :
	\begin{equation*}
	\mathcal{M}(\Lambda_m)=\{P:\: P(x)=\sum_{k=0}^{m}a_kx^{\lambda_k}\}.
	\end{equation*}
\end{definition}

\begin{definition}
	The $L^2$-error of approximation from $\mathcal{M}(\Lambda_m)$ of a function $f\in L^2([0,1])$ is :
	\begin{equation*}
	E_2(f,\Lambda_m)=\underset{P\in\mathcal{M}(\Lambda_m)}{\inf}\|f-P\|_2.
	\end{equation*}
\end{definition} 

$\quad$

\noindent $E_2(h,\Lambda_m)$ appears in an estimate of $\|h\|_{2}$ given by Proposition \ref{norm_2}. Thanks to the Gram-Schmidt process, we define the sequence of Müntz polynomials $\big(L_p(x)\big)$ as $L_0\equiv 1$ and, for $p\ge 1$ :
\begin{equation*}
L_p(x)=\sum_{j=0}^{p}C_{pj}x^{\lambda_j},
\end{equation*}

\noindent where :
\begin{equation*}
C_{pj}=\sqrt{2\lambda_p+1}\frac{\prod_{r=0}^{p-1}(\lambda_j+\lambda_r+1)}{\prod_{\substack{r=0,r\ne j}}^{p}(\lambda_j-\lambda_r)}.
\end{equation*}

$\quad$

\begin{proposition}
	\label{norm_2}
Under the assumption (\ref{moments_estimates}), we have the following estimate :
	We have the following estimate :
	\begin{equation*}
\|h\|_2^2\le C_A\varepsilon^2\sum_{k=0}^{m}\bigg(\sum_{\ell=0}^{k}|C_{k\ell}|\bigg)^2 +E_2(h,\Lambda_m)^2.
	\end{equation*}
\end{proposition}

\begin{proof} Let us denote $\displaystyle \pi(h)=\sum_{k=0}^{m}\langle L_k,h\rangle L_k$ the orthogonal projection of $h$ on $\mathcal{M}(\Lambda_m)$.
	\begin{equation*}
	\begin{aligned}
	\|h\|^2_2&=\|\pi(h)\|^2+ \|h-\pi(h)\|^2_2\\
	&=\sum_{k=0}^{m}\langle L_k,h\rangle ^2+ E_2(\Lambda_m,h)^2.
	\end{aligned}
	\end{equation*}
	\noindent As \begin{center}
		$\displaystyle |\langle L_k,h\rangle| = \bigg|\sum_{\ell=0}^{k}C_{k\ell}\underbrace{\int_{0}^{1}x^{\lambda_\ell}h(x)dx}_{\le C_A\varepsilon} \bigg|\le C_A\varepsilon\sum_{\ell=0}^{k}|C_{k\ell}|$,
	\end{center} 
	\noindent one gets
	\begin{equation*}
	\|h\|^2_2 \le C_A\varepsilon^2\sum_{k=0}^{m}\bigg(\sum_{\ell=0}^{k}|C_{k\ell}|\bigg)^2+E_2(\Lambda_m,h)^2.
	\end{equation*}
\end{proof}
\noindent We would like to find $m(\varepsilon)\in\mathbb{N}$ satisfying :
\begin{equation*}
\lim\limits_{\varepsilon\to 0}m(\varepsilon)=+\infty
\end{equation*}
\noindent and such that \begin{center}
	$\displaystyle \sum_{k=0}^{m(\varepsilon)}\bigg(\sum_{\ell=0}^{k}|C_{k\ell}|\bigg)^2 \le \frac{1}{\varepsilon}$,
\end{center}

\medskip

\noindent in order to obtain $\|h\|^2_2\le C_A\varepsilon + E_2(\Lambda_{m(\varepsilon)},h)$.

\medskip

\begin{lemma}
	\label{ineq_2}
	$\quad$
	\begin{enumerate}
		\item For all $m\in\mathbb{N}$, $\lambda_{m+1}-\lambda_m\ge 2$.
		\item For all $m\in\mathbb{N}$, $\displaystyle \lambda_{m+1}-\lambda_m= 2+O\bigg(\frac{1}{m}\bigg)$.
	\end{enumerate}
\end{lemma}
\begin{proof}
	
	$\quad$
	
	$\quad$
	
	\noindent 1. Let $m\in\mathbb{N}$ and set $a=n-2$. From (\ref{key_2}) we have the equivalence $\lambda_{m+1}-\lambda_m\ge 2\Leftrightarrow y_{m+1}-y_m\ge 1$, where $y_m=\sqrt{m^2+am}$. For $m\in\mathbb{N}$, one has :
	\begin{equation*}
	\begin{aligned}
	y_{m+1}-y_m\ge 1&\Leftarrow \sqrt{(m+1)^2+a(m+1)}-\sqrt{m^2+am}\ge 1\\
	&\Leftrightarrow(m+1)^2+a(m+1)-m^2-am\ge \sqrt{(m+1)^2+a(m+1)}+\sqrt{m^2+am}\\
	&\Leftrightarrow 2m+1+a \ge m+1 + \frac{a}{2}-\frac{a}{8(m+1)}+m+\frac{a}{2}-\frac{a}{8m}+o\bigg(\frac{1}{m}\bigg)\\
	&\Leftrightarrow\frac{a}{8(m+1)}\ge -\frac{a}{8m}+ o\bigg(\frac{1}{m}\bigg),
	\end{aligned}
	\end{equation*}
	\noindent and that is true for $m$ large enough. We assume, without loss of generality, that it is true for all $m\ge m_0$. Hence, for all $m\in\mathbb{N}$, $\lambda_{m+1}-\lambda_m \ge 2$.
	
	$\quad$
	
	\noindent 2. Let $m\in\mathbb{N}$ and $u_m=\sqrt{\kappa_\ell}\:$ for some $\ell\in\mathbb{N}$. Then \begin{center}
		$\displaystyle y_{m+1}=\sqrt{\kappa_{m+1}}=\sqrt{\kappa_m}+1+O\bigg(\frac{1}{m}\bigg)=y_m+1+O\bigg(\frac{1}{m}\bigg),$
	\end{center}
	
	\noindent so we have the result.
\end{proof}

\medskip

\noindent Hence, there is $C>0$ such that, for all $m\in\mathbb{N}$, $\lambda_m\le 2m+C$. By setting $M_1=\max(2,2C+1)$, one gets :
\begin{equation*}
\prod_{r=0}^{p-1}(\lambda_j+\lambda_r+1)\le \prod_{r=0}^{p-1}(2j+2r+2C+1)\le M_1^p 	\prod_{r=0}^{p-1}(j+r+1).
\end{equation*}

\noindent On the other hand, for all $m\in\mathbb{N}$, $\lambda_{m+1}-\lambda_m\ge 2$. Let $m\in\mathbb{N}$ and $(r,j)\in\mathbb{N}$ such that $0\le r,j\le m$, $r\ne j$.
\begin{equation*}
\begin{aligned}
|\lambda_j-\lambda_r| &= |\lambda_j-\lambda_{j-1}|+|\lambda_{j-1}-\lambda_{j-2}|+...+|\lambda_{r+1}-\lambda_r|\\
&\ge 2|j-r|.
\end{aligned}
\end{equation*}

\noindent Consequently:
\begin{equation*}
|\prod_{\substack{r=0,r\ne j}}^{p}(\lambda_j-\lambda_r)|\ge 2^p\bigg|\prod_{\substack{r=0,r\ne j}}^{p}(j-r)\bigg|
\end{equation*}

\noindent 	It follows that
\begin{equation*}
\begin{aligned}
|C_{pj}|&\le \sqrt{4p+2C+1}{\bigg(\frac{M_1}{2}\bigg)^p}\frac{\prod_{r=0}^{p-1}|j+r+1|}{\prod_{\substack{r=0,r\ne j}}^{p}|j-r|}\\
&=\sqrt{4p+2C+1}{\bigg(\frac{M_1}{2}\bigg)^p}\frac{(j+1)...(j+p)}{j(j-1)...2\times 1\times 2\times... (p-j)}\\
&=\sqrt{4p+2C+1}{\bigg(\frac{M_1}{2}\bigg)^p}\frac{(j+p)!}{(j!)^2(p-j)!}
\end{aligned}
\end{equation*}

\noindent The multinomial formula stipulates that for any real finite sequence $(x_0,...,x_m)$ and any $n\in\mathbb{N}$:
\begin{equation*}
\bigg(\sum_{k=0}^{m}x_k\bigg)^n=\sum_{k_1+...+k_m=n}\binom{n}{k_1,k_2,...,k_m}x_1^{k_1}...x_m^{k_m},
\end{equation*}

\noindent where $\displaystyle\binom{n}{k_1,k_2,...,k_m}=\frac{n!}{k_1!k_2!...k_m!}$. 

$\quad$

\noindent As $j+j+(p-j)=j+p$, one deduces that :
\begin{equation*}
\frac{(j+p)!}{(j)!(j)!(p-j)!}\le (1+1+1)^{j+p} = 3^{j+p}
\end{equation*}

\noindent Hence (see \cite{daude2019stability} or \cite{ang2004moment}, chapter $4$, for similar computations) :
\begin{equation*}
\begin{aligned}
\varepsilon^2\sum_{k=0}^{m}\bigg(\sum_{\ell=0}^{k}|C_{k\ell}|\bigg)^2&\le  \varepsilon^2\sum_{k=0}^{m}\bigg(\sum_{\ell=0}^{k}\sqrt{4k+2C+1}\bigg(\frac{M_1}{2}\bigg)^k3^{k+\ell}\bigg)^2\\
&=  \varepsilon^2\sum_{k=0}^{m}\bigg(\frac{3M_1}{2}\bigg)^{2k}(4k+2C+1)\bigg(\sum_{\ell=0}^{k}3^{\ell}\bigg)^2\\
&\le \varepsilon^2(4m+2C+1)\sum_{k=0}^{m}\bigg(\frac{3M_1}{2}\bigg)^{2k}\bigg(\sum_{\ell=0}^{k}3^{\ell}\bigg)^2\\
&\le \varepsilon^2(4m+2C+1)\sum_{k=0}^{m}\bigg(\frac{3M_1}{2}\bigg)^{2k}\frac{3}{2}\times 3^{2k}\\
&\le \varepsilon^2\times \frac{3}{2} (4m+2C+1)\sum_{k=0}^{m}\bigg(\frac{9M_1}{2}\bigg)^{2k}\\
&\le \varepsilon^2\times \frac{3}{2} (4m+2C+1)(m+1)\bigg(\frac{9M_1}{2}\bigg)^{2m}\\
&=\varepsilon^2 g(m)^2
\end{aligned}
\end{equation*}

\noindent where $\displaystyle g(t):=\frac{3}{2}(4t+2C+1)(t+1)\bigg(\frac{9M_1}{2}\bigg)^{2t}$. 

$\quad$

\noindent As $g$ is a strictly increasing function on $\mathbb{R}_+$, we can set, for $\varepsilon$ small enough, $\displaystyle m(\varepsilon)=E\bigg(g^{-1}\bigg(\frac{1}{\sqrt{\varepsilon}}\bigg)\bigg)$. Thanks to this choice, we have \begin{equation*}
g\big(m(\varepsilon)\big)\le \frac{1}{\sqrt{\varepsilon}},
\end{equation*}
\noindent so that
\begin{equation*}
\varepsilon^2\sum_{k=0}^{m(\varepsilon)}\bigg(\sum_{p=0}^{k}|C_{kp}|\bigg)^2\le \varepsilon.
\end{equation*}

$\quad$

\noindent Let us now estimate $E_2(\Lambda_m,h)$. To this end, we recall some definitions.

\begin{definition}
	The index of approximation of $\Lambda_m$ in $L^2([0,1])$ is :
	\begin{equation*}
	\varepsilon_2(\Lambda_m)=\underset{y\ge 0}{\max}\bigg|\frac{B(1+iy)}{1+iy}\bigg|
	\end{equation*}
	where $B:\mathbb{C}\to\mathbb{C}$ is the Blaschke product defined as :
	\begin{equation*}
	B(z):=B(z,\Lambda_m)=\prod_{k=0}^{m}\frac{z-\lambda_k-\frac{1}{2}}{z+\lambda_k+\frac{1}{2}}
	\end{equation*}
\end{definition} 
\noindent We will take advantage of a much simpler expression of $\varepsilon_2\big(\Lambda_m\big)$, thanks to the following Theorem (\cite{lorentz1996constructive}, p.360):

\medskip

\begin{theorem}
	\label{simpler_prod_2}
	Let $\Lambda_m:0=\lambda_0<\lambda_1<...<\lambda_m$ be a finite sequence. Assume that $\lambda_{k+1}-\lambda_k\ge 2$ for $k\ge 0$. Then :
	\begin{equation*}
	\varepsilon_2\big(\Lambda_m\big)=\prod_{k=0}^{m}\frac{\lambda_k-\frac{1}{2}}{\lambda_k+\frac{3}{2}}
	\end{equation*}
\end{theorem}

\begin{definition}
	For a function $f\in L^2([0,1])$,  its $L^2$-modulus of continuity $w(f,.):\:]0,1[\to\mathbb{R}$ is defined as:
	\begin{equation*}
	w(f,u)=\underset{0\le r\le u}{\mathrm{sup}}\bigg(\int_{0}^{1-r}|f(x+r)-f(x)|^2dx\bigg)^{\frac{1}{2}}.
	\end{equation*} 
\end{definition}

\noindent The introduction of the two previous concepts is motivated by the following result (cf \cite{lorentz1996constructive}, Theorem 2.7 p.352) : 

$\quad$

\noindent \begin{theorem}
	\label{Schtr_2}
	Let $f\in L^2([0,1])$. Then there is an universal constant $C>0$ such that :
	\begin{equation*}
	E_2(\Lambda_m)\le C\omega(f,\varepsilon\big(\Lambda_m)\big)
	\end{equation*}
\end{theorem}

\begin{lemma}
	\label{modulus_2}
	$w(h,u)\le C_{A}u$, $\forall u\in [0,1/e^2]$. 
\end{lemma}
\begin{proof}
	
	\noindent We write $h(x)$ as the sum of two functions with disjoint support :
	\begin{equation*}
	h=h_{1}+h_{2},
	\end{equation*}
	\noindent with :

	\begin{enumerate}
		\item[$\bullet$] $\displaystyle h_{1}(t)=t^\alpha BL(-\ln(x)-1){\bf 1}_{[\frac{1}{e^2},\frac{1}{e}]}(t)$, \\
		\item[$\bullet$] $\displaystyle h_{2}(t)=t^\alpha BL(1+\ln(t)){\bf 1}_{[\frac{1}{e},1[}(t)$.\\
	\end{enumerate}
	\noindent Thanks to the second part of Proposition \ref{operator_exist_estimate_2}, the function $BL$ is bounded by a constant $C_A$ so, for $i\in[\![1,2]\!]$, each of the function $h_{i}$ is bounded by some constant $C_{A}$ depending on $A$. Moreover, $BL$ is $C^1$ on $[\frac{1}{e^2},\frac{1}{e}]$ and $[\frac{1}{e},1]$, and, for $i\in[\![1,2]\!]$,  $h_{i}'$ is bounded by a constant $C_{A}$. Let $x\in[0,1/e^2]$, $r\in [0,x]$. We have :
	
	\begin{equation*}
	\begin{aligned}
	\int_{0}^{1-r}|h(t+r)-h(t)|^2dt &=\int_{\frac{1}{e^2}}^{\frac{1}{e}-r^2}|h_{1}(t+r)-h_{1}(t)|^2dt + \int_{\frac{1}{e}-r^2}^{\frac{1}{e}}|h_{2}(x+r)-h_{1}(t)|^2dt\\
	&+\int_{\frac{1}{e}}^{1-r}|h_{2}(t+r)-h_{2}(t)|^2dt\\
	&\le \bigg(\frac{1}{e}-\frac{1}{e^2}-r^2\bigg)\|h_{1}'\|_\infty^2 r^2 + r^2\bigg(\|h_{2}\|_\infty+\|h_{1}\|_\infty\bigg)^2\\
	&+\bigg(1-\frac{1}{e}-r\bigg)\|h_{2}'\|_\infty^2 r^2\\
	&\le C_{A}r^2.
	\end{aligned}
	\end{equation*}
	Taking the square root and the supremum on $r$ on each side, the result is proved.		
\end{proof}

\begin{lemma}
	\label{Idefix_2}
	\begin{equation*}
	\varepsilon_2\big(\Lambda_m\big)=O\bigg(\frac{1}{m}\bigg),\qquad m\to+\infty.
	\end{equation*}
\end{lemma}

\begin{proof}

$\quad$

$\quad$

\noindent Using Theorem \ref{simpler_prod_2} and Lemma \ref{ineq_2}, the expression of $\varepsilon_2(\Lambda_m)$ defined above can be written as
\begin{equation*}
\varepsilon_2(\Lambda_m)=\prod_{k=0}^{m}\frac{\lambda_k-\frac{1}{2}}{\lambda_k+\frac{3}{2}}.
\end{equation*}

\noindent

$\quad$

\noindent Recall there exists $C>0$ such that for all $m\in\mathbb{N}$, $\lambda_m\le 4m+C$. Consequently, one has :

\begin{equation*}
\begin{aligned}
\forall m\in\mathbb{N},\quad \ln\bigg(\prod_{k=0}^{m}\frac{\lambda_{k}-\frac{1}{2}}{\lambda_k+\frac{3}{2}}\bigg)&=\ln\bigg(\prod_{k=0}^{m}\bigg(1-\frac{2}{\lambda_k+\frac{3}{2}}\bigg)\bigg)\\
&=\sum_{k=0}^{m}\ln\bigg(1-\frac{2}{\lambda_k+\frac{3}{2}}\bigg)\bigg)\\
&\le -2\sum_{k=0}^{m} \frac{1}{\lambda_k+\frac{3}{2}}\\
&\le -2\sum_{k=0}^{m} \frac{1}{2k+C+\frac{3}{2}}.
\end{aligned}
\end{equation*}
\noindent But $\displaystyle -2\sum_{k=0}^{m} \frac{1}{2k+C+\frac{3}{2}}\underset{m\to+\infty}{=}-\ln(m)+O(1)$. Hence
\begin{equation*}
\varepsilon_2\big(\Lambda_m\big)=O\bigg(\frac{1}{m}\bigg).
\end{equation*}

\end{proof}

\noindent Hence, as $\displaystyle \varepsilon_2(\Lambda_{m(\varepsilon)})\in [0,1/e^2]$ for $\varepsilon$ small enough, we get thanks to Lemma \ref{modulus_2} and Theorem \ref{Schtr_2}:
\begin{center}
	$\displaystyle E(h,\Lambda_{m(\varepsilon)})_2\le C_{A}\varepsilon_2(\Lambda_{m(\varepsilon)})$.
\end{center} To sum up, we have shown that :
\begin{equation*}
\|h\|_2^2\le C_{A}\bigg(\varepsilon+\varepsilon_2(\Lambda_{m(\varepsilon)})^2\bigg).
\end{equation*}

\noindent Now, we know that $\displaystyle \varepsilon_2(\Lambda_{m(\varepsilon)})^2\le \frac{C_A}{m(\varepsilon)^2}$. By virtue of the double inequality 
\begin{center}
	$\displaystyle \frac{1}{\sqrt{\varepsilon}}+o(1)\le g\big(m(\varepsilon)\big)\le \frac{1}{\sqrt{\varepsilon}}$
\end{center}
one has	\begin{center}
	$\displaystyle \frac{1}{2}\ln\bigg(\frac{1}{\varepsilon}\bigg)\underset{\varepsilon\to 0}{\sim}
	 \ln\bigg(g\big(m(\varepsilon)\big)\bigg)\underset{\varepsilon\to 0}{\sim} C_Am(\varepsilon)$.
\end{center}
Hence (for another $C_A>0$) : $\displaystyle \frac{1}{m(\varepsilon)}\le \frac{C_A}{\ln\big(\frac{1}{\varepsilon}\big)}$. Consequently :
\begin{equation*}
\|h\|_2^2\le C_{A}\frac{1}{\ln\big(\frac{1}{\varepsilon}\big)^2}.
\end{equation*}

$\quad$

\noindent Since $h_{1}$ and $h_{2}$ have disjoint support, we have 
\begin{equation*}
\|h\|_2^2=\|h_{1}\|_2^2+\|h_{2}\|_2^2.
\end{equation*}

\noindent In particular

\begin{equation*}
\|h_{2}\|_2^2\le \|h\|_2^2
\end{equation*}

\noindent But as 
\begin{equation*}
\|h_{2}\|_2^2 = \int_{\frac{1}{e}}^{1}t^{2\alpha}\bigg|BL\big(1+\ln(t)\big)\bigg|^2dt
\end{equation*}
\noindent we get
\begin{equation*}
	\int_{\frac{1}{e}}^{1}t^{2\alpha+1}\bigg| BL\big(1+\ln(t)\big)\bigg|^2\frac{dt}{t}\le \frac{C_{A,a}}{\ln\big(\frac{1}{\varepsilon}\big)^2}
\end{equation*}

\noindent Hence, as we integrate over $\displaystyle\bigg[\frac{1}{e^{1}},1\bigg]$, the term $t^{2\alpha+1}$ is minorated by $(1/e)^{(2\alpha+1)}$. By returning to the $\tau$ coordinate, we obtain :
\begin{equation*}
\|BL(1-\tau)\|_{L^2([0,1])}\le C_{A}\frac{1}{\ln\big(\frac{1}{\varepsilon}\big)},
\end{equation*}
\noindent and then
\begin{equation*}
\|BL\|_{L^2([0,1])}\le C_{A}\frac{1}{\ln\big(\frac{1}{\varepsilon}\big)}.
\end{equation*}

$\quad$

\subsubsection{Invertibility of the $B$ operator}

$\quad$

\noindent Now, we want to prove that $B:L^2(0,1)\to L^2(0,1)$ is invertible and that its inverse is bounded with respect to $C_A$. We can write :
\begin{equation*}
B=I+C
\end{equation*}

\noindent where $\displaystyle Ch(\tau)=\int_{\tau}^{1}H_1(x,\tau)h(x)dx$, with :

\begin{equation*}
\begin{aligned}
H_1(x,\tau)&=	\tilde{P}(x,2\tau-x)+H(x,2\tau-x)+\int_{2\tau-x}^{2\tau}\tilde{P}(x,2\tau -t)H(x,t){\bf 1}_{[0,x]}(t)dt\\
	&+\int_{2\tau}^{x}\tilde{P}(x,t)H(x,t-2\tau)\:dt{\bf 1}_{[2\tau,1]}(x)+\int_{2\tau}^{x}\tilde{P}(x,t-2\tau)H(x,t)\:dt{\bf 1}_{[2\tau,1]}(x)\\
	&+P(x,2\tau-x)+\tilde{H}(x,2\tau-x)+\int_{2\tau-x}^{2\tau}P(x,2\tau -t)\tilde{H}(x,t){\bf 1}_{[0,x]}(t)dt\\
	&+\int_{2\tau}^{x}P(x,t)\tilde{H}(x,t-2\tau)\:dt{\bf 1}_{[2\tau,1]}(x)+\int_{2\tau}^{x}P(x,t-2\tau)\tilde{H}(x,t)\:dt{\bf 1}_{[2\tau,1]}(x).
	\end{aligned}
\end{equation*}

$\quad$

\begin{lemma}
	There is a constant $C_A>0$ such that, for all $h$ in $L^2(0,1)$ :
	\begin{center}
		$\displaystyle \forall n\in\mathbb{N^*},\:\:\forall \tau\in [0,1],\quad |C^nh(\tau)|\le C_A \frac{\big((1-\tau)\|H_1\|_{L^\infty}\big)^{n-1}}{(n-1)!}\|h\|_{L^2([0,1])}$
	\end{center}
\end{lemma}

\begin{proof}	
	
	$\quad$
	
	$\quad$
	
	\noindent By induction :
	
	$\quad$
	
	\noindent $\bullet$ From the estimates of Proposition \ref{noyau_estim_2}, $H$, $\tilde{H}$ and $H_1$ are bounded by a constant $C_A$. Using the triangle inequality and the Cauchy-Schwarz inequality, one immediately gets :
	\begin{equation*}
	\begin{aligned}
	|Ch(\tau)|&\le C_A\int_{\tau}^{1}|h(x)|dx\le C_A(1-\tau)\|h\|_{L^2([0,1])}\le C_A\|h\|_{L^2([0,1])}
	\end{aligned}
	\end{equation*}
	
	\medskip
	
	\noindent $\bullet$ Assume it is true for some $n\in\mathbb{N}^*$. Then :
	\begin{equation*}
	\begin{aligned}
	|C^{n+1}h(\tau)|=\bigg|\int_{\tau}^{1}H_1(x,t)C^nh(x)dx\bigg|    &\le\int_{\tau}^{1}\|H_1\|_\infty C_A\frac{(1-x)^{n-1}\|H_1\|_\infty^{n-1}}{(n-1)!}\|h\|_{L^2(0,1)}dx\\
	&=C_A\frac{\|H_1\|_\infty^n}{(n-1)!}\|h\|_{L^2(0,1)}\int_{\tau}^{1}(1-x)^{n-1}dx\\
	&=C_A \frac{\big((1-\tau)\|H_1\|_{\infty}\big)^{n}}{n!}\|h\|_{L^2([0,1])}
	\end{aligned}
	\end{equation*}
	
\end{proof}

\noindent Thus $\:\:\displaystyle \|C^n\|\le C_A \frac{\big((1-\tau)\|H_1\|_{\infty}\big)^{n-1}}{(n-1)!}\:\:$ for all $n\in\mathbb{N^*}$. It follows that the serie $\sum(-1)^nC^n$ is convergent. Consequently $B$ is invertible, $\displaystyle B^{-1}=\sum_{n=0}^{+\infty}(-1)^nC^n$ and :
\begin{equation*}
\|B^{-1}\|\le C_A.
\end{equation*}

\noindent Hence :
\begin{equation*}
\|q-\tilde{q}\|_{L^2(0,1)}=\|L\|_{L^2(0,1)}\le \|B^{-1}\|\|BL\|_{L^2(0,1)} \le C_{A} \frac{1}{\ln\big(\frac{1}{\varepsilon}\big)}
\end{equation*}

\noindent and the proof of Theorem \ref{stabresult_2} is complete.

$\quad$

\noindent Let us prove Corollary \ref{cor_difpot_2}.

\begin{proof}
	Let $s_1,s_2\ge 0$ and $\theta\in (0,1)$. Using the Gagliardo-Nirenberg inequalities (see \cite{brezis2018gagliardo}), one can write 
	\begin{equation*}
	\|g\|_{H^s(0,1)}\le \|g\|_{H^{s_1}(0,1)}^\theta\|g\|_{H^{s_2}(0,1)}^{1-\theta}
	\end{equation*}
	for every $g\in H^{s_1}(0,1)\cap H^{s_2}(0,1)$ and $s=\theta s_1+(1-\theta)s_2$. As $f$ and $\tilde{f}$ belong to $C(A)$ then $q-\tilde{q}$ belong to $H^2(0,1)$ and $\|q-\tilde{q}\|_{H^2(0,1)}\le C_A$. Hence, for $s_1=0$ and $s_2=2$, we have:
	\begin{equation*}
	\begin{aligned}
	\|q-\tilde{q}\|_{H^s(0,1)}&\le \|q-\tilde{q}\|_{L^2(0,1)}^\theta\|q-\tilde{q}\|_{H^{2}(0,1)}^{1-\theta}\\
	&\le C_A^{1-\theta}\|q-\tilde{q}\|_{L^2(0,1)}^\theta\\
	&\le C_A\frac{1}{\ln\big(\frac{1}{\varepsilon}\big)^\theta}
	\end{aligned}
	\end{equation*}
	
	\noindent with $\displaystyle \theta =\frac{2-s}{2}$. Using the Sobolev embedding $H^1(0,1)\hookrightarrow C^0(0,1)$ with $\|.\|_{\infty}\le 2 \|.\|_{H^1(0,1)}$, one gets (for $s=1$ and $\theta =1/2$):
	\begin{equation*}
	\begin{aligned}
	\|q-\tilde{q}\|_{\infty}\le 2\|q-\tilde{q}\|_{H^1(0,1)}\le  C_A \sqrt{\frac{1}{\ln\big(\frac{1}{\varepsilon}\big)}}
	\end{aligned}
	\end{equation*}
\end{proof}

\subsubsection{Uniform estimate of the conformal factors}

\medskip

\noindent Now we give the proof of Corollary \ref{cor_steklov_stab_2}. Assume that $n\ge3$, $\omega=0$ and let us set $F=f^{n-2}$. We can write $\displaystyle q=\frac{F''}{F}$ and then
\begin{equation*}
(\tilde{F}F'-\tilde{F}'F)'(t)=\tilde{F}F(q-\tilde{q})(t).
\end{equation*}

\noindent For all $t\in[0,1]$, we have : 
\begin{equation*}
\begin{aligned}
\tilde{F}(t)F'(t)-\tilde{F}'(t)F(t)&=(n-2)\tilde{f}^{n-2}f^{n-3}(t)f'(t)-(d-2)\tilde{f}^{n-3}f^{n-2}(t)\tilde{f}'(t)\\
&=(n-2)f^{n-3}(t)\tilde{f}^{n-3}(t)\bigg(\tilde{f}(t)f'(t)-f(t)\tilde{f}'(t)\bigg)
\end{aligned}
\end{equation*}

\noindent Assume that for all $t$ in $[0,1]$, $\tilde{f}(t)f'(t)-f(t)\tilde{f}'(t)\ne 0$, for example $\tilde{f}(t)f'(t)>f(t)\tilde{f}(t)$. Then :
\begin{equation*}
\frac{f'(t)}{f(t)}>\frac{\tilde{f}'(t)}{\tilde{f}(t)}.
\end{equation*} 

\noindent Then, by integrating between $0$ and $1$, one gets :
\begin{equation*}
\ln\big(f(1)\big)-\ln\big(f(0)\big)>\ln\big(\tilde{f}(1)\big)-\ln\big(\tilde{f}(0)\big).
\end{equation*}

\noindent and this is not true as $f(0)=f(1)$ and $\tilde{f}(0)=\tilde{f}(1)$. Consequently, there is $x_0\in[0,1]$ such that 	$\big(\tilde{f}f'-f\tilde{f}'\big)(x_0) =0$. Setting $G(x)=(\tilde{F}F'-\tilde{F}'F)(x)$, we have :
\begin{equation*}
\forall x\in [0,1],\quad G(x)=\int_{x_0}^{x}\tilde{F}F(q-\tilde{q})(t)dt.
\end{equation*}

\noindent From the $L^2$ estimate previously established on $q-\tilde{q}$, one has :
\begin{equation*}
\begin{aligned}
\forall x\in[0,1],\quad|G(x)|&\le\sqrt{|x-x_0|}C_A\|q-\tilde{q}\|_2 \\
&\le C_A \frac{1}{\ln\big(\frac{1}{\varepsilon}\big)}
\end{aligned}
\end{equation*}	

\noindent Hence :
\begin{equation*}
\bigg|\bigg(\frac{F}{\tilde{F}}\bigg)'(x)\bigg|=\bigg|\frac{G(x)}{\tilde{F}(x)^2}\bigg|\le C_A \frac{1}{\ln\big(\frac{1}{\varepsilon}\big)},
\end{equation*}	

\noindent and by integrating betwwen $0$ and $x$ :

\begin{equation*}
\begin{aligned}
\bigg|\frac{F(x)}{\tilde{F}(x)}-1\bigg|=\bigg|\int_{0}^{x}\bigg(\frac{F}{\tilde{F}}\bigg)'(t)dt\bigg|\le \int_{0}^{1}\bigg|\frac{G(t)}{\tilde{F}(t)^2}\bigg|dt \le C_A \frac{1}{\ln\big(\frac{1}{\varepsilon}\big)}.
\end{aligned}
\end{equation*}

\noindent and this last inequality leads to the estimate :
\begin{equation*}
\forall x\in[0,1],\quad	|f^{n-2}(x)-\tilde{f}^{n-2}(x)|\le C_A \frac{1}{\ln\big(\frac{1}{\varepsilon}\big)}.
\end{equation*}

\noindent Setting $k=n-2$, thanks to the relation $\displaystyle a^k-b^k=(a-b)\sum_{j=0}^{k}a^jb^{k-j}$, we get at last :
\begin{equation*}
\forall x\in[0,1],\quad|f(x)-\tilde{f}(x)|\le C_A \frac{1}{\ln\big(\frac{1}{\varepsilon}\big)}.
\end{equation*}

$\quad$

\section{About the Calder\'on problem}

Now, we prove Theorem \ref{Calderon_stab_2}. For $s\in\mathbb{R}$, $H^{s}(\partial M)$ can be defined as
\begin{equation*}
H^{s}(\partial M)=\bigg\{ \psi\in\mathcal{D}'(\partial M),\:\: \psi=\sum_{m\ge 0}\begin{pmatrix}
\psi_m^1\\ \psi_m^2
\end{pmatrix}\otimes Y_m,  \quad \sum_{m\ge 0}(1+\mu_m)^s\bigg(|\psi_m^1|^2+|\psi_m^2|^2\bigg) <\infty\bigg\}.
\end{equation*}

\noindent  Recall that we have denoted $\mathcal{B}(H^{1/2}(\partial M))$ the set of bounded operators from $H^{1/2}(\partial M)$ to $H^{1/2}(\partial M)$ and equipped $\mathcal{B}(H^{1/2}(\partial M))$ with the norm
\begin{equation*}
\|F\|_*=\sup_{\psi\in H^{1/2}(\partial M)\backslash\{0\}}\frac{\|F\psi\|_{H^{1/2}}}{\|\psi\|_{H^{1/2}}}.
\end{equation*}

\begin{lemma}
	\label{bounded_equiv_22}
	We have the equivalence :
	\begin{equation*}
\Lambda_g(\omega)-\Lambda_{\tilde{g}}(\omega)\in \mathcal{B}(H^{1/2}(\partial M))\Leftrightarrow \left\{\begin{aligned}
	&f(0)=\tilde{f}(0)\\
	&f(1)=\tilde{f}(1).
	\end{aligned}\right.
	\end{equation*}
\end{lemma}

	\noindent 
	
	\begin{proof}
	 Let us set 
	\begin{center}
	$\displaystyle C_0=\frac{1}{4\sqrt{f(0)}}\frac{h'(0)}{h(0)},\:$ $\quad\displaystyle C_1=\frac{1}{4\sqrt{f(1)}}\frac{h'(1)}{h(1)},\:$ $\quad\displaystyle A_0=\frac{1}{f(0)}-\frac{1}{\tilde{f}(0)}\quad$ and $\quad\displaystyle A_1=\frac{1}{f(1)}-\frac{1}{\tilde{f}(1)}$.
	\end{center}
\medskip 
\noindent For $m\ge 0$, one has, using the block diagonal representation of $\Lambda_g(\omega)$ and the asymptotics of $M(\mu_m)$ and $N(\mu_m)$ given in Theorem  \ref{Simon_2} and Corollary \ref{CorSimon_2}:

\begin{equation*}
\begin{aligned}
\Lambda_g^m(\omega)-\Lambda_{\tilde{g}}^m(\omega)&=\begin{pmatrix}
\frac{\tilde{M}(\mu_m)}{\sqrt{\tilde{f}(0)}}-\frac{M(\mu_m)}{\sqrt{f(0)}}+C_0-\tilde{C}_0&O\big(e^{-2\mu_m}\big)\\
O\big(e^{-2\mu_m}\big)&\frac{\tilde{N}(\mu_m)}{\sqrt{\tilde{f}(1)}}-\frac{N(\mu_m)}{\sqrt{f(1)}}+\tilde{C}_1-C_1
\end{pmatrix}\\
&=\begin{pmatrix}
A_0\sqrt{\mu_m}+(C_0-\tilde{C}_0)&0\\0&A_1\sqrt{\mu_m}+(\tilde{C}_1-C_1)
\end{pmatrix}+ \begin{pmatrix}
O\bigg(\frac{1}{\sqrt{\mu_m}}\bigg)&O\big(e^{-2\mu_m}\big)\\O\big(e^{-2\mu_m}\big)&O\bigg(\frac{1}{\sqrt{\mu_m}}\bigg)
\end{pmatrix}
\end{aligned}
\end{equation*}
Hence, for any $\big(\psi_m^1,\psi_m^2\big)\in\mathbb{R}^2$:
\begin{equation*}
	\big(\Lambda_g^m(\omega)-\Lambda_{\tilde{g}}^m(\omega)\big)\begin{pmatrix}
	\psi_m^1\\ \psi_m^2
	\end{pmatrix}=
	\sqrt{\mu_m}\begin{pmatrix}
	A_0\psi^1_m\\
	A_1\psi^2_m
	\end{pmatrix}+\begin{pmatrix}
	(C_0-\tilde{C}_0)\psi^1_m\\
	(\tilde{C}_1-C_1)\psi^2_m
	\end{pmatrix}+O\bigg(\frac{\psi^1_m+\psi^2_m}{\sqrt{\mu_m}}\bigg)
\end{equation*}

\noindent For $\displaystyle \psi=\sum_{m\ge 0}\begin{pmatrix}
\psi_m^1\\ \psi_m^2
\end{pmatrix}\otimes Y_m\in H^{1/2}(\partial M)$, one has
\begin{equation*}
\begin{aligned}
\|\big(\Lambda_g(\omega)-\Lambda_{\tilde{g}}(\omega)\big)\psi\|^2_{H^{1/2}(\partial M)}&=\sum_{m\ge 0}(1+\mu_m)^{1/2}\mu_m\bigg(A_0^2|\psi_m^1|^2+A_1^2|\psi_m^2|^2\bigg)\\
&+\sum_{m\ge 0}2(1+\mu_m)^{1/2}\sqrt{\mu_m}\bigg(|A_0(C_0-\tilde{C}_0)||\psi_m^1|^2+|A_1(\tilde{C}_1-C_1)||\psi_m^2|^2\bigg)\hspace{2cm}\\
&+\sum_{m\ge 0}(1+\mu_m)^{1/2}O\big(|\psi_m^1|^2 + |\psi_m^2|^2 \big)
\end{aligned}
\end{equation*}
\noindent Then
\begin{equation*}
\|\Lambda_g(\omega)-\Lambda_{\tilde{g}}(\omega)\|_*<\infty \Leftrightarrow \left\{\begin{aligned}
&A_0=0\\
&A_1=0
\end{aligned}\right.\Leftrightarrow \left\{\begin{aligned}
&\tilde{f}(0)=\tilde{f}(0)\\
&\tilde{f}(1)=\tilde{f}(1).
\end{aligned}\right.
\end{equation*}
\end{proof}

\medskip

\noindent Under the assumptions of Theorem \ref{Calderon_stab_2}, the following estimate holds:

\medskip

\begin{proposition}
	\label{WT_estimates_22}
	Let $\varepsilon>0$. Assume that $\|\Lambda_g(\omega)-\Lambda_{\tilde{g}}(\omega)\|_*\le \varepsilon$. There is $C_A>0$ such that :
	
	\begin{equation*}
	\forall m\in\mathbb{N},\quad\bigg|N(\kappa_m)-\tilde{N}(\kappa_m)\bigg|\le C_A\varepsilon.
	\end{equation*}
\end{proposition}

\begin{proof} For $m\in\mathbb{N}$, consider $\displaystyle \psi_m=\begin{pmatrix}
	0\\1
	\end{pmatrix}\otimes Y_m \in H^{1/2}(\partial M)$. 
	
	\noindent One has :
	\begin{equation*}
	\begin{aligned}
	\big(\Lambda_g(\omega)-\Lambda_{\tilde{g}}(\omega)\big)\psi_m &= \big(\Lambda_g^m(\omega)-\Lambda_{\tilde{g}}^m(\omega)\big) \begin{pmatrix}
	0\\1
	\end{pmatrix}\otimes Y_m\\ &= \displaystyle 
	\begin{pmatrix}
	0&\frac{1}{\sqrt{f(0)}}\frac{h^{1/4}(1)}{h^{1/4}(0)}\big(\frac{1}{\tilde{\Delta}(\mu_m)}-\frac{1}{\Delta(\mu_m)}\big)\\0&\bigg(\frac{\tilde{N}(\mu_m)}{\sqrt{f(1)}}-\frac{N(\mu_m)}{\sqrt{f(1)}}\bigg)+(\tilde{C}_1-C_1)
	\end{pmatrix}\otimes Y_m
	\end{aligned}
	\end{equation*}
	
	\noindent Then
	\begin{equation*}
	\begin{aligned}
	\|\big(\Lambda_g(\omega)-\Lambda_{\tilde{g}}(\omega)\big)\psi_m\|^2_{H^{1/2}(\partial M)} = (\mu_m+1)^{1/2}\bigg[\bigg(\frac{\tilde{N}(\mu_m)}{\sqrt{f(1)}}-&\frac{N(\mu_m)}{\sqrt{f(1)}}+(\tilde{C}_1-C_1)\bigg)^2\\&+\frac{1}{f(0)}\frac{h^{1/2}(1)}{h^{1/2}(0)}\bigg(\frac{1}{\tilde{\Delta}(\mu_m)}-\frac{1}{\Delta(\mu_m)}\bigg)^2\bigg].
	\end{aligned}
	\end{equation*}
	\noindent so, for all $m\ge 0$:
	
	\begin{equation*}
		\begin{aligned}
		(\mu_m+1)^{1/2}\bigg|\frac{1}{\sqrt{f(1)}}\big(\tilde{N}(\mu_m)-N(\mu_m)\big)+(\tilde{C}_1-C_1)\bigg|^2 &\le \|\big(\Lambda_g(\omega)-\Lambda_{\tilde{g}}(\omega\big)\psi_m\|^2_{H^{1/2}(\partial M)}\\
	&\le \|\Lambda_g(\omega)-\Lambda_{\tilde{g}}(\omega)\|_*^2 \|\psi_m\|^2_{H^{1/2}(\partial M)}\\
	&= \|\Lambda_g(\omega)-\Lambda_{\tilde{g}}(\omega)\|^2_*(\mu_m+1)^{1/2} \\
	&\le \varepsilon^2(\mu_m+1)^{1/2}.
			\end{aligned}	
	\end{equation*}
	\noindent Hence
	\begin{equation}
	\label{estimation_cald_2}
	\bigg|\frac{1}{\sqrt{f(1)}}\big(\tilde{N}(\mu_m)-N(\mu_m)\big)+(\tilde{C}_1-C_1)\bigg|\le \varepsilon.
	\end{equation}
	
	\medskip
	
	\noindent Using the asymptotic $N(\mu_m)=-\mu_m+o(1)$, we deduce from (\ref{estimation_cald_2}) that
	\begin{equation*}
		|\tilde{C}_1-C_1|\le \varepsilon
	\end{equation*}
	
	\medskip
	
	\noindent and then that there is $C_A>0$ such that, for all $m\in\mathbb{N}$ :
	\begin{equation*}
		\bigg|N(\mu_m)-\tilde{N}(\mu_m)\bigg|\le C_A\:\varepsilon.
	\end{equation*}
\end{proof}

\noindent As in Lemma \ref{integral_relation_2}, one gets an integral relation between $N(z)-\tilde{N}(z)$ and $q-\tilde{q}$:

\begin{lemma}
	 The following integral relation holds:
	\begin{equation}
	\big(N(z)-\tilde{N}(z)\big)\Delta(z)\tilde{\Delta}(z)=\int_{0}^{1}\big(q(x)-\tilde{q}(x)\big)s_0(x,z)\tilde{s}_0(x,z)dx
	\end{equation}
\end{lemma}

\begin{proof}
\noindent Let us define $\displaystyle \theta :x\mapsto s_0(x,z)\tilde{s_0}'(x,z)-s_0'(x,z)\tilde{s}_0(x,z)$. Then :
\begin{equation*}
\begin{aligned}
\theta'(x)&=\big(\tilde{q}(x)-q(x)\big)s_0(x,z)\tilde{s}_0(x,z)
\end{aligned}
\end{equation*}

\medskip

\noindent By integrating between $0$ and $1$, one gets:
\begin{equation*}
s_0'(1,z)\tilde{s}_0(1,z)-s_0(1,z)\tilde{s}_0'(1,z)=\int_{0}^{1}\big(q(x)-\tilde{q}(x)\big)s_0(x,z)\tilde{s}_0(x,z)dx
\end{equation*}

\noindent As $s_0'(1,z)=N(z)\Delta(z)$ and $s_0(1,z)=\Delta(z)$, one gets for all $z\in\mathbb{C}\backslash\mathcal{P}$ :
\begin{equation*}
\big(N(z)-\tilde{N}(z)\big)\Delta(z)\tilde{\Delta}(z)=\int_{0}^{1}\big(q(x)-\tilde{q}(x)\big)s_0(x,z)\tilde{s}_0(x,z)dx.
\end{equation*}
\end{proof}

\noindent  Just as in Section $4$, let us extend on $[-1,0]$ $q$ and $\tilde{q}$ into even functions and denote $L(x)=q(x)-\tilde{q}(x)$. We recall that for all $m\in\mathbb{N}$, we have set $y_m=\sqrt{\kappa_m}$.

$\quad$

\noindent We will take advantage of this representation to write in another way the equalities
\begin{equation*}
\begin{aligned}
\big(N(\kappa_m)-\tilde{N}(\kappa_m)\big)\Delta(\kappa_m)\tilde{\Delta}(\kappa_m)&=\int_{0}^{1}(q(x)-\tilde{q}(x)s_0(x,\kappa_m)\tilde{s}_0(x,\kappa_m)dx.\\
\end{aligned}
\end{equation*}

\medskip

\begin{proposition}
	\label{operator_exist_estimate_22}
	There is an operator $D:L^2([0,1]) \to L^2([0,1])$ such that :
	\begin{enumerate}
		\item For all $m\in\mathbb{N}$, \begin{center}
			$\displaystyle 
			\big(N(\kappa_m)-\tilde{N}(\kappa_m)\big)s_0(1,\kappa_m)\tilde{s}_0(1,\kappa_m)=\frac{1}{y_m^2}\int_{0}^{1}\cosh(2\tau y_m)DL(\tau)d\tau-\frac{1}{y_m^2}\int_{0}^{1}L(\tau)d\tau.$
		\end{center}
		\item The function $\tau\mapsto DL(\tau)$ is $C^1$ on $[0,1]$ and $DL$ and $(DL)'$ are uniformly bounded by a constant $C_A$.
	\end{enumerate}
\end{proposition}

\begin{proof}
	Using the same calculations as in Proposition \ref{operator_exist_estimate_2} together with the representation formula for $s_0$
	\begin{equation*}
	s_0(x,\kappa_m)=\frac{\sinh(y_m x)}{y_m}+\int_{0}^{x}H(x,t)\frac{\sinh(y_mt)}{y_m}dt
	\end{equation*}
	one can prove that the operator $D$ is given by
	\begin{equation*}
	\begin{aligned}
	DL(\tau)=L(\tau)+\int_{\tau}^{1}\tilde{H}(x,2\tau-x)L(x)dx &+\int_{\tau}^{1}H(x,2\tau-x)L(x)dx\\
	&+\int_{\tau}^{1}L(x)\int_{2\tau-x}^{2\tau}\tilde{H}(x,2\tau -t)H(x,t){\bf 1}_{[0,x]}(t)dtdx\\
	&+\int_{2\tau}^{1}L(x)\int_{2\tau}^{x}\tilde{H}(x,t)H(x,t-2\tau)\:dtdx\\
	&+\int_{2\tau}^{1}L(x)\int_{2\tau}^{x}\tilde{H}(x,t-2\tau)H(x,t)\:dtdx\\
	\end{aligned}
	\end{equation*}
	
	\noindent and so that $DL$ and its derivative are bounded by some constant $C_A>0$.
\end{proof}

\noindent For all $m\in\mathbb{N}$, one has
\begin{equation*}
\begin{aligned}
(N(\kappa_m)-\tilde{N}(\kappa_m))s_0(1,\kappa_m)\tilde{s}_0(1,\kappa_m)&=\frac{1}{y_m^2}\int_{0}^{1}\cosh(2\tau y_m)DL(\tau)d\tau-\frac{1}{y_m^2}\int_{0}^{1}L(\tau)d\tau\\
&=\frac{1}{2y_m^2}\int_{0}^{1}e^{2\tau y_m}DL(\tau)d\tau+\frac{1}{2y_m^2}\int_{0}^{1}e^{-2\tau y_m}DL(\tau)d\tau\\
&\qquad\qquad\qquad\qquad\qquad\qquad-\frac{1}{y_m^2}\int_{0}^{1}L(\tau)d\tau.
\end{aligned}
\end{equation*}

\noindent Hence, by multiplying both sides by $2y_m^2e^{-2y_m}$, one has :
\begin{equation*}
\begin{aligned}
2y_m^2e^{-2y_m}\big(N(\kappa_m)-\tilde{N}(\kappa_m)\big)s_0(1,\kappa_m)\tilde{s}_0(1,\kappa_m)&=\int_{0}^{1}e^{2 y_m(\tau -1)}DL(\tau)d\tau+\int_{0}^{1}e^{-2y_m(\tau+1)}DL(\tau)d\tau\\&\quad\quad-2e^{-2y_m}\int_{0}^{1}L(\tau)d\tau
\end{aligned}
\end{equation*}

\noindent

\noindent The asymptotic
\begin{equation*}
s_0(1,\kappa_m)\sim \frac{e^{y_m}}{y_m},\qquad m\to+\infty,
\end{equation*}

\noindent ensures that $y_m^2e^{-2y_m}s_0(1,\kappa_m)\tilde{s}_0(1,\kappa_m)$ is bounded uniformly in $m$. Moreover, by hypothesis:
\begin{equation*}
	\bigg|\int_{0}^{1}L(\tau)d\tau\bigg|\le \varepsilon
\end{equation*}
so
\begin{equation*}
	\bigg|\int_{0}^{1}e^{2 y_m(\tau -1)}DL(\tau)d\tau+\int_{0}^{1}e^{-2y_m(\tau+1)}DL(\tau)d\tau\bigg|\le C_A\varepsilon.
\end{equation*}

\medskip

\noindent We write :

$\quad$

\noindent $\bullet$ $\displaystyle \int_{0}^{1}e^{2y_m(\tau-1)}DL(\tau)d\tau = \int_{0}^{1}e^{-2\tau y_m}DL(1-\tau)d\tau$,

\medskip

\noindent $\bullet$ $\displaystyle \int_{0}^{1}e^{-2y_m(t+1)}DL(\tau)d\tau=\int_{1}^{2}e^{-2\tau y_m}DL(\tau-1)d\tau$,

$\quad$

\medskip

\noindent Setting
\begin{equation*}
\displaystyle RL(\tau)=DL(1-\tau){\bf 1}_{[0,1]}(\tau)+DL(\tau-1){\bf 1}_{[1,2]}(\tau)
\end{equation*}
\noindent one has for all $m\in\mathbb{N}$

\begin{equation*}
\bigg|\int_{0}^{+\infty}e^{-2\tau y_m}RL(\tau)d\tau\bigg|\le C_A\varepsilon
\end{equation*}

\medskip

\noindent By the change of variable $\tau = -\ln(t)$, we obtain the moment problem : 
\begin{equation*}
\forall m\in\mathbb{N},\quad\bigg|\int_{0}^{1}t^{2 y_m}RL(-\ln(t))dt\bigg|\le C_A\varepsilon
\end{equation*}

\noindent Using the same technique as in section 4.4.1, we prove the stability estimate
\begin{equation*}
\|DL\|_{L^2([0,1])}\le C_{A}\frac{1}{\ln\big(\frac{1}{\varepsilon}\big)}.
\end{equation*}

\noindent But $D$ can be written as (see section 4.4.2)
\begin{equation*}
	D=I+C
\end{equation*}
\noindent with, for all $n\ge 1$, $\displaystyle \|C^n\|\le \frac{\big(C_A(1-\tau)\big)^{n-1}}{(n-1)!}$. Consequently $B$ is invertible and its inverse is bounded by some constant $C_A>0$. Hence 
\begin{equation*}
\begin{aligned}
\|q-\tilde{q}\|_{2}&\le \|D^{-1}\|\|DL\|_{L^2([0,1])}\\
&\le C_{A}\frac{1}{\ln\big(\frac{1}{\varepsilon}\big)}.
\end{aligned}
\end{equation*}

\noindent At last, if $\omega=0$ and $n\ge 3$, we deduce as previously that
\begin{equation*}
\begin{aligned}
\|f-\tilde{f}\|_{\infty}&\le C_{A}\frac{1}{\ln\big(\frac{1}{\varepsilon}\big)}.
\end{aligned}
\end{equation*}

$\quad$

$\quad$

\noindent {\bf Aknowledgements.} The author would like to deeply thank Thierry Daud\'e and Fran\c{c}ois Nicoleau for their encouragements, helpful discussions and careful reading.



\bibliographystyle{abbrv}
\bibliography{bibliographie}

  \end{document}